\documentclass[12pt,a4paper,openbib]{article}
\usepackage[utf8]{inputenc}
\usepackage[T1]{fontenc}
\usepackage[french]{babel}
\usepackage{amsmath}
\usepackage{amsfonts}
\usepackage{dsfont}
\usepackage{tikz}
\usetikzlibrary{cd}
\usepackage{amssymb}
\usepackage{amsthm}
\usepackage{lmodern}
\usepackage{fancyhdr}
\usepackage{lettrine}
\usepackage{bbm}
\usepackage{bm}
\usepackage{enumerate}
\usepackage{multicol}
\usepackage[hidelinks]{hyperref}
\usepackage[nottoc, notlof, notlot]{tocbibind}
\usepackage[left=2cm,right=2cm,top=2cm,bottom=2cm]{geometry} 
\usepackage{graphicx}
\usepackage{mathtools}
\usepackage{wasysym}
\usepackage{stmaryrd}
\usepackage[all,cmtip]{xy}
\usepackage{enumitem}
\usepackage{relsize} 
\usepackage{mathrsfs}
\usepackage{esvect}
\usepackage[toc,page]{appendix}

\makeatletter

\newcommand{\address}[1]{\gdef\@address{#1}}
\newcommand{\email}[1]{\gdef\@email{\url{#1}}}
\newcommand{\@endstuff}{\par\vspace{\baselineskip}\noindent\small
\begin{tabular}{@{}l}\scshape\@address\\\textit{E-mail address:} \@email\end{tabular}}
\AtEndDocument{\@endstuff}

\makeatother

\author{Damien Junger\footnote{This work has been written in a great part during the author PhD thesis at the ENS Lyon. His work are currently funded by the Deutsche Forschungsgemeinschaft (DFG, German Research Foundation) under Germany's Excellence Strategy EXC 2044 –390685587, Mathematics Münster: Dynamics–Geometry–Structure.}}
\address{Mathematisches Institut, Universität Münster,\\ Fachbereich Mathematik und Informatik der Universität Münster,  Orléans-Ring 10, 48149 Münster, Germany.}
\email{djunger@uni-muenster.de}
\title{Cohomologie mod $p$ des fibrés en droites équivariants sur le demi-plan de Drinfeld}

\newtheorem{theointro}{Th\'eor\`eme}

\newtheorem{corointro}[theointro]{Corollaire}
\newtheorem{lemintro}[theointro]{Lemme}
\newtheorem{theo}{Th\'eor\`eme}[section]
\newtheorem{lem}[theo]{Lemme}
\newtheorem{coro}[theo]{Corollaire}
\newtheorem{prop}[theo]{Proposition}
\newtheorem{propintro}[theointro]{Proposition}

\theoremstyle{definition}
\newtheorem{defi}[theo]{D\'efinition}
\newtheorem*{defi*}{D\'efinition}
\newtheorem{claim}[theo]{Fait}

\theoremstyle{remark}
\newtheorem{rem}[theo]{Remarque}
\newtheorem*{rem*}{Remarque}
\newtheorem{const}[theo]{Construction}
\newtheorem{ex}[theo]{Exemple}

\DeclareMathOperator{\cind}{c-ind}
\DeclareMathOperator{\ind}{ind}

\DeclareMathOperator{\spf}{Spf}
\DeclareMathOperator{\spg}{Sp}

\DeclareMathOperator{\mat}{M}
\DeclareMathOperator{\gln}{GL}
\DeclareMathOperator{\pgln}{PGL}
\DeclareMathOperator{\sln}{SL}

\DeclareMathOperator{\stab}{Stab}

\DeclareMathOperator{\vectot}{-Vect}

\DeclareMathOperator{\ens}{Ens}

\DeclareMathOperator{\nilp}{Nilp}
\DeclareMathOperator{\hhh}{H}
\DeclareMathOperator{\rrr}{R}
\DeclareMathOperator{\homm}{Hom}
\DeclareMathOperator{\en}{End}
\DeclareMathOperator{\aut}{Aut}

\DeclareMathOperator{\ost}{Ost}

\DeclareMathOperator{\gal}{Gal}
\DeclareMathOperator{\id}{Id}

\DeclareMathOperator{\pic}{Pic}
\DeclareMathOperator{\ord}{ord}

\DeclareMathOperator{\lie}{Lie}

\DeclareMathOperator{\rg}{rg}
\DeclareMathOperator{\Length}{Length}
\DeclareMathOperator{\sym}{Sym}
\DeclareMathOperator{\std}{Std}
\DeclareMathOperator{\stb}{Stb}
\DeclareMathOperator{\imm}{Im}
\DeclareMathOperator{\supp}{Supp}
\DeclareMathOperator{\fro}{Fr}

\DeclareMathOperator{\leg}{Leg}

 \newcommand{\iso}{\stackrel{\sim}{\fl}}

\font\tengoth=eufb10
\font\sevengoth=eufb7
\font\fivegoth=eufb5

\newfam\gothfam
\textfont\gothfam=\tengoth
\scriptfont\gothfam=\sevengoth
\scriptscriptfont\gothfam=\fivegoth

\def\A{{\mathbb{A}}}

\def\C{{\mathbb{C}}}

\def\F{{\mathbb{F}}}

\def\H{{\mathbb{H}}}

\def\N{{\mathbb{N}}}

\def\P{{\mathbb{P}}}
\def\Q{{\mathbb{Q}}}
\def\R{{\mathbb{R}}}

\def\Z{{\mathbb{Z}}}

\def\BC{{\mathcal{B}}}
\def\CC{{\mathcal{C}}}

\def\IC{{\mathcal{I}}}
\def\GC{{\mathcal{G}}}

\def\LC{{\mathcal{L}}}

\def\OC{{\mathcal{O}}}

\def\TC{{\mathcal{T}}}
\def\UC{{\mathcal{U}}}

\def\XG{{\mathfrak{X}}}
\def\YG{{\mathfrak{Y}}}
\def\ZG{{\mathfrak{Z}}}

\def\mG{{\mathfrak{m}}}

\def\pG{{\mathfrak{p}}}

\def\Ff{{\mathscr{F}}}
\def\Hf{{\mathscr{H}}}

\def\Lf{{\mathscr{L}}}

\def\Mf{{\mathscr{M}}}
\def\Of{{\mathscr{O}}}

\def\bar#1{\overline{#1}}

\def\han#1{{\rm H}^{#1}_{\rm an}}

\def\hgal#1{{\rm H}^{#1}_{\rm Gal}}

\def\hcech#1{\check{\rm H}^{#1}}
\def\ccech#1{\check{\CC}^{#1}}

\def\et{\text{ et }}
\def\si{\text{ si }}
\def\sinon{\text{ sinon }}
\def\and{\text{ and }}

\def\fl{\rightarrow}

\def\fln#1#2{\xrightarrow[#2]{#1}}

\def\flinj{\hookrightarrow}
\def\flsur{\twoheadrightarrow}

\def\limp{\varprojlim}

\begin{document}

\maketitle

\begin{abstract}
We give a classification of all equivariant line of bundles on the semi-stable model $\hat{\mathbb{H}}$ of the Drinfeld upper half plane $\mathbb{H}$ on $\mathbb{Q}_p$  for a certain subgroup $[G]_2$ of ${\rm GL}_2(\mathbb{Q}_p)$ of index $2$. Then we study the cohomology groups of these line bundles that we restrict on the special fiber $\bar{\mathbb{H}}$ and this process provides a whole family of mod $p$ representations. In particular, we exhibit a class of $[G]_2$-equivariant line bundles on $\bar{\mathbb{H}}$ (the so-called positives of weight $-1$) and show that they are in one-to-one correspondence with the irreducible supersingular representations of $[G]_2$ (a notion we define) via the  map ${\Lf}\mapsto \hhh^0(\bar{\H},{\Lf})^\vee$.
\end{abstract}

\tableofcontents
 
\section*{Introduction}

Soit  $p$ un premier et $K$ une extension finie de $\Q_p$, le but de ce travail est d'étudier les représentations fournies par la cohomologie cohérente des fibrés en droite $G:=\gln_2(K)$-équivariant sur le demi-plan de Drinfeld $\H$ sur $K$ (en dimension $1$). Il s'agit de l'ouvert admissible de la droite projective rigide  $\P^1_{\breve{K}}$ \footnote{Avec $\breve{K}$ la complétion de l'extension maximale non ramifiée de $K$}  obtenu en retirant tous les points $K$-rationnels de $\P^1_{\breve{K}}$. Cette démarche a d'abord été réalisée par Morita et Murase (\cite{mormur1,mormur2,mormur3})  en dimension $1$ pour les puissances tensorielles\footnote{Par abus, nous écrirons encore $\Of(k)$ pour la restriction du faisceau $G$-équivariant $\Of(k)$ sur $\P^1_{\breve{K}}$} de $\Omega^1_\H\cong \Of(-2)$ et a été prolongée à quelques fibrés vectoriels pour l'espace symétrique de Drinfeld de dimension quelconque par Schneider-Stuhler \cite{scst}. Toujours en dimension quelconque, Orlik \cite{orl2} a pu décrire toutes les représentations provenant de fibrés vectoriels qui sont  restriction d'un fibré vectoriel équivariant sur l'espace projectif ambiant $\P^d_{\breve{K}}$ (voir aussi \cite{orlstrau}). 

Teitelbaum, dans \cite{teit4},  a quant à lui étudié les structures entières des fibrés $\Of(2k)$ (avec $k>0$) sur le modèle entier $\hat{\H}$ de l'espace symétrique $\H$ en dimension $1$. Pour les fibrés de poids impair $\Of(2k+1)$, cela a été étendu par Grosse-Kl\"onne  dans \cite{GK6} et voir même à certains fibrés en droite en dimension quelconque dans \cite{GK7}. Il est important de noter que les structures entières trouvées pour $\Of(2k+1)$ ne sont pas des fibrés en droite et possèdent des pathologies autour des points singuliers du modèle $\hat{\H}$. De plus certains fibrés « qui  proviennent du premier revêtement » (que nous décrirons un peu plus tard) ne sont pas inclus dans cette étude. 

Dans ce travail, nous voulons reprendre ces constructions en donnant une approche plus systématique qui permettra d'étudier le groupe  $\pic_G(\hat{\H})$ des fibrés $G$-équivariants  sur $\hat{\H}$ tout entier  et qui se débarrassera des pathologies aux singularités qui apparaissaient pour les structures entières de $\Of(2k+1)$. Cela nécessitera d'affaiblir les conditions imposées par l'équivariance et de n'étudier que les fibrés en droite $\pic_{[G]_2}(\hat{\H})$ munis d'une action du sous-groupe $[G]_2\subset G$ d'indice $2$ formé des matrices dont la valuation du déterminant est pair $[G]_2=(v_K\circ \det)^{-1}(2\Z)$. Se concentrer sur cette condition plus faible nous autorisera à considérer des fibrés équivariants supplémentaires, certains d'entre eux provenant naturellement de l'interprétation modulaire de $\hat{\H}$. Donnons en quelques exemples et pour cela, rappelons que $\hat{\H}$ répond à un problème modulaire qui classifie des déformations à isogénie près de $ \OC_D$-modules formels « spéciaux » de dimension $2$ et de hauteur $4$ avec $D$ l'algèbre de quaternion d'invariant $1/2$ d'uniformisante $\Pi_D$. Nous appellerons $\XG$ le $\OC_D$-module universel sur $\bar{\H}$.

\begin{itemize}
\item L'algèbre de Lie $\lie(\XG)$ du module formel universel $\XG/\hat{\H}$ est un  fibré vectoriel de rang $2$ sur $\hat{\H}$ qui se décompose en la somme directe de deux  fibrés en droite $\lie(\XG)_0$ et $\lie(\XG)_1$ qui correspondent aux parties isotypiques pour l'action de $\OC_{(2)}\subset D$ avec $\OC_{(2)}$ l'anneau des entiers de $K_{(2)}$ l'extension non-ramifiée de $K$ de degré $2$. Ces deux fibrés sont stables pour l'action de $[G]_2$ mais sont échangés par tout élément de $G$ qui n'est pas dans $[G]_2$. De plus, en fibre générique, ils sont tous deux  isomorphes à $\Of(1)$ fournissant ainsi la structure entière voulue pour les fibrés de poids impairs. Nous appellerons $\omega_0$, $\omega_1$ les fibrés duaux de $\lie(\XG)_0$ et $\lie(\XG)_1$.
\item Les points de $\Pi_D$-torsion $\XG[\Pi_D]$ du module universel $\XG$ définissent ce que l'on appelle un schéma de Raynaud (voir la section \ref{ssecrayn} pour leur définition et leur classification). Cela entraîne par définition que chaque terme de la décomposition en partie isotypique $\IC=\bigoplus_\chi \Lf_\chi$ de l'idéal d'augmentation $\IC$ de $\XG[\Pi_D]$ pour l'action de $\F_{q^2}^*\cong\OC_D^*/1+\Pi_D\OC_D$ est localement libre de rang $1$. De plus,  parmi tous ces objets, nous distinguerons $2$ fibrés $\Lf_0$, $\Lf_1$ qui sont associés aux deux caractères $\chi_0,\chi_1 :\F_{q^2}\to \OC_{\breve{K}}$ additifs\footnote{Où on a imposé $\chi_i(0)=0$. } modulo $\varpi$. Tous les autres fibrés $\Lf_\chi$ sont des produits tensoriels $\Lf_0$ et $\Lf_1$. De par les symétries du problème modulaire, chacune de ces parties isotypiques est stable sous l'action de $[G]_2$ et  l'action d'un élément $g$ qui n'est pas dans $[G]_2$  échange $\Lf_0$ et $ \Lf_1$.
\item Nous avons aussi l'exemple les formes différentielles à pôle logarithmique $\Omega^1(\log)$ qui est cette fois-ci $G$-équivariant et a été étudié dans \cite{teit4,GK6}.
\end{itemize}

Notons aussi que pour chaque caractère $\chi :[G]_2\to \breve{K}^*$, on peut définir un fibré $\Of(\chi)$ qui est le faisceau structural $\Of$ de $\hat{\H}$ auquel on a tordu l'action naturelle par $\chi$.

Inspiré par les définitions ci-dessus, nous noterons $\pic_{(mod)}(\hat{\H})$ l'ensemble des fibrés  modulaires qui est le sous-groupe de $\pic_{[G]_2}(\hat{\H})$ engendré par $\Of(\chi)$, $\omega_0$, $\omega_1$, $\Lf_0$, $\Lf_1$. Le premier résultat de classification est le suivant :
\begin{theointro}\label{theointroclass}

Les fibrés modulaires sont les fibrés $[G]_2$-équivariants i.e. \[\pic_{[G]_2}(\hat{\H})=\pic_{(mod)}(\hat{\H})\] et $\pic_G(\hat{\H})$ est engendré par les caractères et $\Omega^1(\log)$.

En particulier, le fibré $\Of(1)$ sur $\H$ n'admet pas de structure entière $G$-équivariante.
\end{theointro}

Expliquons la stratégie de la preuve. Pour cela, nous avons besoin d'une description plus concrète de $\pic_{[G]_2}(\hat{\H})$. Nous avons une flèche $\beta:\pic_{[G]_2}(\hat{\H})\to\pic(\hat{\H})^{[G]_2}$ à valeurs dans les fibrés invariants qui correspond à l'oubli de l'action de $[G]_2$. Elle est rarement injective ou surjective pour un espace quelconque  mais elle s'inscrit dans une suite exacte
\[0\to \hgal{1}({[G]_2}, \Of^*(\hat{\H})) \to \pic_{[G]_2}(\hat{\H}) \to \pic(\hat{\H})^{[G]_2} \to \hgal{2}({[G]_2}, \Of^*(\hat{\H}))  \]
où le noyau est l'ensemble des caractères $\hgal{1}({[G]_2}, \Of^*(\hat{\H}))\cong \homm({[G]_2}, \OC_{\breve{K}}^*)$ (les fonctions sur $\hat{\H}$ sont constantes). Déterminer  l'image de la flèche $\beta$ est le point le plus technique de la classification. La première étape vise à comprendre $\pic(\hat{\H})$ puis déterminer les ${[G]_2}$-invariants. Pour énoncer le résultat dans ce sens, nous avons besoin de comprendre la géométrie de la fibre spéciale $\bar{\H}$ du modèle entier $\hat{\H}$. Toutes les composantes irréductibles de $\bar{\H}$ sont isomorphes à  des droites projectives  $\P_s\cong\P^1$ indexées par les sommets $s\in\BC\TC_0$ de l'arbre de Bruhat-Tits $\BC\TC$ de $\pgln_2(K)$. Ces droites projectives se rencontrent transversalement le long de leurs points ${\F}$-rationnels et ces intersections  sont en bijections avec les arêtes $\BC\TC_1$ de l'arbre $\BC\TC$. Le résultat est le suivant :

\begin{propintro}

La flèche naturelle de  restriction à chaque composante irréductible de la fibre spéciale est un isomorphisme : \[\pic(\hat{\H})\iso\prod_{s\in \BC\TC_0}\pic(\P_s)\cong \prod_{s\in\BC\TC_0}\Z\]

\end{propintro}

Comme l'action de ${[G]_2}$ a deux orbites $[G]_2\cdot s_0$, $[G]_2\cdot s_1$ sur $\BC\TC_0$, chaque fibré en droite  $\Lf\in\pic_{[G]_2}(\hat{\H})$ est déterminé à un caractère près par sa restriction aux $2$ composantes irréductibles $\P_{s_0}$, $\P_{s_1}$ et donc par deux entiers $\ord_{s_0}(\Lf)$, $\ord_{s_1}(\Lf)$ que l'on appellera les ordres en $s_0$ et $s_1$. Dit autrement, \[\pic(\hat{\H})^{[G]_2}\cong\Z^2.\] Établir le théorème \ref{theointroclass} revient à déterminer les couples possibles $(\ord_{s_0}(\Lf),\ord_{s_1}(\Lf))$ quand $\Lf$ varie. Pour les fibrés modulaires, cela peut se réaliser explicitement :



\begin{lemintro}

On a
\begin{enumerate}
\item $\ord_{s_0}(\omega_0)=- 1 $, $\ord_{s_1}(\omega_0)=  q$, $\ord_{s_0}(\omega_1)= q$  et $\ord_{s_1}(\omega_1)= -1$.
\item $\ord_{s_0}(\Lf_0)= 1 $ et $\ord_{s_1}(\Lf_0)=-1 $, $\ord_{s_0}(\Lf_1)= -1$ et $ \ord_{s_1}(\Lf_1)=1 $.
\end{enumerate}

En particulier, la restriction de la flèche d'oubli $\beta:\pic_{(mod)}(\hat{\H})\to\Z^2$ a pour image \[\{(k_0,k_1)\in \Z^2 : k_0+k_1\equiv 0\pmod{q-1}\}\]

\end{lemintro}

La stratégie employée pour ce résultat intermédiaire est la suivante. L'interprétation modulaire de $\hat{\H}$ fournit deux   flèches \[\Pi, F : \omega_i|_{\P_{s_j}}\to  \omega_{i+1}|_{\P_{s_j}}.\] Le point technique consiste à calculer les points et les ordres d'annulation de ces flèches. Ainsi chaque morphisme, lorsqu'il est  non-nul, donne lieu à une équation entre les entiers recherchés  et le résultat s'obtient en résolvant le système obtenu. Pour $\Lf_i$, cela a déjà été fait dans \cite{pan} quand $K=\Q_p$ et étendu à un corps quelconque dans \cite{J3}. La preuve du théorème \ref{theointroclass} se termine en montrant que l'on a $\ord_{s_0}(\Lf)+\ord_{s_1}(\Lf)\equiv 0\pmod{q-1}$ pour tout  $\Lf\in\pic_{[G]_2}(\hat{\H})$ et $\ord_{s_0}(\Lf)\equiv\ord_{s_1}(\Lf)\equiv 0\pmod{q-1}$ pour tout  $\Lf\in\pic_{G}(\hat{\H})$  ce qui est fait dans la proposition  \ref{proppoids}.

Les résultats précédant ont aussi quelques conséquences sur les fibrés $[G]_2$-équivariants de $\H$ et de $\bar{\H}$. En effet, tout fibré en niveau entier $\Lf$ donne lieu à un fibré en fibre générique $\Lf[1/\varpi]$ et  un  en fibre spéciale $\bar{\Lf}:=\Lf/\varpi\Lf$. L'image de ces applications est relativement facile à décrire et est une conséquence directe du théorème \ref{theointroclass} :

\begin{corointro}

Tous les fibrés  $\Lf[1/\varpi]\in\pic_{[G]_2}({\H})$ qui admettent une structure entière $\Lf$ sont de la forme \[\Lf[1/\varpi]\cong \Of(\chi)\otimes \Of(k)\otimes\Lf_0[1/\varpi ]^t\] De plus, toutes les structures entières de $\Lf[1/\varpi]$ sont : \[\Lf \cong \Of(\chi)\otimes \omega^{-k}_i\otimes\Lf_0^{t+(q-1)s}\] pour tout $s\in\Z$. Les fibrés $[G]_2$-équivariants sur $\hat{\H}$ triviaux en fibre générique sont engendrés par $\Lf_0^{q-1}$.

Les fibrés $\bar{\Lf}\in\pic_{[G]_2}(\bar{\H})$ proviennent tous d'un fibré de $\pic_{[G]_2}(\hat{\H})$ qui est unique à multiplication par un caractère trivial sur $\id+\varpi \mat_2(K)$ près.
\end{corointro}

Motivé par l'énoncé ci-dessus, nous introduisons les notions suivantes :
\begin{defi*}

Étant donné un fibré $[G]_2$-équivariant $\Lf$ sur $\hat{\H}$ ou sur  $\bar{\H}$, nous appellerons le poids $w(\Lf)$ et le type $t_{i,j}(\Lf)$ (avec $i,j\in\{0,1\}$) $\Lf$ les entiers suivants
\[
w(\Lf)=  \frac{-1}{q-1}(\ord_{s_0}(\Lf)+\ord_{s_1}(\Lf))\et
t_{i,j}(\Lf)= \ord_{s_j}(\Lf)+\ord_{s_j}(\omega_i)w(\Lf). 
\]
\end{defi*}
Ces constructions sont introduites pour que l'on ait les relations suivantes :
\[\Lf\cong \Of(\chi)\otimes \omega_i^{-w(\Lf)}\otimes\Lf_j^{t_{i,j}(\Lf)}\et \Lf[1/\varpi]\cong \Of(\chi)\otimes \Of(w(\Lf))\otimes\Lf_j[1/\varpi ]^{t_{i,j}(\Lf)}.\]
De plus, en fibre générique, tous les fibrés $\Lf[1/\varpi]$ de poids $k$ fixé sont les parties isotypiques de $\pi_*\pi^*\OC(k)$ pour l'action de $\OC_D^*/(1+\Pi_D\OC_D)\cong \F^*_{q^2}$ (a un caractère de $[G]_2$ près) avec $\pi :\Sigma^1\to\H$ la projection du premier revêtement $\Sigma^1:=(\XG[\Pi_D]\backslash\{0\})^{rig}$. 




Maintenant que nous pouvons déterminer entièrement des fibrés en droite $\Lf\in\pic_{[G]_2}(\hat{\H})$ par des couples d'entiers $(\ord_{s_0}(\Lf),\ord_{s_1}(\Lf))$ ou $(\ord_{s_0}(\Lf),w(\Lf))$ (tout cela a un caractère près qui ne changera pas drastiquement la structure des représentations étudiées), nous essaierons de lire la structure et les propriétés de la cohomologie cohérente des fibrés pour un couple donné. On commence par quelques théorèmes d'annulation dont la preuve n'est qu'une généralisation directe de \cite[Theorem 2.1]{GK6}.

\begin{propintro}\label{propintronul}

Soit un fibré $[G]_2$-équivariant $\Lf$ sur $\hat{\H}$, si $\hhh^i(\bar{\H},\bar{\Lf})=0$ pour un certain $i=0,1$, alors on a $\hhh^i(\hat{\H},{\Lf})=0$ et $\hhh^{1-i}(\hat{\H},{\Lf})$ est plat sur $\OC_K$ (la platitude est toujours vérifié en degré $0$) ainsi que $\hhh^{1-i}(\bar{\H},\bar{\Lf})=\hhh^{1-i}(\hat{\H},{\Lf})/\varpi$. 

Le résultat  précédent est applicable uniquement dans  l'un des trois cas suivants

\begin{itemize}
\item Le fibré est positif i.e. les entiers $(\ord_{s_0}(\Lf),\ord_{s_1}(\Lf))$ sont positifs. Dans ce cas, le terme $\hhh^1(\bar{\H},\bar{\Lf})$ s'annule.
\item Le fibré est négatif i.e. les entiers $(\ord_{s_0}(\Lf),\ord_{s_1}(\Lf))$ sont négatifs. Ici, $\hhh^0(\bar{\H},\bar{\Lf})=0$.
\item Le fibré est mixte (i.e. les entiers $(\ord_{s_0}(\Lf),\ord_{s_1}(\Lf))$ signes opposé) et $\ord_{s_j}(\Lf)\le p-1$ pour  $j=0,1$. Ces hypothèses entraînent $\hhh^0(\bar{\H},\bar{\Lf})=0$.
\end{itemize}

\end{propintro}

Nous pouvons aussi décrire la cohomologie mod $p$ des fibrés $[G]_2$-équivariants sur $\bar{\H}$. Plus précisément, nous faisons apparaître une filtration où les termes de la graduation font apparaître des représentations \[\chi\otimes\ind_{\gln_2(\OC_K)\varpi^\Z}^{[G]_2}\sym^r \bar{\F}^2, \ \ \ind_{I\varpi^\Z}^{[G]_2}\psi\] avec $I$ l'Iwahori, $\psi$ un caractère de $I$ et $\chi$ un caractère de ${[G]_2}$. Nous renvoyons à \ref{theofilt} pour l'énoncé précis dont la reproduction ici n'apportera guère peu de clarté à l'exposition. Cependant, ce que nous obtenons dans le cas des fibrés en droite positifs (voir la proposition \ref{propintronul}) de poids $-1$ quand $K=\Q_p$ est relativement frappant et nous terminerons cette introduction en expliquant ce résultat. Pour bien comprendre l'énoncé, nous allons rappeler la classification des mod $p$ représentations irréductibles admissibles  de $G$ par Barthel-Livne \cite{BL1, BL2}, Breuil \cite{breuil1} ainsi que celle de ${[G]_2}$ qui en découlera aisément. Pour $G$, les représentations irréductibles admissibles se distinguent alors en quatre familles disjointes (voir la section \ref{ssecg} pour leur description explicite) \nocite{breuil2,breuilmez2}

\begin{enumerate}
\item Les caractères
\item La série spéciale
\item La série principale
\item Les représentations supersingulières
\end{enumerate}

 Pour relier  les représentations irréductibles de  $G$ et de $[G]_2$, on rappelle le  fait standard qui affirme que la restriction à $[G]_2$ d'une représentation irréductible de $G$ est soit irréductible, soit la somme de deux représentations irréductibles et que l'on obtient toute représentation irréductible de $[G]_2$ comme un facteur direct de la restriction d'une représentation irréductible $G$. Décrire les représentations irréductibles de revient alors à comprendre comment chaque représentation des familles précédentes se décompose sous la restriction. On en déduit par restriction  quatre familles analogues de ${[G]_2}$-représentations irréductibles admissibles qui les exhaustent toutes : les caractères, la série spéciale, la série principale et les représentations supersingulières. D'après la description des représentations irréductibles de $\sln_2(K)$ prouvé indépendamment par Abdellatif \cite{abd1}, Cheng \cite{cheng}, on sait que les restrictions des représentations des familles $1$, $2$ et $3$ sont toutes irréductibles sur ${[G]_2}$ et la restriction des supersingulières de $G$ se décomposent\footnote{D'après ce qui vient d'être dit, on peut quand $K=\Q_p$, donner une autre caractérisation des représentations supersingulières qui évitent d'introduire certains opérateurs $T$ dans une algèbre de Hecke. En effet, on peut directement montrer qu'une représentation irréductible de $G$ est supersingulière si et seulement si sa restriction à ${[G]_2}$ se décompose. Savoir si un tel résultat est vrai sur un corps quelconque est une question ouverte à la connaissance de l'auteur.} quand $K=\Q_p$.

Le résultat suivant montre que les représentations supersingulières peuvent être géométriquement réalisées comme des sections globales de certains fibrés en droite $\bar{\Lf}\in\pic_{[G]_2}(\bar{\H})$.

\begin{theointro}

L'application  \[\bar{\Lf}\mapsto \hhh^0(\bar{\H},\bar{\Lf})^\vee\] fournit une bijection entre les fibrés en droite $\bar{\Lf}\in\pic_{[G]_2}(\bar{\H})$ positifs de poids $-1$ et les représentations supersingulières de $[G]_2$.

\end{theointro}

Ce résultat a la conséquence remarquable

\begin{corointro}

La représentation de $[G]_2$ \[ \hhh^0(\hat{\H},{\Lf})^\vee[1/\varpi]\] est irréductible pour un fibré en droite ${\Lf}\in\pic_{[G]_2}(\hat{\H})$ positif de poids $-1$.

\end{corointro}

\begin{rem}

\begin{itemize}
\item Malheureusement, le présent travail ne décrit pas la représentation ci-dessus. Toutefois, les deux résultats précédents suggèrent très fortement que la flèche \[{\Lf}\mapsto \ind_{[G]_2}^G \hhh^0(\hat{\H},{\Lf})^\vee\] mettra en bijection les fibrés en droite\footnote{Notons que $2$ fibrés $\Lf_1$, $\Lf_2$ définissent la même représentation $\ind_{[G]_2}^G \hhh^0(\hat{\H},{\Lf}_1)^\vee\cong \ind_{[G]_2}^G \hhh^0(\hat{\H},{\Lf}_2)^\vee$ si l'une est obtenue par rapport à l'autre en tordant l'action de $[G]_2$ par la conjugaison d'un élément de $w\in G$ qui n'est pas dans $[G]_2$. Le quotient  $\pic_{[G]_2}(\hat{\H})/\sim$ apparaissant dans l'énoncé  est alors l'ensemble des classes d'isomorphismes de fibrés équivariants modulo cette torsion.} $\Lf\in\pic_{[G]_2}(\hat{\H})/\sim$ positifs de poids $-1$ avec les représentations irréductibles de $G$ qui correspondent par  Langlands local aux représentations galoisiennes de poids de Hodge-Tate $(0,0)$. Cela fera l'objet d'un travail ultérieur. 
\item Un tel résultat conjectural est à comparer avec ceux obtenus dans \cite{brasdospi,pan} où, en fibre générique, ils étudient les formes différentielles des revêtements de la tour de Drinfeld (qui correspondent aux fibres de poids $-2$) et obtiennent ainsi les représentations qui correspondent aux représantations galoisiennes de poids de Hodge-Tate $(0,1)$.
\item Les deux représentations $\hhh^0(\hat{\H},{\Lf})^\vee[1/\varpi]$, $\hhh^0({\H},{\Lf}[1/\varpi])^\vee_b$ vivent sur des espaces topologiques de nature très différente. L'espace vectoriel sous-jacent de la première est un espace de Banach alors que c'est un espace de type compact  pour la seconde. Pour l'étude de l'admissibilité les représentations pour certains fibrés nous renvoyons à \cite{psm}.
\item Le résultat est évidemment faux quand le corps $K$ n'est pas $\Q_p$ même s'il est sûrement encore  possible de décrire explicitement $\hhh^0(\bar{\H},\bar{\Lf})^\vee$. C'est ce que nous avons fait quand $K$ est totalement ramifié où on voit que toutes représentations supersingulières est un quotient d'une telle représentation. 
\end{itemize}

\end{rem}

Les arguments utilisés pour prouver les résultats précédents se basent sur la suite exacte \[0\to \Lf\fln{\times\varpi}{}\Lf\to\bar{\Lf}\to 0\] et la suite exacte longue associée pour relier la cohomologique cohérente à la cohomologie mod $p$ et la suite exacte \[0 \fl \overline{\Lf}(\bar{\H})  \fl \prod_{s\in \BC\TC_0} \overline{\Lf}(\P_s) \fl \prod_{a=(s,s')\in \BC\TC_1} \overline{\Lf}(\P_s\cap \P_{s'}) \fl \hhh^1(\bar{\H}, \overline{\Lf})\fl \prod_{s\in \BC\TC_0} \hhh^1(\P_s, \overline{\Lf}) \fl 0\] qui peut être vue comme résultant d'un analogue de la suite spectrale de Cech pour le recouvrement $\bar{\H}\cong\bigcup_{s\in\BC\TC_0}\P_s$ par les composantes irréductibles.

\subsection*{Notations}

Dans tout l'article, on fixe un nombre premier\footnote{Beaucoup de résultat de cet article sont sûrement vrais quand $p=2$. Toutefois, la plupart des arguments faisant apparaître des caractères et se faisant apparaître des constants de la forme $(p^f-1)/2$ sont ou bien faux ou demande quelques adaptations.} $p\neq 2$ et une extension finie $K$ de $\Q_p$. On note $\mathcal{O}_K$ son anneau des entiers, $\varpi$ une uniformisante et $\F=\F_q$ son corps r\'esiduel. On note $C=\hat{\bar{K}}$ la complétion d'une clôture algébrique de $K$ et $\breve{K}$ la complétion de l'extension maximale non ramifiée de $K$. La plupart des espaces rigides considérés dans cet article vivront sur $\breve{K}$. On notera $K_{(2)}\subset\breve{K}$ l'extension non ramifiée de degré $2$ sur $K$ d'anneau des entiers $\OC_{(2)}$ et de corps résiduels $\F_{(2)}=\F_{q^2}$. Il pourra être  utile de poser $N=q^2-1=(q-1)\tilde{N}$ et  d'introduire $\varpi^{1/N}$ une racine $N$-ième de l'uniformisante $\varpi$. Nous étudierons aussi l'algèbre des quaternions $D$ sur $K$ d'invariant $1/2$ qui contient $K_{(2)}$. Nous pouvons la voir comme $D=K_{(2)}[\Pi_D]$ soumis aux relations $\Pi_D^2=\varpi$, $\Pi_D a=\sigma(a)\Pi_D$ pour $a\in K_{(2)}$ où $\sigma$ est l'unique endomorphisme non triviale de $\gal(K_{(2)}/K)$. L'anneau des entiers de $D$ sera noté $\OC_D$ et le corps résiduel  est $\OC_D/\Pi_D\cong\F_{(2)}$.

Le but de cet article est d'étudier les représentations du groupe $G=\gln_2(K)$. Il contient un sous-groupe d'indice de constituer des matrices dont la valuation du déterminant est pair $[G]_2=(v_K\circ \det)^{-1}(2\Z)$. Nous exhibons deux représentants  $w= \left( \begin{matrix}
0 & 1\\ \varpi & 0
\end{matrix} \right)$, $\alpha= \left( \begin{matrix}
1 & 0\\ 0 & \varpi
\end{matrix} \right)$ de la classe non triviale de $G/[G]_2$. Nous essaierons de comprendre la théorie des représentations de $G$ via celle de $[G]_2$ (voir le paragraphe \ref{ssecgdeux}  pour plus de précisions sur cette philosophie). Pour toute représentation $\rho$ de $H$ un sous-groupe de $[G]_2$,  on pourra considérer la  représentation $\rho^w$ de $wHw^{-1}$ caractérisée par \begin{equation}\label{eq:w}
\rho^w(h)=\rho(w^{-1}hw)
\end{equation} pour tout $h\in H$. La contragrédiente  d'une représentation $\rho$ sera notée $\rho^\vee$.

 Le centre de $G$ sera noté $Z=\{\lambda \id_2\}$ et on posera $G^{\circ}:=\gln_2(\OC_K)$ et $\bar{G}=\gln_2(\F)=\gln_2(\OC_K)/(1+\varpi M_2(\OC_K))$. On notera  $T$ le tore de $G$ formé des matrices diagonales  et $B$ le Borel des matrices triangulaires supérieures et $\bar{T}$, $\bar{B}$ le tore et le Borel analogue de $\bar{G}$   (à ne pas confondre avec le Borel opposé à $B$ dans $G$). Les sous-groupes d'Iwahori $I$ et le pro-$p$-Iwahori $I_1$ seront les  pré-images de $\bar{B}$ et $\bar{T}$ dans $G^\circ$. 

Étant donné deux groupes $H_1\subset H_2$, et $\sigma$ une représentation de $H_1$ d'espace sous-jacent $V_\sigma$, on identifiera les espaces sous-jacents de l'induction $\ind_{H_1}^{H_2}\sigma$ et de l'induction compacte $\cind_{H_1}^{H_2}\sigma$ avec \begin{equation}\label{eq:ssjac}
\prod_{s\in H_2/H_1}V_s\ \et\ \bigoplus_{s\in H_2/H_1}V_s
\end{equation} où chaque terme $V_s$ est isomorphe (en tant qu'espace vectoriel) non canoniquement à $V_\sigma$. Pour tout élément $v=(v_s)_{s\in H_2/H_1}$ de $\prod_{s\in H_2/H_1}V_s$ ou de  $\bigoplus_{s\in H_2/H_1}V_s$, on appellera support de $v$ l'ensemble $\supp(v)\subset H_2/H_1$ \begin{equation}\label{eq:supp}
\supp(v):=\{s\in H_2/H_1:v_s\neq 0\}.
\end{equation}

Nous confondrons les caractères $\chi: \Gamma\to R^*$ (avec $\Gamma =K^*,\OC_K^*\varpi^{2\Z},{\F}^*$ et   $R=\breve{K}, \OC_{\breve{K}},\bar{\F}$)  avec les caractères $\chi\circ\det$ de $G$, $[G]_2$, $\bar{G}$ associés. Par exemple, on notera $\delta_a$ avec $a\in R^*$  le caractère de $\OC_K^*\varpi^{2\Z}$ et donc de $[G]_2$ et $G^{\circ}\varpi^\Z$ défini par \begin{equation}\label{eq:del}
\delta_a(\OC_K^*)=\{1\},\ \ \delta_a(\varpi^{2\Z})=a.
\end{equation} De même, le symbole de Legendre définit des caractères de $\F^*$ et $\OC_K^*\varpi^{2}$ via \begin{equation}\label{eq:leg}
\OC_K^*\varpi^{2\Z}\fln{\pmod{(1+\varpi\OC_K)\varpi^{2\Z}]}}{}\F^*\to \{\pm 1\}\subset R^*
\end{equation} et donc des caractères  de $[G]_2$ et $\bar{\F}$ que l'on appellera caractère de Legendre et que l'on  notera $\leg$. On peut aussi voir  tout caractère  $\mu$ de $T$ (resp. $\bar{T}$) pourra être vu comme un caractère de $B$ (resp. $\bar{B}$ ou de $I$) par inflation. Par exemple, nous écrirons $\mu_{m,n}$ avec $m$, $n$ dans $\Z/(q-1)\Z$ les caractères de $T$, $\bar{T}$, $B$, $\bar{B}$, $I$ à valeurs dans $\bar{\F}^*$ définis par \begin{equation}\label{eq:mu}
\mu_{m,n}\left( \begin{matrix}
a & 0\\ 0 & d
\end{matrix} \right)=a^m d^n.
\end{equation}

\section{Conventions sur le demi-plan de Drinfeld}

\subsection{L'arbre de Bruhat-Tits} 

Intéressons-nous à l'arbre de Bruhat-Tits $\BC\TC$ le graphe dont les  sommets $\BC\TC_0$ s'identifient à $\gln_2(K) / K^*\gln_2(\OC_K)$ l'ensemble des réseaux de $K^2$ à homothétie près et les arêtes $\BC\TC_1$ sont les  couples $a=(s_0,s_1)\in  \BC\TC_0^2$  pour lesquels  on peut trouver des représentants  $M_0, M_1\subset K^2$ de $s_0,  s_1$ tels que  
\[
\varpi M_1\subsetneq M_0\subsetneq   M_1. 
\]
   En réduisant modulo $\varpi$ la chaîne d'inclusion ci-dessus, on obtient un drapeau qui caractérise l'arête $a= (s_0,s_1)$ quand $s_0$ est fixé:
\[
0 \subsetneq M_0/ \varpi M_1\subsetneq M_1/\varpi M_1.
\]
Ainsi les arêtes adjacentes à $s_1$ s'identifient aux points $\F$-rationnels de la droite projective $\P^1_\F$.

On a une action de $G:= \gln_2(K)$  sur les sommets qui préserve les relations d'incidence.   Elle est transitive sur $\BC\TC_0$ et les stabilisateurs sont des conjugués de $\gln_2(\OC_K) \varpi^{\Z} :=G^{\circ} \varpi^{\Z}$.  De plus,  $G^{\circ} \varpi^{\Z}$ agit transitivement sur $\P^1(F)$ et donc sur  les arêtes incidentes au sommet standard $s_1=[\OC_K^2]$.  D'où la transitivité de $G$ sur $\BC\TC_1$.  On appellera\footnote{Il peut paraître surprenant et contre-intuitif d'imposer que $s_1$ est le sommet standard et non $s_0$ (cela interviendra aussi pour les notions de parité). Cela provient du faite que l'on a décrit les arêtes $([M_0],[M_1])$ comme des chaînes de la forme $\varpi M_1\subsetneq M_0\subsetneq   M_1$ et non $\varpi M_0\subsetneq M_1\subsetneq   M_0. $. Cela est motivé par la volonté de suivre les notations existantes dans la littérature que nous citons tout au long de ce travail.}  $(s_0,s_1)=( [\OC_K\times \varpi \OC_K], [\OC_K^2])$ l'arête standard. 

On a une notion de distance sur les sommets qui est donnée par la longueur du plus court chemin sur le graphe $\BC\TC$. Nous appellerons sommets pairs $\BC\TC_0^{(0)}$ ceux qui sont à distance paire de $s_0$ et sommets impairs $\BC\TC_0^{(1)}$ les autres.  Le sous-groupe maximal de $G$ qui préserve la parité est l'ensemble des matrices dont la valuation du déterminant est paire $[G]_2=(v_K\circ \det)^{-1}(2\Z)$  si bien que les  sous-ensembles $\BC\TC_0^{(0)}$ et $\BC\TC_0^{(1)}$ sont exactement les orbites de $[G]_2$.  Notons $\hat{\BC\TC}_1$ les arêtes muni d'une orientation.  Elle sera dite positive si la source est un sommet pair et négative dans le cas contraire.  Le sous-groupe $[G]_2$ agit transitivement sur les arêtes d'orientation positive $\hat{\BC\TC}_1^+$ et  négative $\hat{\BC\TC}_1^-$ (et en particulier sur les arêtes non orientées). 

Nous écrirons $w= \left( \begin{matrix}
0 & 1\\ \varpi & 0
\end{matrix} \right)$ et $\alpha= \left( \begin{matrix}
 1 &0\\ 0 & \varpi
\end{matrix} \right)$ deux représentants de  l'élément non trivial de $G/[G]_2\cong \Z/2\Z$. Ils  permutent $\BC\TC^{(0)}_0$ et $\BC \TC_0^{(1)}$ ainsi que $\hat{\BC \TC}_1^+$ et $\hat{\BC \TC}_1^{-}$ l'un en permutant les deux sommets standards, l'autre en agissant par translation sur l'appartement standard $\{[\OC_K\times \varpi^n \OC_K]=\alpha^n s_1\}\subset\BC\TC_0$.  La décomposition de Cartan \begin{equation}\label{eq:cartan}
G=\bigsqcup_{n\in\N} G^\circ  \alpha^nG^\circ \varpi^\Z
\end{equation}  montre que les orbites de $\BC\TC_1$ sous l'action de  $G^\circ \varpi^\Z$ sont $G^\circ \varpi^\Z \alpha^n s_1$ et sont formés des sommets à distance fixée $n$ du sommet standard $s_1$.

\subsection{Le demi-plan de Drinfeld}

\'Ecrivons $\H$ l'ouvert admissible de la droite projective rigide sur $\breve{K}$ obtenu en retirant les points $K$-rationnelles.  La construction suivante justifiera la structure d'espace rigide.    On a une flèche de réduction
 \[\tau : \H(\C_p)\to \{\text{normes sur } K^{2}\}/\{\text{homothéties}\} \] 
 donnée par 
 $$\tau (z): v\mapsto |z_0v_0+z_1v_1|$$ 
 si $z=[z_0,z_1] \in \H (\C_p)$. L'image $\tau(z)$ ne dépend pas du représentant de $z$ dans $\C_p^2$ car les normes sont vues à  homothétie près. Le fait de prendre le complémentaire des hyperplans $K$-rationnels assure que $\tau(z)$ est bien une norme sur $K^{2}$.

D'après un résultat classique de Iwahori-Goldmann \cite{iwgo} l'espace des normes sur $K^{2}$ à  homothétie près s'identifie bijectivement (et de manière $G$-équivariante) à  l'espace topologique $|\BC\TC|$, ce qui permet de voir $\tau$ comme une application 
$$\tau: \H(\C_p)\to |\BC\TC|.$$ 

L'observation cruciale est de voir que les espaces $\tau^{-1}(s)=: \H_s$ et $\tau^{-1}(a)=: \H_a$ pour $s\in \BC\TC_0$ et $a\in \BC\TC_1$ sont les affinoides (cf.  \cite[§1.4]{J2})
\begin{equation}\label{eq:ouvsimp}
\H_s= \spg(\breve{K}\langle X, \frac{1}{X^q-X} \rangle) \et  \H_a= \spg(\breve{K}\langle  X,Y, \frac{1}{X^{q-1}-1} , \frac{1}{Y^{q-1}-1} \rangle/ (XY-\varpi) ).
\end{equation}

Notons que si $s\in a$,  $\H_s$ s'obtient comme l'ouvert de $\H_a$ où on a imposé  $|X|=1$ ou $|Y|=1$.   

Les espaces $\hat{\H}_s= \spf \Of^+(\H_s)$ et $\hat{\H}_a= \spf \Of^+(\H_a)$ se recollent pour définir un modèle entier de réduction semi-stable $\hat{\H}$ de $\H$.  



\subsection{Fibre spéciale de modèle entier $\hat{\H}$ et fonctions inversibles en fibre spéciale}

Nous décrivons ici les composantes irréductibles de la fibre spéciale $\bar{\H}$ de $\hat{\H}$ et leur combinatoire.    Pour cela nous avons besoin de définir l'étoile  ouverte  $\ost(s)$ d'un sommet $s$.  Il s'agit de l'ouvert de $\BC\TC$ constitué de l'union de toutes les arêtes qui rencontrent $s$ privée de  l'autre extrémité i.e. $\ost(s)=\bigcup_{a\in \BC\TC_1 :a=(s,s')}a\backslash s'$.  Chaque  préimage d'une étoile ouverte décrit le tube au-dessus  d'une composante irréductible de $\bar{\H}$.    Ainsi,  les composantes sont en bijection avec les sommets et les intersections avec les arêtes.  De plus, 
 chacune d'elles s'identifient à une droite projective que l'on notera $\P_s$.   Deux composantes s'intersectent en un unique  point qui est $\F$-rationnel et on fixe, pour tout sommet $s$, une bijection  entre $\P_s(\F)$ et l'ensemble des arêtes rencontrent $s$. 
 
 Le lieu lisse d'une composante $\P_s\backslash \P_s(\F)$ s'identifie à $\bar{\H}_s$ la fibre spéciale de l'ouvert $\hat{\H}_s$ introduit précédemment.   De manière similaire,  $\bar{\H}_a= \P_s\cup \P_{s'} \backslash (\P_s(F)\Delta \P_{s'}(F))$  pour $a=(s,s')$. Toujours dans ce cas,  $\bar{\H}_a$  admet deux composantes irréductibles $V(X)$ et $V(Y)$ qui correspondent à $\bar{\H}_a\cap \P_{s}$ et $\bar{\H}_a\cap \P_{s'}$. 
 
Nous terminons cette section par ce résultat qui décrit le lien entre les fonctions inversibles en niveau entier et celles en fibre générique et en fibre spéciale.

\begin{prop}\label{proposthsig}

Soit $s\in\BC\TC_0$ et $a=\{s,s'\}\in\BC\TC_1$, on a 
\[\Of^*(\hat{\H}_s)/\OC^*_{\breve{K}}\Of^{**}(\hat{\H}_s)\cong \Of^*(\bar{\H}_s)/\bar{\F}^*=\prod_{b\in \F} (X-b)^{\Z},\]
\[\Of^*(\hat{\H}_a)/\OC^*_{\breve{K}}\Of^{**}(\hat{\H}_a)\cong \Of^*(\bar{\H}_a)/\bar{\F}^*= \prod_{b\in \F^*} (X-b)^{\Z}\times \prod_{b\in \F^*}(Y-b)^{\Z}\] (les variables $X\in \Of(\hat{\H}_s)$ et $X,Y\in \Of(\hat{\H}_a)$ sont les mêmes que dans  \eqref{eq:ouvsimp}). On a aussi \[\Of^*({\H}_s)=\Of^*(\hat{\H}_s)\varpi^\Z\et \Of^*({\H}_a)=\Of^*(\hat{\H}_a)X^\Z Y^\Z\]
\end{prop}

\begin{proof}
La preuve apparaît dans les arguments de la démonstration de \cite[Prop. 3.8.]{J2} mais nous les répétons très succinctement pour le confort du lecteur.

Commençons par la dernière identité de l'énoncé. Étant donné une fonction inversible $u$ sur $\Of^*({\H}_s)$ ou $\Of^*({\H}_a)$, on utilise le fait que $\hat{\H}_s$ et $\hat{\H}_a$ sont pluri-nodales sur $\spf(\OC_{\breve{K}})$ (au sens de \cite[Definition 1.1]{ber7}) pour montrer qu'il existe $\lambda\in \breve{K}^*$ tel que $u/\lambda$ est dans $\Of(\hat{\H}_s)\backslash\mG_{\breve{K}} \Of(\hat{\H}_s)$ ou $\Of(\hat{\H}_a)\backslash\mG_{\breve{K}} \Of(\hat{\H}_a)$. Pour le  sommet, on obtient  directement $u/\lambda\in \Of^*(\hat{\H}_s)$ car le modèle est lisse. Sur l'arête,  on peut  trouver un élément de la forme $X^\alpha$ ou de la forme $Y^\alpha$ avec $\alpha\in\Z$ qui a les mêmes ordres d'annulation que $u/\lambda$ en les points génériques des composantes irréductibles de la fibre spéciale. On en déduit que $\frac{u}{\lambda X^\alpha}$ ou $\frac{u}{\lambda Y^\alpha}$ est inversible dans $\Of(\hat{\H}_a)$ ce qui prouve le dernier point.

Pour les deux premières identités, le sous-groupe $\OC_{\breve{K}}^*\Of^{**}(\hat{\H}_s)$ (resp. $\OC^*_{\breve{K}}\Of^{**}(\hat{\H}_s)$) est clairement le noyau de la flèche de projection $\Of^*(\hat{\H}_s)\to \Of^*(\bar{\H}_s)/\bar{\F}^*$ (resp. $\Of^*(\hat{\H}_a)\to \Of^*(\bar{\H}_a)/\bar{\F}^*$) et on se ramène à calculer $\Of^*(\bar{\H}_s)/\bar{\F}^*$ et $\Of^*(\bar{\H}_a)/\bar{\F}^*$. Sur le sommet, le résultat est clair  car $\bar{\H}_s\cong \P^1_\F\backslash \P^1(\F)$ dont on sait calculer les fonctions inversibles. Pour une arête, on a deux composantes irréductibles de la fibre spéciale $V(X)$ et $V(Y)$ et une fonction sur $\bar{\H}_a$ est déterminé par la restriction à chacune de ces composantes. De plus, une telle fonction est inversible si et seulement si chacune de ces restrictions l'est et on en déduit une injection :
\[\Of^*(\bar{\H}_a)/\bar{\F}^*\flinj \Of^*(V(Y))/\bar{\F}^*\times\Of^*(V(X))/\bar{\F}^*\cong  \prod_{b\in \F^*} (X-b)^{\Z}\times \prod_{b\in \F^*}(Y-b)^{\Z}\]
Mais on voit que chaque fonction $((X-b))_{b\in \F^*}$ et  $((Y-b))_{b\in \F^*}$ est inversible en fibre spéciale par définition et la flèche précédente est une bijection.
\end{proof}
\subsection{Interprétation modulaire\label{ssecmod}}

Nous aimerions étudier des fibrés en droites provenant d'une interprétation modulaire $\hat{\H}$ que nous allons présenter ici.  

Si $A$ est une $\OC_K$-algèbre, un \emph{$\OC_D$-module formel} sur ${\rm Spec}(A)$ (ou, plus simplement, sur $A$) est un groupe formel $F$ sur $A$ muni d'une action de $\OC_D$, notée $\iota : \OC_D \to \mathrm{End}(F)$, qui est compatible avec l'action naturelle de $\OC_K$ sur l'espace tangent $\mathrm{Lie}(F)$, \emph{i.e.} pour $a$ dans $\OC_K$, $d\iota (a)$ est la multiplication par $a$ dans $\mathrm{Lie}(F)$. Le 
  $\OC_D$-module formel $F$ est dit \emph{spécial} si $\mathrm{Lie}(F)$ est un 
  $\OC_{(2)}\otimes_{\OC_K} A$-module localement libre de rang $1$. On a le r\'esultat classique suivant: 





\begin{prop}
  Sur un corps alg\'ebriquement clos de caract\'eristique $p$, il existe un unique 
   $\OC_D$-module formel sp\'ecial de dimension $2$ et de $(\OC_K$-)hauteur $4$, à isogénie près. 
\end{prop}

  On notera $\Phi_{\bar{\F}}$ l'unique (à isogénie près) 
$\OC_D$-module formel sp\'ecial $\Phi_{\bar{\F}}$ sur $\bar{\F}$ de dimension $2$ et hauteur $4$.

Consid\'erons le foncteur $\GC^{Dr} : \mathrm{Nilp} \to \mathrm{Ens}$ envoyant $A\in \mathrm{Nilp}$ sur l'ensemble des classes d'isomorphisme de triplets $(\psi, F, \rho)$ avec : 
\begin{itemize}[label= \textbullet] 
\item $\psi : \bar{\F} \to A/ \varpi A$ est un $\F$-morphisme, 
\item $F$ est un $\OC_D$-module formel sp\'ecial de dimension $2$ et de hauteur $4$ sur $A$, 
\item $\rho : \Phi_{\bar{\F}} \otimes_{\bar{\F}, \psi} A/ \varpi A \to F_{A/ \varpi A}$ est une quasi-isog\'enie de hauteur z\'ero. 
\end{itemize}
   
 Le théorème fondamental suivant, à la base de toute la théorie, est dû à Drinfeld : 

\begin{theo}[\cite{dr2}]\label{Drrep}
Le foncteur $\GC^{Dr}$ est repr\'esentable par $\hat\H$.
\end{theo} 
 
Dans toute la suite nous appellerons  $\XG$ le module formel universel sur $\hat{\H}$.  Un objet d'importance fondamentale à étudier est son algèbre  de Lie ainsi que sa décomposition en composantes isotypiques  pour l'action de $\OC_{(2)}$
\[
\lie (\XG)= \lie (\XG)_0 \oplus \lie(\XG)_1
\]  
et le dual de ces parties isotypiques $\omega_0$ et $\omega_1$. 


\section{Fibrés équivariants}

\subsection{Fibrés  en niveau entier}


Nous allons étudier les fibrés sur le modèle entier $\hat{\H}$. En fibre générique, la question est plus simple comme en témoigne le résultat suivant (\cite{J1}) 
\begin{theo}\label{theopicgen}
$\pic(\H)=0$
\end{theo} 
Le problème est plus subtil en niveau entier et le but de cette section est de prouver le résultat suivant :

\begin{theo}\label{theopicentier}
La flèche naturelle de  restriction à chaque composantes irréductibles de la fibre spéciale est un isomorphisme : \[\pic(\hat{\H})\iso\prod_{s\in \BC\TC_0}\pic(\P_s)\cong \prod_{s\in\BC\TC_0}\Z.\]
\end{theo}

D'après le résultat précédent, nous pouvons donner la définition suivante.
\begin{defi}
Pour tout fibrés en droites $\Lf$, nous appellerons ordre en $s\in \BC\TC_0$ et noterons $\ord_s(\Lf)$ l'image dans $\Z$ par la projection $\pic(\bar{\H})\to\pic(\P_s)$ issue de l'isomorphisme précédent.
\end{defi}

Avant de prouver le résultat, nous donnons une autre caractérisation de ces fibrés qui nécessite  la notion de norme introduite par    Teitelbaum  dans \cite[p656]{teit1}. 

\begin{defi}
Soit $X$ un schéma formel sur $\OC_{\breve{K}}$ et $\Mf$ un faisceau en $\Of_{X^{rig}}$-modules.    Une norme sur $\Mf$ est la donnée pour toute ouvert admissible $U\subset X^{rig}$ d'une fonction 
\[
\Vert \cdot  \Vert_U: \Mf(U)\to \R_{\geq 0}.   
\] 
qui vérifie les axiomes suivants 
\begin{enumerate}

\item Si $U\subset V \subset X^{rig}$ et $f\in \Mf(V)$,  $\Vert f \Vert_V \subset  \Vert f \Vert_U$.

\item  Pour $V\subset X^{rig}$,  et $\UC$ un  recouvrement  admissible de $V$,   on a 
\[
\Vert f \Vert_V \leq  \max_{U\in \UC} \Vert f  \Vert_U. 
\]

\item  Pour  $U\subset X^{rig}$,  $f\in \Of(U)$ et $g\in\Mf(U)$,  on a $\Vert f g \Vert_U = \vert f \vert  \Vert g \Vert_U$ où $\vert \cdot \vert$ est la norme spectrale.

\end{enumerate}

\end{defi}

\`A tout faisceau en $\Of$-modules $\Mf$ muni d'un norme sur $X^{rig}$,  on peut associer un faisceau $\hat{\Mf}$ sur $X$    défini par 
\[
\forall U \subset X,\; \;   \hat{\Mf} (U)= \{ f \in  \Mf(U^{rig}) :  \Vert f \Vert_{U^{rig}} \leq 1  \}.  
\]

\begin{ex}
Si $X= \hat{\H}$,  $\Of_{\hat{\H}}$ provient de $\Of_{\H}$ muni de la norme spectrale par ce procédé.  
\end{ex}

\begin{prop}
\label{propnorme}
Tout fibré en droites sur $\hat{\H}$  provient d'une norme sur $\Of_{\H}$.   
\end{prop}

\begin{rem}
En revanche,  toute norme n'induit pas un fibré. 
\end{rem}

\begin{proof}
Cela repose sur  ce résultat technique qui nous servira aussi dans les sections suivantes.

\begin{lem}
Soit $s\in \BC\TC_0$ un sommet et $a\in \BC\TC_1$ une arête, on a    
$\pic(\hat{\H}_{a})=\pic(\hat{\H}_{s})=0$.
\end{lem}

\begin{proof}

D'après \cite[Sec. 3.7.4]{frvdp}, on a $\pic(\hat{\H}_a)=\pic(\bar{\H}_{a})$ et $\pic(\hat{\H}_{s})=\pic(\bar{\H}_{s})$ avec
\[
\Of(\bar{\H}_{a})=\bar{\F}[X,Y, \frac{1}{(X^{q-1}-1)(Y^{q-1}-1)}]/(XY) 
\]
\[
\Of(\bar{\H}_{s})=\bar{\F}[X, \frac{1}{(X^{q}-X)}].
\]
L'anneau précédent étant factoriel,  on a $\pic(\hat{\H}_{s})=0$.
 On a de plus le recouvrement par les composantes irréductibles $\bar{\H}_{a}=V(X)\cup V(Y)$. On prouve l'exactitude\footnote{On reprend l'argument de \cite{wa}} de
\begin{equation}
\label{eqexost}
1 \fln{}{} \Of_{\bar{\H}_{a}}^* \fln{\alpha}{} \iota_* \Of^*_{V(X)}\times \iota_*\Of^*_{V(Y)} \fln{\beta}{} \iota_* \Of^*_{V(X,Y)} \fln{}{}1. 
\end{equation}
Pour l'exactitude au centre, on observe la suite exacte (version additive)
\begin{equation}
\label{eqexo}
0 \fln{}{} \Of_{\bar{\H}_{a}} \fln{\tilde{\alpha}}{} \iota_* \Of_{V(X)}\times \iota_* \Of_{V(Y)} \fln{\tilde{\beta}}{} \iota_* \Of_{V(X,Y)} \fln{}{}0, 
\end{equation}
et les identifications ensemblistes ${\rm Im}(\tilde{\alpha}|_{\Of^*_{\bar{\H}_{a}}})={\rm Im}(\alpha)$  et $\ker(\tilde{\beta}|_{\Of^*_{\bar{\H}_{a}}})=\ker(\beta)$. Le faisceau $\iota_* \Of_{V(X,Y)}^*$  est concentré en le point fermé $V(X,Y)$ et la tige en ce point vaut $\bar{\F}^*$. Ainsi $\beta$ est surjectif en toutes les tiges.  

La suite exacte longue associée à \eqref{eqexost} nous donne 
\[
\Of^*(V(X))\times\Of^*(V(Y))\fl \Of^*(V(X,Y)) \fl \pic(\bar{\H}_{a}) \fl \pic(V(X))\times \pic(V(Y)).
\]
La première flèche est clairement surjective car $\Of^*(V(X,Y))=\bar{\F}^*$. De plus, $\pic(V(X))=\pic(V(Y))=0$ car $V(X)$ est affine et $\Of(V(X))$ est factoriel en tant qu'ouvert de $\A^1_{\bar{\F}}$. D'où le résultat.


\end{proof}

Prenons  $\Lf\in \pic(\hat{\H})$,  le fibré associé en fibre générique $\Lf[1/\varpi]$ est trivial.  Fixons une arête $a\in \BC\TC_1$ et montrons que l'on peut trouver un générateur de $\Lf|_{\hat{\H}_a}$ qui se prolonge en une section  globale génératrice\footnote{Pour tout $U\subset \hat{\H}$,  on a une inclusion $\Lf(U)\subset \Lf[1/\varpi](U^{rig})$. } de $\Lf[1/\varpi]$.      Soit $v_1$ un générateur de $\Lf|_{\hat{\H}_a}$ et $v_2$ un générateur de $\Lf[1/\varpi]$.   Elles définissent toutes les deux des générateurs de $\Lf[1/\varpi]|_{{\H}_a}$.  Donc,  $v_2= u v_1$ avec $u\in \Of^*({\H}_a)$. Mais on a   $
\Of^*(\H_a)=  \Of^*(\hat{\H}_a) \Of^*(\H)$ d'après \ref{proposthsig} (car $X$, $Y$ et $\OC_{\breve{K}}^*$ sont dans $\Of^*(\H)$) et  écrivons $u= u_1  u_2$ avec $u_1\in \Of^*(\hat{\H}_a)$ et $u_2\in \Of^*(\H)$.  On voit alors que $v_a:=u_2^{-1}v_2= u_1v_1  $ est une section globale génératrice de $\Lf[1/\varpi]$ et de $\Lf|_{\hat{\H}_a}$.  

D'après ce qui précède,  on trouve des sections $v_a\in \Lf[1/\varpi](\H)$  pour tout $a\in \BC \TC_1$,   tel que la restriction  à  $\H_a$ est  un générateur de $\Lf(\hat{\H}_a)$.   Pour $U\subset \H_a$ et $r=fv_a \in \Lf[1/\varpi](U)$, on définit la norme $\Vert r \Vert_U:= |f|$ où  $|\cdot |$ est la norme spectrale dans $\Of(U)$.  Notons que si $V\subset \H_s$ pour un sommet $s\in \BC\TC_0$ alors  $\Vert \cdot \Vert_V$  ne dépend pas de l'arête $a$ contenant $s$.  Soit $V\subset \H$    un ouvert quasi-compact quelconque et  $r\in \Lf(V)$,  posons $\Vert r  \Vert_V= \max_{a\in \BC \TC_1} ||r  ||_{V\cap  \H_a} $.  Nous avons donc exhibé une norme sur $\Of\cong \Lf[1/\varpi]$ dont le faisceau associé  est $\Lf$ par construction.


\end{proof}

Donnons ici la preuve de \ref{theopicentier} :

\begin{proof}[Démonstration de \ref{theopicentier}]
Par abus, nous appellerons $\BC\TC_1$ les recouvrements de $\H$ (resp. $\hat{\H}$, $\bar{\H}$)  par les ouverts $\H_a$ (resp. $\hat{\H}_a$, $\bar{\H}_a$) avec $a\in\BC\TC_1$. D'après l'annulation du groupe de Picard de $\hat{\H}_a$ montré dans \ref{LemmeBase}, on a une identification $\pic(\hat{\H})\cong \hcech{1}(\hat{\H},\BC\TC_1,\Of^*)$. De plus, la suite exacte d'après \ref{proposthsig} \[0\to\OC_{\breve{K}}^*\Of^{**}(\H_{\sigma})\to \Of^*(\hat{\H}_{\sigma})\to \Of^*(\bar{\H}_{\sigma})/\bar{\F}^*\to 0, \forall\sigma\in\BC\TC_0\cup\BC\TC_1\] induit une suite exacte courte de complexe\footnote{Ici, on voit $\OC_{\breve{K}}^*\Of^{**}_{\H}$ et $\Of^{*}_{\bar{\H}}/\bar{\F}^*$ comme des préfaisceaux} \[0\to \ccech{\bullet}({\H},\BC\TC_1,\OC_{\breve{K}}^*\Of^{**}_{\H})\to\ccech{\bullet}(\hat{\H},\BC\TC_1, \Of^{*}_{\hat{\H}}) \to\ccech{\bullet}(\bar{\H},\BC\TC_1,\Of^{*}_{\bar{\H}}/\bar{\F}^*)\to 0\]  qui  donne  lieu à une suite exacte longue\footnote{On n'a pas de cohomologies en degré supérieur ou égal à $2$ car les intersections de trois ouverts sont triviales\label{footnarb}} \[\hcech{1}({\H},\BC\TC_1,\OC_{\breve{K}}^*\Of^{**}_{\H})\to \hcech{1}(\hat{\H},\BC\TC_1,\Of^*_{\hat{\H}})\to \hcech{1}(\bar{\H},\BC\TC_1,\Of^*_{\bar{\H}})\to 0.\]

Il suffit donc de montrer $\hcech{1}({\H},\BC\TC_1,\OC_{\breve{K}}^*\Of^{**}_{\H})=0$ et $\hcech{1}(\bar{\H},\BC\TC_1,\Of^*_{\bar{\H}})\cong\prod_{s\in \BC\TC_0}\pic(\P_s)$. Pour la première identité, on a une suite exacte courte \[0\to \ccech{\bullet}({\H},\BC\TC_1,1+\mG_{\breve{K}})\to\ccech{\bullet}({\H},\BC\TC_1, \OC_{\breve{K}}^*) \oplus \ccech{\bullet}({\H},\BC\TC_1,\Of^{**}_{\H}) \to\ccech{\bullet}({\H},\BC\TC_1,\OC_{\breve{K}}^*\Of^{**}_{\H})\to 0\] et donc une suite exacte longue \[\hcech{1}({\H},\BC\TC_1, \OC_{\breve{K}}^*) \oplus \hcech{1}({\H},\BC\TC_1,\Of^{**}_{\H}) \to\hcech{1}({\H},\BC\TC_1,\OC_{\breve{K}}^*\Of^{**}_{\H})\to \hcech{2}({\H},\BC\TC_1,1+\mG_{\breve{K}}).\] On a directement $\hcech{2}({\H},\BC\TC_1,1+\mG_{\breve{K}})=0$ (voir note \ref{footnarb}) et pour n'importe quelle faisceau $\Ff$, le groupe $\hcech{1}({\H},\BC\TC_1,\Ff)$ s'injecte dans $\han{1}({\H},\Ff)$ par suite spectrale de Cech. Mais l'annulation des groupes $\han{1}({\H},\Ff)$ pour $\Ff=\OC_{\breve{K}}^*\Of^{**}_{\H}$ a été montrée\footnote{Pour les faisceaux $\Ff$ constants, l'annulation des groupes de cohomologie $\han{1}({\H},\Ff)$ n'a pas été explicitement énoncée dans \cite{J1} mais la stratégie employée permet aussi de traiter ces cas beaucoup plus simples. On observe d'abord l'acyclicité des couronnes d'après \cite{vdp} et on en déduit le même résultat pour certaines unions standard de telles couronnes (qui correspondent aux fibrations $X^d_t(\beta)$ considérées dans \cite[§2.2]{J1}) par calcul explicite du complexe de Cech (qui calcule la cohomologie simplicial du segment.). Le lemme combinatoire \cite[Lem. 5.7]{J1} permet alors d'étendre l'acyclicité pour les ouverts $\bar{U}_n$ apparaissant dans le recouvrement Stein de $\H$ (voir \cite[§1.3]{J3} pour leur définition). Le résultat pour $\H$ s'obtient par passage à la limite en prouvant l'annulation $\rrr^1 \limp_n \han{1}(\bar{U}_n,\Ff)$ qui est claire car le système projectif $(\han{1}(\bar{U}_n,\Ff))_n$ est constant.} dans \cite[Th. 7.1, Rem. 6.4]{J1} ce qui entraîne $\hcech{1}({\H},\BC\TC_1,\OC_{\breve{K}}^*\Of^{**}_{\H})=0$.

Pour le second terme $\hcech{1}(\bar{\H},\BC\TC_1,\Of^*_{\bar{\H}})$, on observe $\pic(\bar{\H}_a)=0$ pour $a\in\BC\TC_1$  d'où $\hcech{1}(\bar{\H},\BC\TC_1,\Of^*_{\bar{\H}})\cong \pic(\bar{\H})$. On veut plutôt calculer ce groupe de Picard   sur le recouvrement $\bar{\H}=\bigcup_{s\in \BC\TC_0}\P_s$   
 ce qui revient à établir l'exactitude de la suite :
\begin{equation}\label{eq:suiteex}
0\to\Of^*_{\bar{\H}}\to\prod_{s\in \BC\TC_0}\iota_*\Of^*_{\P_s}\to \prod_{\{s,t\}\in \BC\TC_1}\iota_*\Of^*_{\P_s\cap\P_t}\to 0
\end{equation} avec $\iota$ désignant l'immersion fermée $\P_s\to {\bar{\H}}$ ou encore $\P_s\cap\P_t\to {\bar{\H}}$.

Il suffit de vérifier cette propriété après restriction à chaque ouvert $\bar{\H}_a$. Mais à chaque restriction, on retrouve la suite \eqref{eqexo} dont on a montré l'exactitude. On obtient alors une suite exacte longue d'après \eqref{eq:suiteex} \[\prod_{s\in \BC\TC_0}\Of^*(\P_s)\fln{\beta}{} \prod_{\{s,t\}\in \BC\TC_1}\Of^*(\P_s\cap\P_t)\to \pic(\bar{\H})\to\prod_{s\in \BC\TC_0}\pic(\P_s)\to \prod_{\{s,t\}\in \BC\TC_1}\pic(\P_s\cap\P_t).\] La flèche $\beta$ s'identifie à $\prod_{s\in \BC\TC_0}\bar{\F}^*\fln{\Delta}{} \prod_{\{s,t\}\in \BC\TC_1}\bar{\F}^*$ où \[\Delta((\lambda_s)_{s\in \BC\TC_0})=(\lambda_s^{{\rm Sgn}(s)}\lambda_t^{{\rm Sgn}(t)})_{\{s,t\}\in \BC\TC_1}\in \prod_{\{s,t\}\in \BC\TC_1}\bar{\F}^* \]  avec pour convention ${\rm Sgn}(s)=1$ si $s$ est pair et ${\rm Sgn}(s)=-1$ sinon. Ainsi, $\beta$ est surjective par contractibilité de $\BC\TC$. On a de plus $\pic(\P_s\cap\P_t)=0$ car chaque intersection de composantes est un point fermé.

On en déduit que la flèche de restriction \[\pic(\bar{\H})\iso\prod_{s\in \BC\TC_0}\pic(\P_s)\] est un isomorphisme ce qui prouve le résultat voulu.
\end{proof}

\subsection{Définition et interprétation cohomologique}

Notre but ici est d'étudier des fibrés en droites sur $\H$ ou $\hat{\H}$ muni d'une action d'un sous-groupe $H$ de $G$. Cela nous amène à considérer les objets suivants.

\begin{defi}
Soit un espace $X$ (qui peut être un espace rigide, un schéma ou un schéma formel...) muni d'une action d'un groupe $H$, un fibré en droites $\Lf$ sur $X$ est $H$-équivariant s'il est muni d'une action de $H$ semi-linéaire ie.  pour tout  $g\in H$, $U\subset X$  on a une  fonction  $\rho_g: \Lf(g^{-1}U) \to \Lf(U)$  qui vérifie
 \[
\forall f\in \Of(g^{-1}U), \forall  r\in \Lf(g^{-1}U),  \;\; \rho_g(fr)=  (g\cdot f) \rho_g(r) ,
\]  
\[\forall h\in H, \rho_{gh}=\rho_g\circ g^{-1}\rho_h.\]
\end{defi}

On a plusieurs relations sur les fibrés en droites équivariants.  Deux fibrés $\Lf$ et $\Lf'$ sont équivalents de manière faible s'il existe un isomorphisme de fibrés $\Lf \iso \Lf'$ et de manière forte si cet isomorphisme commute avec les morphismes $(\rho_g)_{g\in H}$ provenant de l'action. On note $\pic_H(X)$ le groupe pour le produit tensoriel des classes de  fibrés en droites $H$-équivariants à équivalence forte.   La proposition suivante précise le lien entre ces notions

\begin{prop}\label{propactforte}
Soit un espace $X$ comme précédemment, on a une suite exacte :  
\begin{equation}\label{eq:actforte}
0\to \hgal{1}(H, \Of^*(X)) \to \pic_H(X) \to \pic(X)^H \to \hgal{2}(H, \Of^*(X)).  
\end{equation}
\end{prop}

 


\begin{proof}
 L'argument est le même que \cite[Proposition 4.2.]{J3}.  La flèche $\pic_H(X)\to \pic(X)^H$ est le morphisme d'oubli de l'action de $H$.  Son noyau classifie les actions de $H$ sur le faisceau trivial $\Of_X$. Si $\Lf$ un  fibré équivariant de ce noyau, il admet un générateur global $v$ en tant que $\Of_{X}$-module.  Pour tout $h\in H$,  $h\cdot v$ est  un autre générateur, donc $h\cdot v= c(h)v$ avec $c(h)\in \Of^*(X)$.  Le fait que $\Lf$ soit muni d'une action de $H$ montre que $c$ est un cocycle dans $\hgal{1}(H, \Of^*(X))$.  Réciproquement, la donnée d'un cocycle $c$ détermine une action $(\rho_h)_{h\in H}$ de $H$ sur $\Of_{X}$ donné par 
\[
\rho_h : f\in \Of_{X}(h^{-1}U) \mapsto c(h)(h\cdot f)    
\]
pour $U\subset X$.  Ainsi le noyau s'identifie au  groupe de cohomologie $\hgal{1}(H, \Of^*(X))$.
 
 Si $\Lf\in \pic(X)^H$, on a une famille d'isomorphismes $(\rho_h:  h^{-1}\Lf\iso \Lf )_{h\in H}$ et on définit une application $c$
\[
(g,h)\in H^2 \mapsto  \rho_{gh}^{-1} \circ (\rho_g\circ g^{-1} (\rho_h) )\in  \aut(\Lf)\cong  \Of^*(X). 
\]
L'application $c$ détermine un cocycle dans $\hgal{2}(H, \Of^*(X))$ qui est trivial si et seulement si la famille $(\rho_h)_{h\in H}$  peut être modifiée   de façon à  définir  une action de $H$ sur $\Lf$.  
\end{proof}

Appliquons ce résultat au cas particulier de demi-plan de Drinfeld :
\begin{coro}\label{corocohoequiv}
Soit $H \subset G$, $(s_0,s_1)=a$ l'arête standard de $\BC\TC_1$, , on a les points suivants :
\begin{itemize}
\item $\pic_H(\H)\cong \hgal{1}(H, \Of^*(\H))$
\item On a des suites exactes
\begin{align}
0 \to \hom([G]_2, \OC_{\breve{K}}^*) \to &\pic_{[G]_2}(\hat{\H}) \fln{(\ord_{s_0}, \ord_{s_1})}{} \Z^2\\
0 \fl \hom(G, \OC_{\breve{K}}^*) \fl &\pic_{G}(\hat{\H}) \fln{\ord_{s_1}}{} \Z
\end{align}
\item On a une décomposition canoniques en produits \[\pic_{gG^\circ g^{-1}}(\P_s)=\hom(gG^\circ g^{-1},\bar{\F}^*)\times \Z\] pour $s=g\cdot s_1\in \BC\TC_0$ un sommet.
\end{itemize}

\end{coro}

\begin{rem}
Le calcul de $\Of^*(\H)$ a fait l'objet de travaux \cite{J2,vdp,gek,gehr} et ce groupe  s'inscrit dans une suite exacte courte 
\[
0\to {\breve{K}}^*\to \Of^*(\H) \to \Z \llbracket \P^1(K)  \rrbracket^0\to 0
\]
où  $\Z \llbracket \P^1(K)  \rrbracket^0$ sont les distributions à valeurs dans $\Z$ sur l'espace profini $\P^1(K)$ de masse  totale nulle.  
\end{rem}

\begin{proof}
L'annulation du groupe de Picard de $\H$ entraîne directement le premier point d'après la suite exacte \eqref{eq:actforte}. Comme les fonctions inversibles sur $\hat{\H}$ sont constantes d'après \cite[Lemme 3.]{ber3} ou \cite[Théorème 7.1.]{J1}, on a $\hgal{1}(H, \Of^*(\hat{\H}))= \homm(H,\OC_{\breve{K}}^*)$. De plus, $\BC\TC^{(1)}_0$ et $\BC\TC^{(0)}_0$ (resp. $\BC\TC_0$) sont les orbites de l'action de $[G]_2$ (resp. $G$) sur $\BC\TC_0$ et le deuxième point découle encore de la suite exacte \eqref{eq:actforte} d'après \ref{theopicentier}. Pour le troisième point, on peut supposer $g=\id$. Le fait que les fonctions inversibles sur l'espace projectif sont constantes ainsi que le calcul de son groupe de Picard entraîne l'exactitude de la suite \[0\to\hom(G^\circ,\bar{\F}^*)\to \pic_{G^\circ}(\P_s)\to \Z \to 0.\]
De plus, chacun des faisceaux $\Of(k)$ peut être muni d'une action  de $G^\circ $ pour laquelle $\Of(k)(\P_s)\cong \sym^k(\bar{\F}^2)$. Cette construction permet de scinder la suite exacte précédente.
\end{proof}


Comme en fibre générique, nous pouvons aussi  donner une interprétation cohomologique aux fibrés $H$-équivariants en niveau entier pour $H= G$ ou $[G]_2$.    Notons $Z^i$ et $B^i$ les cocycles et les cobords de la résolution de Bar calculant les groupes de cohomologie $\hgal{i}(H, \Of^*(\H))$. On introduit $\hat{Z}_1\subset Z_1$ le sous-module de cocycle $c$ tel que 
\[
c(\stab_H(s)) \subset  \Of^*(\hat{\H}_{s}),
\]  
avec $s=s_0$ ou $s_1$ les sommets  standards. On impose de plus $c(w)\in \Of^*(\hat{\H}_{a_0})$ avec $a_0=(s_0,s_1)$ l'arête standard  si $H=G$. 

\begin{prop}
On a un diagramme commutatif dont les lignes horizontales sont exactes 
\[
\begin{tikzcd}
0 \ar[r] & B^1 \ar[r] & Z^1 \ar[r] & {\pic}_{H}(\H) \ar[r] & 0  \\
0 \ar[r] & \hat{B}^1 \ar[r] \ar[u] & \hat{Z}^1 \ar[r]\ar[u] & {\pic}_{H}(\hat{\H}) \ar[r]\ar[u] & 0
\end{tikzcd}
\] 
avec $\hat{B}^1$ l'image de $\Of^*(\H)\cap \Of^*(\hat{\H}_{a_0})$ dans $B^1$ par $u\mapsto gu/u$. 

\end{prop}

\begin{proof}
Nous commencerons par le résultat intermédiaire suivant :

\begin{lem}\label{lemstab}
Soit $H=G$ ou $[G]_2$ et $g,g'\in H$ pour lesquels  $ga_0\cap g' a_0\neq 0$ avec $a_0=(s_0,s_1)$ l'arête standard. Alors pour tout cocycle $c\in \hat{Z}_1$, on a
\[c(\stab_H(a_0))\subset \Of^*(\hat{\H}_{a_0}),\]
\[\frac{c(g)}{c(g')}\in \Of^*(\hat{\H}_{ga_0\cap g' a_0}).\]
\end{lem}

\begin{proof}
Reprenons $g,g',c$ de l'énoncé et écrivons $T=ga_0\cap g' a_0\neq 0$ pour simplifier. Commençons par étudier $c(\stab_H(a_0))$. Si $H=[G]_2$, on a $\stab_H(a_0)=\stab_H(s_0)\cap \stab_H(s_1)$ et on en déduit d'après \ref{proposthsig}  :
\[c(\stab_H(a_0))\subset \Of^*(\hat{\H}_{s_0})\cap \Of^*(\hat{\H}_{s_1})\cap \Of^*(\H)\subset\Of^*(\hat{\H}_{a_0}).\]
Si $H=G$, on a $\stab_G(a_0)=\stab_{[G]_2}(a_0) \amalg w \stab_{[G]_2}(a_0)$. Nous avons déjà traité $\stab_{[G]_2}(a_0)$ et prenons $w h\in w \stab_{[G]_2}(a_0)$. On a $c(w h)=c(w) w\cdot c(h)\in \Of^*(\hat{\H}_{a_0})$ ce qui entraîne l'inclusion  recherchée.

Intéressons-nous à $\frac{c(g)}{c(g')}$ quand $H=[G]_2$. On a   alors ${g'}^{-1}T=s_0$, $s_1$ ou $a_0$ et  $g^{-1}g'\in \stab_{H}  ({g'}^{-1}T)$ d'où $c(\stab_{H}  ({g'}^{-1}T))\subset\Of^*(\hat{\H}_{{g'}^{-1}T})$ par hypothèse.     La condition de cocycle impose 
\[
\frac{c(g)}{c(g')} = \frac{c(g'{g'}^{-1} g) }{c(g')}= g'c({g'}^{-1}g)\in g'\Of^*(\hat{\H}_{{g'}^{-1}T})=\Of^*(\hat{\H}_{T}). 
\]

Passons au cas où $H=G$.   Si $ga_0$ et $g'a_0$ ont la même orientation alors $g^{-1}g' \in \stab_{[G]_2}({g'}^{-1}T)$  et on se ramène au cas précédent.  Dans le cas contraire,  supposons $g\in [G]_2$ quitte à échanger $g$ et $g'$.  On vérifie alors l'identité  $h:={g'}^{-1}g w^{-1}\in \stab_{[G]_2}(g'^{-1}T)$  d'où 
\[
\frac{c(g)}{c(g')}=\frac{c(g'hw)}{c(g')}= (g' c(h))( g'h c(w)) \in \Of^*(\hat{\H}_{T}).  
\]
\end{proof}

Revenons à la preuve du résultat. La commutativité du diagramme ainsi que l'exactitude de la première ligne est claire. Il suffit de prouver l'exactitude de la deuxième suite.   Prenons d'abord $c\in \hat{Z}^1$, cela induit une action sur $\Of_{\H}$ donnée par 
\[
f\in \Of^*(U) \mapsto  c(g) (g\cdot f)   
\]
 sur chaque arête $a=g a_0$,  on définit le réseau $\Lf|_{\hat{\H}_{a}}= \Of_{\hat{\H}_a} c(g)\subset \Of_{\H_{a}}$.    Pour que l'on puisse définir un fibré en niveau entier il faut vérifier les compatibilités suivantes qui découlent de \ref{lemstab} : 
\begin{itemize}

\item Ce réseau ne dépends  pas de  la décomposition $a=g a_0$,  i.e.  $c(g)/c(g')\in \Of^*(\hat{\H}_{a})$ si $a=ga_{0}= g'a_{0}$.  

\item Si deux arêtes $ga_0$, $g'a_0$ s'intersectent le long d'un sommet,  le réseau associé se recolle en l'intersection, i.e.    $c(g)/c(g')\in \Of^*(\hat{\H}_{g a_0 \cap g'a_0})$. 

\end{itemize}


Réciproquement, si $\Lf$ est un fibré équivariant,  on se donne $r$ un générateur commun à $\Lf|_{\hat{\H}_{a_0}}$  et $\Lf[\frac{1}{\varpi}]$, (voir la preuve du proposition \ref{propnorme}).  Si $T\subset a_0$ est un sous-complexe simplicial, et $g\in \stab_{H}(T)$,   $g\cdot r$ est encore en générateur de $\Lf|_{\hat{\H}_T}$ donc $g\cdot r= c(g) r$ avec $c(g)\in \Of^*(\hat{\H}_{T})$.  L'application $g\mapsto c(g)$ est bien un cocycle dans $\hat{Z}^1$.  Le choix de $r$ est unique à multiplication par une fonction inversible $u\in \Of^*(\H)\cap \Of^*(\hat{\H}_{a_0})$ et le cocycle associé à $ur$ est $c(g)(g u /u)$.  D'où l'identification $\pic_{H}(\hat{\H})\cong \hat{Z}^1/\hat{B}^1$.

\end{proof}






\subsection{Exemples \label{ssecex}}

Nous décrivons ici plusieurs classes d'exemples de fibrés $H$-équivariants  en fibre générique et en niveau entier quand $H=G$ ou $[G]_2$. 

$\bullet$ Les objets les plus simples sont les fibrés provenant des caractères. Plus précisément, étant donné un caractère $\chi : H\to {\breve{K}}^*$ (resp. $\chi : H\to \OC_{\breve{K}}^*$), on peut définir sur le faisceau structural de $\H$ (resp. $\hat{\H}$) une action de de donnée par :
\[\rho_h : f\in\Of(h^{-1}U)\mapsto \chi(h)h\cdot f\in\Of(U), \forall h\in H\]
Nous appellerons $\Of(\chi)$ le fibré  obtenu ainsi  et écrirons de même $\Lf(\chi):=\Lf\otimes \Of (\chi)$ pour $\Lf$ $H$-équivariant.

$\bullet$ En fibre générique, un moyen de construire des fibrés dans $\pic_G(\H)$ consiste à restreindre des fibrés équivariants sur la droite projective ambiante. On obtient alors sur $\H$ les fibrés $\Of(k)$ et $\Of(k)(\chi)$. Notons que, sur $\Of(k)$, nous avons imposé l'unique action de $G$ pour laquelle on a un isomorphisme $G$-équivariant $\Of(k)(\P_K^1)\cong \sym^k({\breve{K}}^2)$. Les propriétés cohomologiques de ces objets ont par exemple fait l'objet des travaux \cite{mormur1,mormur2,mormur3,teit4,scst,orl2}. Nous nous proposons dans cet article d'étudier les possibles modèles entiers sur de ces faisceaux avec l'espoir de compléter les résultats de \cite{teit4,GK6} et de leur donner des intuitions plus géométriques et plus modulaires.

$\bullet$ Pour les faisceaux $\Of(2k)$, il est relativement aisé de construire des modèles entiers. En effet, en fibre générique, le faisceau $\Of(-2)$ s'identifie à $\Omega^1_{\H/{\breve{K}}}$ qui admet le sous-réseau $\Omega^1_{\hat{\H}/\OC_{\breve{K}}}(d\log)$ des formes différentielles à pôles logarithmiques. Plus précisément, sur une arête $\hat{\H}_a$ dont les sections sont de la forme \[\Of(\hat{\H}_a)=\OC_{\breve{K}}\langle  X,Y, \frac{1}{X^{q-1}-1} , \frac{1}{Y^{q-1}-1} \rangle/ (XY-\varpi),\] on a \[\Omega^1(d\log)(\hat{\H}_a)=\Of(\hat{\H}_a)dX/X=\Of(\hat{\H}_a)(-dY/Y).\] Ce dernier est stable sous l'action de $G$ et fourni donc un objet de $\pic_G (\hat{\H})$ qui en fibre générique devient $\Of(2)$ par construction. Pour les fibrés de poids $2k$, il suffit de considérer les puissances tensorielles de $\Omega^1_{\hat{\H}/\OC_{\breve{K}}}(d\log)^{\otimes -k}$.

$\bullet$ La tâche est relativement plus subtile en poids impairs. Nous allons par exemple prouver  qu'il n'existe pas de modèle entier $G$-équivariant pour de tels fibrés. Par contre, si l'on s'intéresse plutôt aux faisceaux $[G]_2$-équivariants, nous avons alors une famille dénombrable de modèle pour chaque $\Of(2k+1)$. Utilisons  l'interprétation modulaire de $\hat{\H}$ pour exhiber deux exemples fondamentaux dans le cas de $\Of(-1)$. Rappelons que $\hat{\H}$ admet un module formel spécial universel $\XG$ dont le plan tangent se décompose en somme de deux fibrés en droite
\[\lie(\XG)=\lie(\XG)_0\oplus\lie(\XG)_1\]
suivant les deux plongements de $\OC_{{(2)}}$ dans $\OC_{\breve{K}}$. Chacune de ces parties isotypiques fournissent un objet de $\pic_{[G]_2}(\hat{\H})$ et ces dernières sont échangées par l'action la matrice $w= \left( \begin{matrix}
0 & 1\\ \varpi & 0
\end{matrix} \right)$. De plus, l'action de l'uniformisante $\Pi_D\in \OC_D$ sur $\XG$ induit des morphismes de fibrés $\Pi :\lie(\XG)_i\to\lie(\XG)_{i+1}$ (pour $\in\Z/2\Z$) dont la composée est la multiplication par $\varpi$. Nous en déduisons que ces deux fibrés sont alors isomorphes en fibre générique ce qui n'est pas le cas en niveau entier. Enfin, nous allons montrer que les fibrés duaux $\omega_0$, $\omega_1$ sont des modèles entiers pour le faisceau $\Of(-1)$ sur $\H$.

$\bullet$ L'interprétation modulaire de $\hat{\H}$ fournit deux autre exemples de fibrés équivariants qui proviennent en quelque sorte du premier revêtement de la tour de Drinfeld $\Sigma^1/\H$. Plus précisément, les points de $\Pi_D$-torsion $\XG[\Pi_D]$ du module formel universel $\XG/\hat{\H}$ est fini plat sur $\hat{\H}$ et admet une action du groupe cyclique $(\OC_D/\Pi_D)^*\cong\F_{q^2}=: \F^*_{(2)}$ qui se transporte à l'idéal d'augmentation $\IC$ de $\XG[\Pi_D]$. Ce dernier se décompose alors suivant les caractères de $\F^*_{(2)}$. Chacune de ces parties isotypiques sont des éléments de $\pic_{[G]_2}(\hat{\H})$ et nous noterons $\Lf_0$, $\Lf_1$ les deux fibrés associés aux caractères de $\F^*_{(2)}$ qui se prolongent en des applications $\F$-linéaires  $\F_{(2)}\to \bar{\F}$. Pour les autres caractères, les fibrés peuvent s'écrire de manière unique sous la forme $\Lf_0^a\otimes \Lf_1^b$ avec $0\le a,b\le q-1$(voir  \cite[Proposition 1.3.1]{Rayn}). En fibre générique, ils sont tous les deux de torsion  d'ordre $q+1$ dans $\pic_{[G]_2} (\H)$ et correspondent à deux parties isotypiques $(\pi_*\Of_{\Sigma^1})[\chi]$ du poussé-en-avant du faisceau structural du premier revêtement $\Of_{\Sigma^1}$.

\begin{rem}
Il peut paraître surprenant que les exemples de fibrés  $\omega_0$, $\omega_1$, $\Lf_0$ et $\Lf_1$  sont dans $\pic_{[G]_2}(\hat{\H})$ et non  dans $\pic_{G}(\hat{\H})$. Pour expliquer ce phénomène, il faut décrire un peu plus précisément le problème modulaire représenté par $\hat{\H}$ et les symétries qu'il possède. Il s'agit du  foncteur $\GC$ qui, à une algèbre $A\in \mathrm{Nilp}$, associe l'ensemble des triplets $(X,\rho,\iota)$ avec $X$ un $\OC_K$-module  formel sur $A$, $\rho$ une isogénie $\Phi_{\bar{\F}} \otimes A/ \varpi A \to X_{A/ \varpi A}$ de hauteur $0$ (cf \ref{ssecmod} pour la définition de $\Phi_{\bar{\F}}$) et une action de  $\OC_D$ sur  $X$ qui en fait un module spécial. De manière plus générale, on peut aussi étudier le foncteur $\tilde{\GC}$ qui classifie les mêmes triplets $(X,\rho,\iota)$ mais où $\rho $ peut être de hauteur quelconque. Ce dernier est représentable par $\hat{\H}\times \Z$ et on a des actions naturelles qui commutent du groupe des isogénies $G$ de $\Phi_{\bar{\F}}$ sur le terme $\rho$ ainsi que celle de $D^*$ sur le terme $\iota$. De plus, les éléments $g\in G$ (resp. $\Pi_D^n b\in D^* =\Pi_D^\Z\times\OC_D^*$) envoient $\hat{\H}\times\{i\}\subset \hat{\H}\times\Z$ sur $\hat{\H}\times\{i+v(\det(g))\}$ (resp. $\hat{\H}\times\{i+n\}$). Si l'on choisit une identification non-canonique $\hat{\H}\cong\hat{\H}\times\{0\}$, l'action usuelle d'un élément $g\in G$ sur $\hat{\H}$ correspond à la restriction de celle de $g/\Pi_D^{v(\det(g))}\in G\times D^*$ sur $\hat{\H}\times\{0\}$ (qui est laissée stable par $g/\Pi_D^{v(\det(g))}$). Mais, $g/\Pi_D^{v(\det(g))}$ commute avec $D^*$ si et seulement si $v(\det(g))$ est pair. Ainsi, les éléments de $[G]_2$ préservent les parties isotypiques $\omega_0$, $\omega_1$ (resp. $\Lf_0$, $\Lf_1$), et ceux dans le complémentaire les échangent.
\end{rem}

\begin{defi}
Nous noterons $\pic_{(mod)}(\hat{\H})$ le sous-groupe de $\pic_{[G]_2}(\hat{\H})$ engendré par les fibrés $\Of(\chi)$, $\omega_0$, $\omega_1$, $\Lf_0$ et $\Lf_1$. Les éléments de $\pic_{(mod)}(\hat{\H})$ seront appelés  fibrés modulaires.
\end{defi}

L'un des résultats principaux de cet article consistera à montrer que tous les fibrés $[G]_2$-équivariants sont modulaires.
\subsection{Classifications}

Dans cette section, nous allons énoncer le théorème de classification de $\pic_{[G]_2}(\hat{\H})$ et expliquer les grandes lignes de la preuve. Pour donner un sens précis à l'énoncé, nous aurons besoin de la notion de poids et de type d'un fibré.


\begin{defi}
Soit $\Lf\in \pic(\hat{\H})^{[G]_2}$ et $i,j\in \Z/2\Z$, on appelle le poids $w(\Lf)$  de $\Lf$  l'élément de $\frac{1}{q-1}\Z$ 
\[
w(\Lf)= - \frac{1}{q-1}(\ord_{s_0}(\Lf)+\ord_{s_1}(\Lf))
\]
et le type $t_{i,j}(\Lf)$ 
\[
t_{i,j}(\Lf)= \ord_{s_j}(\Lf)+\ord_{s_j}(\omega_i)w(\Lf). 
\]
\end{defi}

\begin{rem}

Le poids et le type se prolongent naturellement aux fibrés sur $\bar{\H}$ et à ceux sur $\H$ pour lesquels on a fixé une structure entière (\ref{corogen}).

\end{rem}

Le théorème de classification est le suivant

\begin{theo}
Les fibrés $[G]_2$-équivariants sont  les fibrés modulaires et sont aussi les fibrés de $\pic(\hat{\H})^{[G]_2}$ de poids  entier.

Les fibrés $G$-équivariants sont  de la forme $\Omega^1(\log)^k$ à un caractère de $G$ près et sont aussi les fibrés  de $\pic(\hat{\H})^G\subset\pic(\hat{\H})^{[G]_2}$ de poids  pair. 

En particulier, le fibré $\Of(1)$ sur $\H$ n'admet pas de structure entière $G$-équivariante.
\end{theo}

\begin{rem}

Pour pallier au fait que les fibrés de poids impaires $\Of(2k+1)$ n'ont pas de structure entière $G$-équivariante, on peut procéder de la manière suivante. Le $\OC_D$-module formel $\XG$ universel sur $\hat{\H}$ fournit une flèche $\Pi: \omega_i\to\omega_{i+1}$ et on peut considérer que l'on peut voir $\Ff:=(\omega_0\oplus\omega_1)/(\Pi(\omega_0\oplus\omega_1))\otimes\OC_{\breve{K}[\sqrt{\varpi}]}$ comme un réseau de $\Of(-1)\otimes\OC_{\breve{K}[\sqrt{\varpi}]}$ stable sous l'action de $[G]_2$. Pour obtenir un fibré $G$-équivariant sur $\hat{\H}\otimes\OC_{\breve{K}[\sqrt{\varpi}]}$, il suffit d'imposer que l'élément  $w= \left( \begin{matrix}
0 & 1\\ \varpi & 0
\end{matrix} \right)\in G\backslash [G]_2$ agisse par $\Pi/\sqrt{\varpi}$. On prouve aisément que ce procédé fonctionne et qu'il donne la structure entière de $\Of(-1)\otimes\OC_{\breve{K}[\sqrt{\varpi}]}$ considérée dans \cite{GK6}. 

Comme les morphismes $\Pi$ sont bijectifs en dehors des points singuliers de la fibre spéciale  $\bigcup_{s\in \BC\TC_0} \P_s(\F)$ (voir \ref{theoOrdOmegai}), on a des isomorphismes $G$-équivariants $\Ff|_U\cong \omega_0|_U\cong\omega_1|_U$ sur l'ouvert complémentaire de $\bigcup_{s\in \BC\TC_0} \P_s(\F)$. En revanche, toujours d'après \ref{theoOrdOmegai}, la tige de $\Ff$ en les points singuliers s'identifie à $\bar{\F}^2$.

\end{rem}

La preuve  occupera les deux sections suivantes  où l'on établira le résultat clé (voir \ref{coroOrdreOmega}, \ref{theoordli} et \ref{propordreomegalog}) :

\begin{lem}\label{lemord}
On a :
\begin{enumerate}

\item $\ord_{s_0}(\omega_0)=- 1 $ et    $\ord_{s_1}(\omega_0)=  q$. 
 
\item  $\ord_{s_0}(\omega_1)= q$  et $\ord_{s_1}(\omega_1)= -1$.

\item $\ord_{s_0}(\Lf_0)= 1 $ et $\ord_{s_1}(\Lf_0)=-1 $

\item $\ord_{s_0}(\Lf_1)= -1$ et $ \ord_{s_1}(\Lf_1)=1 $

\item  $\ord_{s_0}(\Omega^1(\log))= \ord_{s_1}(\Omega^1(\log))=q-1 $
\end{enumerate}
\end{lem}

De ce résultat, nous pouvons en déduire

\begin{coro}
Les fibrés modulaires sont exactement ceux dont le poids est entier et  se décomposent de manière unique sous la forme  
\[
\Lf=\omega_i^{\otimes -w(\Lf)}\otimes \Lf_j^{\otimes t_{i,j}(\Lf)}(\chi) 
\] avec $\chi\in \homm([G]_2,\OC_{\breve{K}})$.
\end{coro}

\begin{proof}
D'après le calcul des ordres précédents, on voit que chaque générateur de $\pic_{(mod)}(\hat{\H})$ a un poids entier et que les fibrés de poids  nul sont les puissances  de $\Lf_j$ à un caractère près. Comme $\omega_i$ a poids $-1$, tout fibré de poids entier s'écrit (à un caractère près) comme une combinaison de $\omega_0$ et de $\Lf_0$. Une analyse rapide montre que le poids et le type sont les coefficients de cette décomposition.
\end{proof}


Comme on a l'inclusion $\pic_{(mod)}(\hat{\H})\subset \pic_{[G]_2}(\hat{\H})$, le théorème de classification revient à prouver :

\begin{prop}\label{proppoids}

Tout fibré $[G]_2$-équivariant sur $\hat{\H}$ ou sur  $\bar{\H}$ a un poids entier et tout fibré $G$-équivariant est de poids pair.

\end{prop}

\begin{proof}
Prenons un fibré $[G]_2$-équivariant $\Lf$ sur $\hat{\H}$, il suffit de raisonner sur $\bar{\Lf}$ et travailler en fibre spéciale. Quitte à tensoriser par des puissances de $\omega_0|_{\bar{\H}}$ et de $\Lf_0|_{\bar{\H}}$ et $\delta_a$ voir \eqref{eq:del}, on peut supposer  que les ordres $k_0=\ord_{s_0}(\Lf)$, $k_1=\ord_{s_1}(\Lf)$ de $\Lf$ sont positifs \ref{lemord} et que $\varpi\id$ agisse trivialement. L'idée de la preuve consiste à voir que la donnée de ces ordres détermine à un caractère près l'action de $\stab_{[G]_2}(s_i)=w^{1-i}G^{\circ}w^{i-1} \varpi^\Z$ sur $\bar{\Lf}(\P_{s_i})$ d'après \ref{corocohoequiv} (ici, $w$ est la matrice $\left( \begin{matrix}
0 & 1\\ \varpi & 0
\end{matrix} \right)$ introduite précédemment). Par restriction, cela fournit deux actions de l'Iwahori( + le centre) $\stab_{[G]_2}(s_0)\cap\stab_{[G]_2}(s_1)$ sur $\bar{\Lf}(\P_{s_0}\cap\P_{s_1})$ qui doivent coïncider. Cette condition va entraîner le résultat.

Plus précisément, on a des surjections $w^{1-i}G^{\circ}w^{i-1} \varpi^\Z\fln{c_{w^{i-1}}}{}G^{\circ}\to\bar{G}$ pour $i=0,1$ (avec $\bar{G}=\gln_2(\F)$ et $c_{w^{i-1}}$ la flèche qui envoie $\varpi^\Z$ sur $1$ puis conjugue par $w^{i-1}$) et des isomorphismes $\stab_{[G]_2}(s_i)$-équivariants par inflation :
\[\bar{\Lf}(\P_{s_i})\cong \sym^{k_i}(\bar{\F}^2)\otimes {\det}^{\otimes r_i}\]

Pour décrire explicitement les  actions sur  $\bar{\Lf}(\P_{s_0}\cap\P_{s_1})$, nous aurons besoin de considérer le  système de coordonnées homogènes usuel  $[z_0^{(1)}:z_1^{(1)}]$ sur $\P_{\breve{K}}^1$ qui définit encore un système de coordonnées pour $\P_{s_1}$ en fibre spéciale. Le même raisonnement montre que $[z_0^{(0)},z_1^{(0)}]:=w\cdot [z_0^{(1)},z_1^{(1)}]$ fournit  un système similaire sur la composante $\P_{s_0}$. Par construction de la flèche de réduction vers l'immeuble de Bruhat-Tits, on voit que $V^+(z_0^{(1)})=\P_{s_0}\cap\P_{s_1}=V^+(z_0^{(0)})$. On peut alors voir que les sections $\bar{\Lf}(\P_{s_i})$ s'identifient aux polynômes homogènes en $(z_0^{(i)},z_1^{(i)})$ de degré $k_i$ et l'on notera $(z_0^{(i)})\subset \sym^{k_i}(\bar{\F}^2)\otimes {\det}^{\otimes r_i}$ le sous-espace  des polynômes divisibles par $z_0^{(i)}$ (qui est stable sous l'action de l'Iwahori $I$). Par la description précédente de $\P_{s_0}\cap\P_{s_1}$, on a des isomorphismes $I$-équivariant : 
\[\bar{\Lf}(\P_{s_0}\cap\P_{s_1})\cong\sym^{k_i}(\bar{\F}^2)\otimes {\det}^{\otimes r_i}/(z_0^{(i)})\cong\bar{\F} \cdot (z_1^{(i)})\]
Mais on a $\left( \begin{matrix}
a & 0\\ 0 & d
\end{matrix} \right)\cdot (z_1^{(1)})=a^{r_1}d^{r_1+k_1}z_1^{(1)}$ et $\left( \begin{matrix}
a & 0\\ 0 & d
\end{matrix} \right)\cdot (z_1^{(0)})=w\left( \begin{matrix}
d & 0\\ 0 & a
\end{matrix} \right)w^{-1}\cdot(z_1^{(0)})=a^{r_0+k_0}d^{r_0}z_1^{(0)}$ pour tout $a$, $d$ dans $\F$. L'isomorphisme de représentation précédent entraîne les relations $a^{r_1}d^{r_1+k_1}=a^{r_0+k_0}d^{r_0}$ pour tout $a$, $d$ d'où  :
\begin{equation}\label{eq:r0r1}
r_0+k_0\equiv r_1 \pmod{q-1}\et r_0\equiv r_1+k_1 \pmod{q-1}
\end{equation}
En résolvant le système précédent, on obtient : \[k_0+k_1\equiv 0 \pmod{q-1}\]
ce qui prouve le résultat pour $[G]_2$. Si $\bar{\Lf}$ est de plus $G$-équivariant, on a en plus de la relation \eqref{eq:r0r1} l'égalité $r_0=r_1$ ce qui entraîne \[k_0\equiv k_1\equiv 0 \pmod{q-1}\] et le résultat



\end{proof}

Nous terminons cette section par quelques conséquences du théorème de classification. La première permet de décrire les fibrés en fibre générique de $\pic_{[G]_2}({\H})$ qui admettent un modèle entier. Ce résultat  utilise essentiellement le faite que $\omega_1[1/p]\cong \Of(1)$ (voir \cite[p656]{teit2} ou \ref{RemPullbackPeriodes}) :
\begin{coro}\label{corogen}
Si $\Lf\in\pic_{[G]_2}({\H})$ provient d'un élément de $\pic_{[G]_2}(\hat{\H})$ alors on peut trouver des entiers $k,t$ tels que 
\[
\Lf\cong\Lf_0[\frac{1}{p}]^{\otimes t}(k)(\chi).
\] 
avec $\Lf(k):=\Lf\otimes \Of(k)$. L'entier $k$ est unique et $t$ l'est seulement modulo $q+1$. Ces entiers prolongent la notion de poids et de type aux fibrés sur $\H$ qui admettent une structure entière\footnote{Notons que $k$ et $t$ ne dépendent pas de la structure entière  choisie.}.
\end{coro}%

Nous avons aussi une classification similaire en fibre spéciale qui n'est pas très surprenante au vu de \ref{theopicentier} et \ref{propactforte}.

\begin{coro}\label{coromodspe}

La flèche $\Lf\mapsto\bar{\Lf}=\Lf/\varpi \Lf$ met en bijection les fibrés de $\pic_{[G]_2}(\hat{\H})$ à torsion par un caractère de $[G]_2$ à valeurs dans $1+\varpi\OC_{\breve{K}}$ près et les fibrés de $\pic_{[G]_2}(\bar{\H})$.

En particulier, un élément $\bar{\Lf}$ de $\pic_{[G]_2}(\bar{\H})$ est équivalent à la donnée de  $(a,r,k_1,k_2)\in \bar{\F}^*\times\Z/(q-1)Z\times \Z^2$ tels que $q-1|k_1+k_2$ ou encore de $(a,r,k,w)=(a,r,k_1,k_1+k_2)$. Ces éléments  sont caractérisés par $k_i=\ord_{s_i}(\bar{\Lf})$ et $\bar{\Lf}|_{\P_{s_1}}\cong_{G^\circ\varpi^\Z}\delta_a\otimes {\det}^r\otimes\Of(k_1)  $.

\end{coro}

\begin{proof}
La première partie est claire d'après \ref{lemord} et \ref{proppoids} que l'on peut encore appliquer en fibre spéciale. De plus, les ordres en les composantes $\P_{s_0}$ et $\P_{s_1}$ déterminent un fibré sur $\bar{\H}$ à un caractère $\chi : [G]_2\to \bar{\F}$ près. On observe aisément que le choix de $a$ et $r$ pour lequel $\bar{\Lf}|_{\P_{s_1}}\cong_{G^\circ\varpi^\Z} {\det}^r\otimes\Of(k_1)  $ est équivalent au choix de $\chi$. Observons aussi que, d'après la preuve de \ref{proppoids}, on a $\bar{\Lf}|_{\P_{s_0}}\cong \delta_a\otimes{\det}^{r+k_1}\otimes\Of(k_0)$.
\end{proof}

On sait aussi que les formes différentielles à pôle logarithmiques $\Omega^1(\log)$ est un fibré modulaire et le résultat suivant précise cette affirmation :

\begin{coro}[Isomorphisme de Kodaira-Spencer]
On a un isomorphisme $G$-équivariant $\widetilde{KS}: \omega_0\otimes \omega_1 \to \Omega^1(\log)$.  
\end{coro}

La démonstration de ce fait pourra être trouvé dans \ref{theoks}.

\begin{rem}
À première vue, le fibré $\omega_0\otimes \omega_1$ est uniquement $[G]_2$-équivariant mais on peut prolonger l'action à $G$ tout entier car tout élément dans le complémentaire $G\setminus [G]_2$ échange les fibrés $\omega_0$ et $\omega_1$.
\end{rem}
%






\section{Théorie de Cartier}

Pour prouver le résultat \ref{lemord}, nous aurons besoin de mieux comprendre l'interprétation  modulaire $\GC $ de l'espace symétrique $\hat{\H}$ ainsi que le théorème de représentatibilité \ref{Drrep}. Plus précisément, ce théorème induit  une bijection entre les points géométriques en fibre spéciale $\pi:\GC(\bar{\F})\iso\bar{\H}(\bar{\F})$ que nous voulons rendre explicite. Cela nécessitera d'avoir un moyen de classifier les modules formels, possible grâce à la théorie de Cartier sur $\bar{\F}$ dont nous allons rappeler les résultats principaux dans les trois sous-sections §\ref{ssecwitt}, §\ref{sseccart}, §\ref{ssecisog}  qui vont suivre. Dans le cas qui nous intéresse, à savoir lorsque l'on travaille sur $\bar{\F}$, les modules formels spéciaux décrits par $\GC(\bar{\F})$ sont alors complètement caractérisés par la donnée d'un module de Cartier  $M$ qui est un $ \OC_{\breve{K}}$-module libre de rang $4$  muni d'une $\Z/2\Z$-graduation et d'opérateurs $Pi$, $V$, $F$ homogènes de degré $1$ respectivement $\bar{\F}$-linéaire, $\sigma^{-1}$-linéaire et $\sigma$-linéaire. Décrire la flèche $\pi$ revient alors à résoudre un problème d'algèbre semi-linéaire relativement accessible  réalisé dans \ref{theopsi}. La construction de cette bijection repose sur des éléments de la preuve du résultat classique suivant dont nous rappellerons les différents arguments :
\begin{lem}
\label{propClassIsog}
Il existe une unique classe d'isogénie de $\OC_D$-module formel spécial sur $\bar{\F}$  de dimension $2$ et hauteur $4$.   
\end{lem}

D'après ce qui précède, la donnée d'un point géométrique $y\in \bar{\H}(\bar{\F})$ fournit un module de Cartier $M_y$. Nous pouvons aussi rendre explicite l'inverse de la flèche $\pi:\GC(\bar{\F})\iso\bar{\H}(\bar{\F})$ et ainsi donner une base à $M_y$ qui respecte la graduation et sur laquelle on peut écrire l'action des différents opérateurs  $\Pi$, $V$, $F$. Pour rendre plus simple l'énoncé de ce résultat, nous pouvons supposer quitte à translater $y$ par $\gln_2 (K)$ que $y\in \P_s(\bar{\F})\setminus \{\infty\}$ avec $s$ le sommet standard de $\BC\TC_0$ et $\infty$ un point $\bar{\F}$-rationnel de $\P_s$ fixé. Le théorème de classification est alors le suivant :
\begin{lem}
Soit $y\in \P_s(\bar{\F})\backslash \{\infty\}\cong \bar{\F}$, $[y]$ un relevé dans $\OC_{\breve{K}}$  et $M_y=M_{y,0}\oplus M_{y,1}$ le module de Cartier associé muni de sa $\Z/2\Z$-graduation.   On peut trouver une  $\OC_{\breve{K}}$-base    $\{x_{0,i},  x_{1,i} \}_{i\in\Z/2\Z}$ de $M_{y,i}$ tel que 
\[
\begin{array}{l}
\Pi x_{0,1} =F x_{0,1}= V x_{0,1} = \varpi x_{0,0}-[y] x_{1,0}, \\ 
\Pi x_{1,1}=F x_{1,1}= V x_{1,1}= x_{1,0}, \\ 
\Pi x_{1,0}= F x_{1,0} = V x_{1,0} = \varpi x_{1,1}, \\ 
\Pi x_{0 ,0}=  x_{0,1}+[y] x_{1,1}, \\
F x_{0 ,0 }=  x_{0,1}+[y^q] x_{1,1},   \\
V x_{0 ,0 }=  x_{0,1}+[y^{1/q}] x_{1,1}.
\end{array}
\]
\end{lem}

Dans la section suivante §\ref{seccalcordre}, nous utiliserons la description précédente de $\pi^{-1}$ pour étudier les deux flèches 
\[\Pi :M_{y,i}/VM_{y,i+1}\to M_{y,i+1}/VM_{y,i}\]
\[F :M_{y,i}/VM_{y,i+1}\to (M_{y,i+1}/VM_{y,i})^{(q)}\]
et déterminer sur quels fermés Zariski de $\bar{\H}(\bar{\F})$ elles s'annulent et à quel ordre (grâce  théorie de Dieudonné-Messing.). Cela nous permettra de calculer les  quantités $\ord_{s_i}(\omega_{i})-\ord_{s_{i+1}}(\omega_{i})$ et $\ord_{s_{i+1}}(\omega_{i})-\ord_{s_i}(\omega_{i})$ et déterminera entièrement les ordres de fibrés $(\omega_i)_{i=0,1}$ (ceux des fibrés $(\Lf_i)_{i=0,1}$ en  découlera aisément)  prouvant ainsi le lemme \ref{lemord}.

\subsection{Vecteurs de Witt \label{ssecwitt}}

Il existe un unique (voir par exemple \cite[§2.1]{boca}) endofoncteur $W_{\OC_K}$ de la catégorie de $\OC_{K}$-algèbres  tel que pour toute $\OC_K$-algèbre  $A$,  $W_{\OC_K}(A) = A^{\N}$ en tant qu'ensemble et tel que les flèches 
\[
w_n : (a_i)\in A^{\N} \mapsto  a_0^{q^n}+\varpi a_1^{q^{n-1}}+\cdots + \varpi^{n} a_n \in A 
\] 
définissent des morphismes d'anneaux $W_{\OC_K}(A)\to A$.  

\begin{ex}
$W_{\OC_K}(  \F )= \OC_{K}$,    $W_{\OC_K}(\bar{\F})= \OC_{\breve{K}}$.   Si de plus,  $A$ est une $K$-algèbre,   on a un isomorphisme de $\OC_{K}$-algèbre $W_{\OC_{K}}(A)\fln{\prod w_n}{}   A^{\N}$.  
\end{ex}

On a deux endomorphismes $\tau$ et $\sigma$ sur $W_{\OC_K}$ définis par
\[
\tau(a_0,a_1,\ldots)=(0, a_0, a_1,\ldots) \et w_n\sigma=w_{n+1}
\] 
ainsi qu'un morphisme multiplicatif  $[-]: A \to W_{\OC_K}(A)$ donné par $[a]=(a,0,0,\ldots)$.  Ces morphismes vérifient les relations 
\begin{align*}
 \sigma \tau &= \varpi \\ 
 \tau(a) b &= \tau(a \sigma(b))  \\
 \sigma([a])&=[a^{q} ]
\end{align*}
Si $A$ est un $\F$-algèbre, on a de plus 
\[
\sigma(a_0,a_1,\ldots )= (a_0^q, a_1^q, \ldots) \et  \sigma\tau = \tau\sigma = \varpi. 
\]

Rappelons que l'on a noté $K_{(2)}$  l'extension non ramifiée de $K$ de degré $2$ d'anneau $\OC_{(2)}$. Fixons $\zeta \in K_{(2)}$ une racine primitive $(q^{2}-1)$-i\'eme de l'unité de tel manière que $K_{(2)}= K[\zeta]$.  
\begin{rem}
\label{remK2act}
 Si $A$  est une $\OC_{{(2)}}$-algèbre,  $W_{\OC_K}(A)$ admet une structure de $\OC_{{(2)}}$-algèbre $ \chi_0:  \OC_{{(2)}} \to  W(A)$ avec $\chi_0(\zeta)= [\zeta]$.      On appelle $\chi_1 = \chi_0 \circ \sigma$ avec $\sigma: \OC_{{(2)}} \to \OC_{{(2)}}$ l'unique automorphisme non trivial de $\gal(K_{(2)}/K)$.   
\end{rem}

\subsection{Anneau de Cartier et classification des modules formels \label{sseccart}}

Considérons pour toute $\OC_K$-algèbre $A$, l'algèbre non-commutative $W_{\OC_K}(A)\{F;V\}$ où les éléments $F$ et $V$ sont soumis aux relations 
\begin{align*}
F a & = \sigma(a)F\\
V \sigma(a) & = a V\\
V x F & = \tau(x) \\
F V & =\varpi.  
\end{align*}
L'anneau de Cartier noté $E_{\OC_K}(A)$ sera la complétion de $W_{\OC_K}(A)\{F;V\}$ pour la topologie $V$-adique,  i.e  la topologie dont une base des voisinages est donnée pour les idéaux à droites engendrés par $V^n$.

Un ($\OC_K$-)module de Cartier $M$ sur $A$ est un $E_{\OC_K}(A)$-module à gauche tel que 
\begin{itemize}
\item[•] $M$ est séparé complet pour la filtration $V$-adique,  

\item[•]  $V$ est injectif sur $M$,

\item[•] $M/VM$ est libre de rang fini sur $A$. 

\end{itemize}

Le dual d'un module de Cartier $M$ est le $W_{\OC_K}(A)$-module libre $M^{\vee}= \mathrm{Hom}_{W_{\OC_K}(A)}(M,  W_{\OC_K}(A))$ muni des applications $F$ et $V$ données par $F_{M^{\vee}}=V_{M}^{\vee}$ et $V_{M^{\vee}}=F_{M}^{\vee}$.   

\begin{theo}
La catégorie de $\OC_{K}$-modules formels sur $A$ est équivalente à la catégorie de modules de Cartier sur $A$. 

De plus, si  $\XG$ est un $\OC_K$-module formel sur $A$  de module de Cartier $M(\XG)$, on a une identification $\lie \XG= M(\XG)/ V M(\XG)$   et $M(\XG)^{\vee} = M(\XG^{\vee})$  .   
\end{theo}



Nous aimerions énoncer un analogue de cette équivalence pour les  $\OC_D$-modules formels.  Rappelons que $\OC_{D}= \OC_{{(2)}}[\Pi_D]$    où $\Pi_D$ vérifie les relations $\Pi_D^2= \varpi$ et $\Pi_{D} a = \sigma(a) \Pi_D$ avec $a\in \OC_{{(2)}}$ et $\sigma$ l'unique automorphisme non trivial de $\gal(K_{(2)}/K)$. De plus, on a $\OC_{{(2)}}=\OC_{K}[\zeta]$ avec $\zeta \in K$ une racine primitive $(q^{2}-1)$-i\'eme de l'unité. Expliquons comment prolonger l'action de $\OC_{K}$ à $\OC_{{(2)}}$ puis à $\OC_{D}$ sur un module de Cartier $M$. 

   Si $A$ est une $\OC_{{(2)}}$-algèbre et $\XG$ un $\OC_D$-module formel sur $A$, l'action de $\zeta$ et de $\Pi_D$ sur le $\OC_K$-module formel sous-jacent induisent des $E_{\OC_K}(A)$-endomorphismes $\zeta$, $\Pi$ sur le module de Cartier $M(\XG)$ qui déterminent complètement la structure de  $\OC_D$-module sur $X$.

Réciproquement, si $A$ est une $\OC_{{(2)}}$-algèbre,  $E_{\OC_{K}}(A)$ admet deux plongements de $\OC_{{(2)}}$-algèbre d'après la remarque \ref{remK2act} et tout module $M$ sur $E_{\OC_K}(A)$ admettant une action $j: \OC_{{(2)}} \to  \en_{E_{\OC_{K}}(A)}(M)$  se décompose comme $M=M_0\oplus M_1$  en tant que $W_{\OC_{K}}(A)$-module avec 
\[
M_i=\{m\in M : \forall b\in \OC_{{(2)}},  j(b)m=\chi_i(b)m \}. 
\]
Comme $F$ et $V$  sont  semi-linéaires,    ils envoient $M_i$ dans $M_{i+1}$.   

 Si de plus $M$ admet un endomorphisme $\Pi$ tel que $\Pi^2= \varpi$  et $  \Pi (j(b))=j(\sigma(b))\Pi$  pour $b\in \OC_{{(2)}}$,  on a des morphismes induits 
 \[
 \Pi_i : M_i \to M_{i+1}.
 \]
La donnée de la décomposition $M=M_0\oplus M_1$ ainsi que des morphismes $\Pi_i$ est équivalente à \begin{itemize}
\item une $\Z/2\Z$-graduation sur $E_{\OC_K}(A)[\Pi]:=E_{\OC_K}(A)[X]/(X^2-\varpi)$,
\item les éléments de $W_{\OC_K}(A)$ sont homogènes de degré $0$,
\item les opérateurs $F,V,\Pi$ sont de degré $1$,
\end{itemize} et nous appellerons ($\OC_{K}[\Pi]$-)module de Cartier gradué un module de Cartier muni d'une telle graduation.  On en déduit l'énoncé

 \begin{theo}
 Si $A$ est une $\OC_{{(2)}}$-algèbre, on a une équivalence de catégories  entre les $\OC_{D}$-modules formels  sur $A$  et les modules de Cartier gradués sur $A$.  
 \end{theo}

Nous dirons qu'un module de Cartier gradué  $M$ est spécial si $M_i/ VM_{i+1}$ est un $A$-module libre de rang $1$  pour $ i=0,1$.  L'équivalence de catégories ci-dessus  fait correspondre les $\OC_{{(2)}}$-modules formels spéciaux aux  modules de Cartier gradués spéciaux.

\subsection{Isogénies et groupes $p$-divisibles \label{ssecisog}}

\begin{defi}
Un morphisme $\varphi:  \XG \to  \YG$    de $\OC_K$-modules  formels est une isogénie si  $\ker \phi$ est un groupe fini et plat.   
\end{defi}

\begin{defi}[Hauteur]
\label{defiHauteur}
Soit $A$ une  $\F_p$-algèbre  de type finie,    $\pG \subset A$ un idéal premier et $\varphi:  \XG \to  \YG$ une isogénie de groupes formels sur $A$, la hauteur $\mathrm{ht}(\pG)$ de $\varphi$ en $\pG$  est  $\rg_{k(\pG)}( \ker \varphi \otimes k(\pG) )$.  
 On dit que $\varphi$ est de hauteur $h$ si $\mathrm{ht}(\pG)=h$ pour tout $\pG\subset A$ premier.  
\end{defi} 

On a le point suivant \cite[Exercise 5.44.a)]{zink}
\begin{prop}\label{prophaut}
Dans la situation de la définition \ref{defiHauteur}, on a 
\[
\mathrm{ht}(\pG) = \Length_{W( k(\pG))} \left( (M(\YG)/ \varphi_* M(\XG))\otimes_{W(A)} W(k(\pG))    \right). 
\]
\end{prop}


\begin{ex}\label{exfrob}
Si $A$ est une $\F_p$-algèbre,    le Frobenius  $\mathrm{Frob}_q:  \XG\to  \XG^{(q)}$   induit au niveau de modules de Cartier  la multiplication par $V$.   Comme $M(\XG)/ V M(\XG)^{(q)}$ est un module libre isomorphe à $\lie \XG$,  $\mathrm{Frob}_q$ est une isogénie de hauteur $\dim \XG$ \cite[Lemma 5.19.]{zink}.  
\end{ex}

Nous expliquons  comment détecter la propriété d'être une isogénie au niveau de module de Cartier.     Nous nous ramenons d'abord à la fibre spéciale. 

\begin{prop}
Soit $A$ une $\OC_K$ algèbre,  $I\subset A$ un idéal nilpotent et $A_0= A/I$,   alors $\varphi: \XG\to \YG$ et une isogénie si et seulement si $\phi_0:  \XG_{A_0}\to  \YG_{A_0}$ l'est.  
\end{prop}

\begin{proof}
Voir \cite[ Lemma 5.5.]{zink}.
\end{proof}

\begin{theo}
\label{TheoIsogenie}
Si $A$ est un $\F_p$-algèbre, $\varphi:  \XG \to  \YG$ un morphisme    de $\OC_K$-modules  formels,   $\varphi_*:  M(\XG)\to M(\YG)$ est le morphisme induit, les conditions suivantes sont équivalentes :  
\begin{enumerate}

\item  $\varphi$ est une isogénie,

\item  il existe $\psi$ tel que $\psi\circ \varphi= \mathrm{Frob}^n$, 

\item     $M(\YG)/ \varphi_*( M(\XG ))$ est de $V$-torsion, 

\item $\phi_*[\frac{1}{V}]$ est un isomorphisme.  

\end{enumerate}
\end{theo}

\begin{proof}
L'équivalence entre $3.$ et $4.$ est claire et entre deux et quatre, cela résulte du faite que $\fro$ agit par multiplication par $V$ sur les modules de Cartier et de \ref{exfrob}. Enfin l'équivalence entre $1.$ et $2.$ est une application de \cite[Theorem 5.24]{zink} (du moins pour le sens direct, l'autre sens étant clair.) (voir aussi \cite[Theorem 5.25]{zink})
\end{proof}

\begin{prop}
\label{ProppDivisibleDef}
Soit $\XG$ un $\OC_K$-module  formel,  les conditions suivantes sont équivalentes

\begin{enumerate}

\item $[\varpi]$ est une isogénie. 

\item $\varpi:  M(\XG)\to M(\XG)$ est injectif.

\item $F :  M(\XG)\to M(\XG)^{(q)}$ est injectif. 
\end{enumerate}
Si $\XG$ satisfait ces conditions on dit qu'il est $p$-divisible.  La hauteur de $\XG$ désignera la hauteur de $[\varpi]$.  
\end{prop}

\begin{proof}
L'équivalence entre (1) et (2) découle du  Théorème \ref{TheoIsogenie}.2..  L'équivalence entre (2) et (3) provient de l'identité $\varpi \equiv VF \mod [p]$ et de l'injectivité de $V$. 
\end{proof}

\begin{rem}
\label{RemarkExtUniv}
Observons la suite exacte pour tout groupe $p$-divisible $\XG$
\[
0 \to V M(\XG)/ VFM(\XG) \to  M(\XG)/VF M(\XG) \to  M(\XG)/VM(\XG)  \to 0.
\]
On a des identifications $M(\XG)/VM(\XG) = \lie \XG$ et $VM(\XG)/ VFM(\XG) \cong M(\XG)/FM(\XG)\cong  (\lie \XG^{\vee})^{\vee}$.   La suite exacte précédente correspond  à la suite obtenue en appliquant le foncteur $\lie$  à l'extension universelle de $\XG$ introduite dans \cite[§4 Theorem 2.2]{messing}.  
\end{rem}

\subsection{Classification des modules de Cartier spéciaux sur $\bar{\F}$ \label{ssecclasscart}}

Ici $A= \bar{\F}$.   Si $\XG$ est un $\OC_K$-module formel $p$-divisible sur $A$,   $M(\XG)$ est un $\OC_{\breve{K}}$-module de type fini sans $\varpi$-torsion d'après la proposition \ref{ProppDivisibleDef}.2. et est donc libre de rang $ \mathrm{ht}(\XG)$ (voir \ref{prophaut}).     On s'intéresse  aux $\OC_D$-modules formels spéciaux  de dimension $2$ et hauteur $4$.  

Nous voulons dans un premier temps établir ce résultat

\begin{prop}
Il existe une unique classe d'isogénie de $\OC_D$-module formel spécial sur $\bar{\F}$  de dimension $2$ et hauteur $4$.   

\end{prop}

Pour cela, nous aurons besoin de quelques notions et notations.  Il sera essentiel de comprendre les inclusions des sous-groupes 
\[
VM(\XG)_i,  \;\;  FM(\XG)_i \et \Pi M(\XG)_i
\]
dans $M(\XG)$ avec $i=0,1$ et $M(\XG)=M(\XG)_0\oplus M(\XG)_1$ la décomposition en tant que $\OC_{{(2)}}$-module.

\begin{claim}
Si $\XG$ est un $\OC_{D}$-module formel spéciale  de dimension $2$ et hauteur $4$ sur une $\bar{\F}$,  on a    
\[[M(\XG)_{i+1}: V M(\XG)_i ] = [M(\XG)_{i+1}: F M(\XG)_i ] =[M(\XG)_{i+1}: \Pi M(\XG)_i] =1 .\]
\end{claim}

En effet, on a  $[M(\XG)_{i+1}: V M(\XG)_i ] =1$ par hypothèse de dimension et caractère spécial de $\XG$.   Le reste découle du fait que  $\varpi M(\XG)_{i}= V F M(\XG)_{i}= FV M(\XG)_{i} = \Pi^2 M(\XG)_{i} $,  de l'injectivité de $F$,  $V$,  $\Pi$,  et du fait que $[ M(\XG)_i:\varpi M(\XG)_i]= 2$. L'observation précédente nous amène à considérer :

\begin{defi}\label{deficrit}
On dit que $i$ est un indice critique si les conditions équivalentes suivantes sont vérifiées 

\begin{enumerate}

\item L'application induite par $\Pi$ sur l'espace  tangent  $\Pi^*: (\lie \XG)_i \to (\lie \XG)_{i+1}$
 est  nulle. 

\item $\Pi M(\XG)_i \subset VM(\XG)_i$  

\item $\Pi M(\XG)_i= VM(\XG)_i $

\end{enumerate}

\end{defi}

\begin{prop}
Il existe un indice critique pour tout $\OC_D$-module formel spécial sur une  $\bar{\F}$-algèbre.    
\end{prop}

\begin{proof}
En effet,  la composition  d'application linéaire suivante est nulle:  
\[
\Pi^*\circ \Pi^*=\varpi^*: (\lie \XG)_i\to  (\lie \XG)_{i+1}  \to   (\lie \XG)_i.
\] 
Comme chacun des espaces vectoriels considérés est de dimension $1$,  l'une des applications $\Pi^*$ est nulle.

\end{proof}

\begin{prop}\label{propbasisog}
Soit $\XG$ un module formel spécial  de dimension $2$ et hauteur $4$ sur $\bar{\F}$.  Si $i$ est un indice critique,  alors on peut trouver $x_0,  x_1\in M(\XG)_i$  tel que 
\[
M(\XG)_i=\OC_{\breve{K}} x_0 \oplus  \OC_{\breve{K}}  x_1, \;\; \Pi  x_j= Vx_{j}= F x_j.   
\] 
\end{prop}

\begin{proof}
Par hypothèse,  $\Pi M(\XG)_i = V M(\XG)_i$.  Comme $V$ et $\Pi$ sont injectives,  on peut construire un isomorphisme $\sigma$-linéaire  $ V^{-1} \Pi: M(\XG)_i\to  M(\XG)_i$.   Par classification de Dieudonné-Manin $( M(\XG)_i[1/p] ,  V^{-1}\Pi)$  est l'isocrystal \'etale de dimension $2$ car $M(\XG)_i $ en est un réseau $V^{-1} \Pi$-stable.   On en déduit que $M(\XG)_i= \OC_{\breve{K}} x_0 \oplus \OC_{\breve{K}} x_1$ avec $x_0,x_1$ des points fixes pour $V^{-1}\Pi$.   De plus,  on a $VFx_j= \varpi x_j = \Pi^2 x_j= V \Pi x_j $ d'où  $F x_j= \Pi x_j$ par injectivité de $V$.   
\end{proof}

Nous sommes maintenant en mesure de démontrer  la proposition  \ref{propClassIsog}.  Cela revient à prouver,  d'après   le théorème \ref{TheoIsogenie}, 

\begin{prop}
Soit $\XG$ un module formel spécial sur $\bar{\F}$,  on peut trouver $x_0, x_1 \in M(\XG)$ tel que 
\[
M(\XG)[1/p]= \breve{K} x_0 \oplus \breve{K} x_1 \oplus \breve{K} \Pi x_0 \oplus \breve{K} \Pi x_1 
\]
avec $\Pi x_j = F x_j = V x_j$.   De plus,  on peut imposer que   $x_0,  x_1$ forment une $\OC_{\breve{K}}$-base de $M(\XG)_i$ pour $i$ un indice critique.  

Les quasi-isogénies de $\XG$ s'identifient à $\gln_2(K)$.   
\end{prop}

\begin{proof}
On décompose $M(\XG)[1/p]=M(\XG)_0[1/p] \oplus M(\XG)_1[1/p]$.   En utilisant la proposition précédente,  $M(\XG)_i[1/p]= \breve{K} x_ 0\oplus \breve{K} x_1$ avec $x_0,x_1\in M(\XG)_i^{V^{-1}\Pi}$ et $i$ un indice critique.   Comme $\Pi$ induit un isomorphisme $\breve{K}$-linéaire $M(\XG)_i [1/p] \iso M(\XG)_{i+1}[1/p]$,  on en déduit la première partie du résultat.    

Pour la deuxième, une quasi-isogénie $\varphi: M(\XG)[1/p]\iso M(\XG)[1/p]$ est équivalente aux choix de l'image de $x_0,x_1$ dans  $M(\XG)_i^{V^{-1}\Pi}[1/p]= K x_0 \oplus K x_1$ car elle commute avec $\Pi$, $F$, $V$ et $\breve{K}$.   Cela donne une application bijective 
\[
\mathrm{QIsog}(\XG)\to \gln(M(\XG)_i^{V^{-1}\Pi}[1/p] )\cong \gln_2(K).  
\]    
\end{proof}

Fixons $\Phi_{\bar{\F}}$ un module formel spécial sur $\bar{\F}$ de dimension $2$ de hauteur $4$ et $M^{(0)}$ le module de Cartier associé.  On peut supposer que $1$ est critique dans $M^{(0)}$ et  fixons une base \[(M^{(0)}_1)^{V^{-1}\Pi}= \OC_{{K}}x_0 \oplus \OC_{{K}} x_1.\]       Nous nous intéressons maintenant à $\GC^{Dr}(\bar{\F})$  qui classifie les classes d'isomorphisme des couples\footnote{Notons que pour un anneau $A$ quelconque de $\nilp_{\OC_{\breve{K}}}$, un point de $\GC(A)$ est la donnée d'un morphisme $\psi : \bar{A}\to A/\varpi A$ en plus de celle d'un couple $(  \XG , \varphi)$. Lorsque $A=\bar{\F}$, la donnée de $\psi$ est redondante.} $(  \XG , \varphi)$
\begin{itemize}[label= \textbullet] 
\item $\XG$ est un $\OC_D$-module formel sp\'ecial de dimension $2$ et de hauteur $4$ sur $\bar{\F}$, 
\item $\varphi : \Phi_{\bar{\F}} \to \XG$ est une quasi-isog\'enie de hauteur $0$. 
\end{itemize}

Nous voulons rendre explicite l'identification de $\GC(\bar{\F})$ avec $\bar{\H}(\bar{\F})= \bigcup_{s\in \BC \TC_0}  \P_s(\bar{\F}) $. Pour cela,  donnons-nous un module de Cartier $M$ associé à un module spécial de hauteur $4$ et dimension  $2$ sur $\bar{\F}$ ainsi qu'une quasi-isogénie $\varphi:  M[1/p]\iso M^{(0)}[1/p]$ de hauteur $0$.  Par théorie de Cartier,  la donnée de $(M, \varphi)$ est équivalente  à la donnée d'un point $\GC^{Dr}(\bar{\F})$.    On définit deux $\OC_K$-réseaux de   $(M^{(0)}_1)^{V^{-1}\Pi}[1/p]\cong K^{2}$   donnés par 
\[
\varphi (M_1^{V^{-1}\Pi}) \et \varphi (\Pi M_0^{V^{-1}\Pi}).  
\]    
Appelons $s_0(M):=[\varphi (\Pi M_0^{V^{-1}\Pi})],s_1(M):=[\varphi (M_1^{V^{-1}\Pi})]$ les sommets associés dans l'arbre de Bruhat-Tits.  On identifiera $[(M_1^{(0)})^{V^{-1}\Pi}]$  au sommet standard de $\BC \TC_0$.   

\begin{prop}
$s_0(M)$ est pair  si et seulement  si $0$ est critique.  De même,  $s_1(M)$ est impair si et seulement si $1$ est critique.    
\end{prop}
\begin{proof}
L'indice $1$ est critique si et seulement si   $\varphi M_1 = \varphi M_{1}^{V^{-1}\Pi}\otimes \OC_{\breve{K}}$ si et seulement si $[(M^{(0)}_1)^{V^{-1}\Pi}  :  \varphi (M_1^{V^{-1}\Pi}) ]=0$  (car   $  [M_{1}^{(0)}: \varphi M_1]=0$)   si et seulement si $s_1(M)$ a la même parit\'e que le sommet standard.   Pour l'autre équivalence,  on observe que  $[  M^{(0)}_1 : \Pi M_1^{(0)}]=1$  donc $0$ est critique si et seulement si  $[(M_1^{(0)})^{V^{-1}\Pi}: \varphi (\Pi  M_1^{V^{-1}\Pi}) ]=1$ si et seulement si $s_0(M)$ n'a pas la même parité que  le sommet standard.     
\end{proof}

\begin{const}

Fixons $s$ un sommet de $\BC \TC_0$ et intéressons-nous aux couples $(M, \varphi)$ tel que $s=s_i(M)$ avec $i$ critique.     Si $i$ est critique,   on a une suite d'inclusions 
\[
   \varpi M_{i} \subset \Pi M_{i+1} \subset M_{i}
\]
 qui détermine un point   \[y_i(M)\in \P(M_{i}/ \varpi M_{i})\cong\P_{s_i(M)}(\bar{\F}).\]    Notons que si $j$ n'est pas critique,  $s_j(M)=s_{j+1}(M)$   et on définit  $y_{j}(M)=y_{j+1}(M)$ dans  $\P_{s_j(M)}(\bar{\F})=\P_{s_{j+1}(M)}(\bar{\F})$.  

La construction précédente définit une flèche $G$-équivariante  $\psi:\GC^{Dr}(\bar{\F})\to (\bar{\H}\times \bar{\H} )(\bar{\F}) $ donnée par 
$(M, \varphi)\mapsto (s_0(M),s_1(M),  y_0(M), y_1(M))$.  

\end{const}

\begin{theo}\label{theopsi}
La flèche $\psi$ se factorise par la diagonale   $ \bar{\H}(\bar{\F}) \subset (\bar{\H}\times \bar{\H} )(\bar{\F})$ et induit une bijection $G$-équivariante  $\tilde{\psi}:\GC^{Dr}(\bar{\F})\cong \bar{\H}(\bar{\F}) $
\end{theo}

\begin{proof}
  S'il existe un indice non critique,  les couples $(s_i(M),  y_i(M))$ ne dépendent pas de $i$ et définissent un point de $\P_{s_i(M)}(\bar{\F}) \subset \bar{\H}(\bar{\F})$.  Lorsque les deux indices sont critiques,  on observe la suite d'inclusions 
\[
\varpi M_i^{V^{-1}\Pi} \subset \Pi M_{i+1}^{V^{-1}\Pi} \subset M_i^{V^{-1}\Pi}.  
\]
Ceci montre que $(s_0(M),s_1(M))$ est une arête  de $\BC \TC_1$    et $y_{0}(M)$ et $y_1(M)$ correspondent à l'unique point d'intersection entre $\P_{s_0(M)}$ et $\P_{s_1(M)}$.  Nous avons prouvé que  $\psi$ se factorise par la diagonale  de $ (\bar{\H}\times \bar{\H} )(\bar{\F})$ et que la flèche  $\tilde{\psi}:\GC^{Dr}(\bar{\F})\cong \bar{\H}(\bar{\F}) $  est bien définie.

Montrons maintenant l'injectivité de $\tilde{\psi}$.  Soit $(M^{(1)}, \varphi^{(1)}),  (M^{(2)}, \varphi^{(2)})$ deux éléments de  $\GC^{Dr}(\bar{\F})$ tel que $s_i(M^{(1)})= s_{i}(M^{(2)})$ et $y_i(M^{(1)})= y_i(M^{(2)})$ pour $i=0,1$.   Si $i$ est critique,   le $\OC_{\breve{K}}$-réseau  $\varphi^{(j)}(M^{(j)}_i)$ vaut $\varphi^{(j)}(M^{(j)}_i)^{V^{-1}\Pi}\otimes\OC_{\breve{K}}$ et est donc déterminé par $s_i(M)$ qui  ne dépend pas de $j$.    Si $i$ n'est pas critique, $i+1$ l'est et le même raisonnement montre que  $\varphi^{(j)}(M^{(j)}_{i+1})$ ne dépend pas de $j$. De plus,  on a une suite d'inclusions 
\[
\varpi   \varphi^{(j)}(  M_{i+1}^{(j)} ) \subset  \varphi^{(j)}(  \Pi M_{i}^{(j)} ) \subset \varphi^{(j)}(  M_{i+1}^{(j)})
\]
qui décrit  $y_i(M^{(j)})$ et détermine $\varphi^{(j)}(M^{(j)}_{i})$.  On en déduit que $\varphi^{(1)}(M^{(1)})= \varphi^{(2)}(M^{(2)})$.  Comme ces applications sont injectives,  on a une flèche $(\varphi^{(1)})^{-1} \circ \varphi^{(2)} :  M^{(2)} \to M^{(1)}$ qui est une bijection.

Pour la surjectivité de $\tilde{\psi}$,  il suffit de raisonner sur une composante irréductible et intéressons nous à  $\P_s(\bar{\F})$ avec $s$ le sommet standard.       Si $y\in \P_s(\bar{\F})$,   la droite $\LC\subset \bar{\F}^2 \cong  M_1/ \varpi M_1$     associée définit un module $\varpi M_1\subset  \tilde{M}_0\subset  M_1$. Posons  $M_0= \Pi^{-1} \tilde{M}_0$,  montrons que le sous-module  $M:= M_0 \oplus M_1$ de $M^{(0)}[1/p]$ définit un point de $\GC(\bar{\F})$  par i.e. $M$ est  stable par   $\Pi,  F, V$ et  $[M^{(0)}: M]=0$. Si c'est le cas, il s'enverra sur $y$ par construction. 

$\bullet$ On a $M_1= M_1^{(0)}$ d'où $[M_1^{(0)}:M_1]=0$ et \[[M_0^{(0)}:M_0]=[\Pi M_0^{(0)}:\tilde{M}_0]=[M_1: \tilde{M}_0]-[  M_1 :  \Pi M_0^{(0)}]=0\] car $[M_1: \tilde{M}_1]=1= [  M_1 :  \Pi M_0^{(0)}]$. On a donc $[M^{(0)}:  M]=0$.    

$\bullet$ Par construction,  $\Pi M_0= \tilde{M}_0 \subset M_1$ et \[\varpi\Pi M_1 \subset  \Pi \tilde{M}_0 =\Pi^2  M_0=\varpi M_0. \] Comme la multiplication par $\varpi$ est injective,   $\Pi M_1\subset M_0$ et $M$ est bien stable par $\Pi$.   Comme $1$ est critique, $M_1=M_1^{V^{-1}\Pi}\otimes \OC_{\breve{K}}$ et d'après la preuve de \ref{propbasisog} les actions de $\Pi$, $V$ et $F$ coïncident sur $M^{V^{-1}\Pi}_1$. Ainsi, on a  $\Pi M_1 = F M_1 = V M_1 \subset M_0$.   De plus,  $VM_0 = F^{-1}\Pi^2 M_0  \subset  F^{-1} \Pi M_1= M_1$.   Le même raisonnement sur $F$ montre que $M$ est stable par $F$ et $V$ ce qui termine la preuve.  \end{proof}

\begin{rem}
\label{RemPullbackPeriodes}
  Par définition de $\P^1=\gln_2/B$ en tant que faisceau,   $\Of(-1)\subset \std^{(2)} \otimes \Of_{\P^1}$ est l'objet universel paramétrant un sous-fibré en droites du faisceau libre de rang $2$. En fixant $s$ un sommet de parité $i$ et en considérant la restriction de $\tilde{\psi}$ à $\P_s$, on en déduit  
\[
\tilde{\psi}^*(\std^{(2)}\otimes \P_s)=   M_{i}/\varpi M_{i} \et \tilde{\psi}^*(\Of(-1))= \Pi M_{i+1}/ \varpi M_i=\Pi \lie(\XG|_{\P_s})_{i+1}
\]
par construction de l'application des périodes via la filtration de Hodge. 

Notons aussi que l'on peut réaliser le même jeu en fibre générique et construire une application de périodes  $\pi :(\GC^{Dr})^{rig}\to \P_{\breve{K}}^1$ donnée au niveau des points géométriques  \[(M,\varphi)\in\GC^{Dr}(C)\mapsto[\Pi M_{0}/ \varpi M_1\subset M_{1}/ \varpi M_1]\in \P^1_{\breve{K}}(C)\] où $\varpi$ désigne ici l'uniformisante de $W_{\OC_K}(C)$. Cela montre en particulier que $\omega_0[1/\varpi]\cong\Of(-1)$ dans $\pic_{[G]_2}({\H})$ e de même pour $\omega_1[1/\varpi]$ via $\Pi:\omega_0[1/\varpi]\iso\omega_1[1/\varpi]$.
\end{rem}

La preuve de la surjectivité de $\psi$ construit pour chaque point géométrique $y\in \bar{\H}$ un module de Cartier spécial $M_y$ avec une isogénie $\varphi_y$ tel que $\tilde{\psi}(M_y,\varphi_y)=y$. Le résultat suivant va un peu plus loin et fournit une base explicite à $M_y$ sur laquelle les opérateurs $\Pi$, $V$, $F$ sont relativement faciles à décrire.

\begin{coro}
\label{LemmeBase}
Soit $y\in \P_s(\bar{\F})\backslash \{\infty\}\cong \bar{\F}$ avec $s\in \BC\TC^{(i)}_0$ et $\infty$ un point géométrique de $\P_s$ fixé.  Si $M=M_0\oplus M_1$ est le module de Cartier associé, on peut trouver des   $\OC_{\breve{K}}$-bases    $\{x_{0,j},  x_{1,j} \}$ de $M_j$ ($j=0,1$) tel que 
\[
\begin{array}{l}
\Pi x_{0,i} =F x_{0,i}= V x_{0,i} = \varpi x_{0,i+1}-[y] x_{1,i+1}, \\ 
\Pi x_{1,i}=F x_{1,i}= V x_{1,i}= x_{1,i+1}, \\ 
\Pi x_{1,i+1}= F x_{1,i+1} = V x_{1,i+1} = \varpi x_{1,i}, \\ 
\Pi x_{0 ,i+1}=  x_{0,i}+[y] x_{1,i}, \\
F x_{0 ,i+1 }=  x_{0,i}+[y^q] x_{1,i},   \\
V x_{0 ,i+1}=  x_{0,i}+[y^{1/q}] x_{1,i}.
\end{array}
\]
\end{coro}

\begin{proof}

Supposons ici que $s$ est le sommet standard,  les autres sommets se traiterons de manière similaire.  On voit $M$ comme un réseau  de $M^{(0)}[1/\varpi]$  avec $M^{(0)}_1= M_1$.  On fixe $x_{0,1}$ et $x_{1,1}$ une $\OC_{K}$-base de $M_1^{V^{-1}\Pi}$,    identifions $\P_s(\bar{\F})$ avec $\P( M_1 /\varpi    M_1)$ via  
\[
[z_0,z_1]\mapsto \bar{\F} (z_0  x_{0,1}+ z_1 x_{1,1} ).  
\]
La chaîne d'inclusions  $ \varpi  M_1 \subset  \Pi M_0 \subset  M_0 $ définit   un point $[\tilde{z}_0,  \tilde{z}_1]$ dans $ \P_s(\bar{\F})$.   Sans perte de généralité,  on peut supposer que $\tilde{z}_0\neq 0$ (ie. $[\tilde{z}_0,  \tilde{z}_1]\neq \infty$) et  posons $y= \tilde{z}_1/\tilde{z}_0$.    Alors,  l'image de $\Pi M_0$ dans  $M_1 / \varpi  M_1$ est  $ \bar{\F} ( x_{0,1}+yx_{1,1})$.   Ainsi,    $(x_{0,0}, x_{1,0}):=(  \frac{1}{\varpi}\Pi (  x_{0,1}+[y]  x_{1,1}),    \Pi x_{1,1} )$ est une base de $M_1$.  

Nous avons décrit explicitement une base de $M$ et la description de la multiplication par $\Pi$,  $V$,  $F$ dans cette base découle par un calcul direct.  

\end{proof}

\section{\'Etude des fibrés en droites modulaires\label{seccalcordre}}

\subsection{Calcul des ordres $\omega_0$ et  $\omega_1$}
Soit $\XG$ un module spécial sur $A$ un anneaux de $\nilp_{\OC_{\breve{K}}}$,  et $M= M_0 \oplus M_1$ le module de Cartier associé.  L'opérateur $\Pi$  induit  une  flèche  $M_i/VM_{i+1}\to M_{i+1}/VM_i$,  cela définit un morphisme de fibrés  
\[
\Pi^* :  \omega_{i}\to \omega_{i+1}  .
\]
Quand $A$ est une $\F$-algèbre,  le Frobenius relatif $F: \XG\to \XG^{(q)}$ induit un morphisme 
\[
F^*:  \omega_{i}^{(q)}|_{\bar{\H}} \to  \omega_{i+1}|_{\bar{\H }}.
\]
Le  but de cette section est d'étudier ces deux morphismes en fibre spécial.

\begin{theo}
\label{theoOrdOmegai}
Soit $s\in  \BC\TC_0^{(i)}$,    les flèches    $\Pi_*:(  \lie \XG)_i|_{\P_s} \to (\lie \XG)_{i+1}|_{\P_s}$ et $F_*:  (\lie  \XG)^{(q)}_{i}|_{\P_s}\to( \lie \XG)_{i+1}|_{\P_s}$  sont nulles et 
\begin{align*}
(\lie \XG)_{i}|_{\P_s}/ \Pi_* ( \lie \XG)_{i+1}|_{\P_s} & \cong \bigoplus_{y \in  \P_s(\F)}   \iota_{y,*} \bar{\F}  \\
(\lie \XG)_{i}|_{\P_s} / F_* (\lie \XG)^{(q)}_{i+1} |_{\P_s} &\cong   \bigoplus_{y \in  \P_{s}(\F_{(2)})} \iota_{y,*} \bar{\F}.      
\end{align*}
\end{theo} 

\begin{coro}
\label{coroOrdreOmega}
On a 
\begin{enumerate}

\item $\ord_{s_0}(\omega_0)=- 1 $ et    $\ord_{s_0}(\omega_1)= q$. 
 
\item  $\ord_{s_1}(\omega_0)=  q$  et $\ord_{s_1}(\omega_1)= -1$.

\end{enumerate}
\end{coro}

\begin{proof}
Les identités du théorème précédent entrainent le système d'équations 
\begin{align*}
-\ord_{s_0} \omega_0  -(- \ord_{s_0} \omega_1) &  = q+1 \\
-\ord_{s_0} \omega_0  -(- q\ord_{s_0} \omega_1) & =q^2+1
\end{align*}
cela implique le point (1).   Le point (2) s'obtient de la même manière.  
\end{proof}

\begin{proof}[D\'emonstration du théorème \ref{theoOrdOmegai}]
Fixons $s$ un sommet de $\BC\TC^{(i)}_0$ et $y\in \P_s(\bar{\F})$.  \'Etant donné $\infty \neq y$ un point $\F$-rationnel de $\P_s$, on peut voir $y$ comme un élément de $\bar{\F}\cong \P_s(\bar{\F})\backslash \{\infty\}$ (resp. $ \F_{(2)}\cong\P_s(\F_{(2)})\backslash \{\infty\} $). D'après \ref{theopsi},  correspond à $y$ un module formel spécial $(\XG_y, \varphi_y)$ et donc un module  de Cartier $(M_y,\varphi_y)$  dont une base $\{x_{j,k}\}_{j,k\in \Z/2\Z}$ a été décrite explicitement dans le lemme \ref{LemmeBase}. Pour, nous voulons savoir en quels points $y$ de $\P_s(\bar{\F})$ s'annule les flèches suivantes et à quel ordre (pour $j=0,1$)
\[\Pi_* :M_{y,j}/M_{y,j+1}\to M_{y,j+1}/M_{y,j}\]
\[F_* :(M_{y,j}/M_{y,j+1})^{(q)}\to M_{y,j+1}/M_{y,j}.\]    
 Lorsque $j=i$, on a $VM_{y,i}=FM_{y,i}=\Pi M_{y,i}$ d'après \ref{deficrit} et \ref{LemmeBase} et les flèches $\Pi_*$ et $F_*$ s'annulent. En particulier, les flèches sous-jacentes entre les fibrés en droite sont nulles.

Supposons maintenant $j=i+1$, alors $VM_{y,i+1}/\varpi M_{y,i} =\bar{\F}\cdot(x_{0,i}+[y^{1/q}] x_{1,i})$ et $x_{1,i}$ (resp. $x_{0,i+1}$) engendre $M_{y,i}/VM_{y,i+1}$ (resp. $M_{y,i+1}/VM_{y,i}$). On observe alors la congruence
\[\Pi_* x_{0,i+1}\equiv ([y]-[y^{1/q}] )x_{1,i}+x_{0,i}+[y^{1/q}] x_{1,i}\equiv (y-y^{1/q} )x_{1,i} \pmod{VM_{y,i+1}}\]
\[F_* x_{0,i+1}\equiv ([y^q]-[y^{1/q}] )x_{1,i}+x_{0,i}+[y^{1/q}] x_{1,i}\equiv (y^q-y^{1/q} )x_{1,i} \pmod{VM_{y,i+1}}\] qui montre que la flèche $\Pi_*$ (resp. $F_*$) s'annulent si et seulement si $y$ est dans $\F$ (resp. $\F_{(2)}$).

Maintenant, nous voulons montrer que l'ordre d'annulation de $\Pi_*$ (resp. $F_*$) en les $y$ ci-dessus  vaut $1$. Appelons $T_y \bar{\H}$ (resp. $T_y \P_s$) la préimage de $y$ dans $\H(\bar{\F}[\epsilon])$ (resp. $\P_s(\bar{\F}[\epsilon])$). Il est suffisant d'établir que pour $(\tilde{M}, \tilde{\varphi})\in T_y \P_s$   qui n'est pas la déformation triviale,  les flèches $\Pi_*$ et $F_*$ sont non-nulles sur les algèbres de Lie.

Le résultat suivant classifie les déformations dans $T_y \bar{\H}$ et est une conséquence directe de la théorie de Dieudonné-Messing.

\begin{lem}\label{lemfiltmess}
L'ensemble $T_y \bar{\H}$ est en  bijection  avec l'ensemble des sous-$\bar{\F}[\epsilon]$-module $D\subset M_y\otimes \bar{\F}[\epsilon]$ de rang $2$  tels que :
\begin{itemize}
\item $D\otimes_{\bar{\F}[\epsilon]} \bar{\F}\cong VM_y$
\item Chaque $D_j:=D\cap (VM_{y,j}\otimes \bar{\F}[\epsilon])$ vérifie $D_j\otimes_{\bar{\F}[\epsilon]} \bar{\F}\cong VM_{y,j}$ pour $j=0,1$ et $D=D_0\oplus D_1$.
\item $(\Pi \otimes \bar{\F}[\epsilon]) D\subset D$.
\end{itemize}
\end{lem}

\begin{proof}
Considérons d'abord une déformation quelconque $\tilde{\XG}/\bar{\F}[\epsilon]$ (qui ne provient pas forcément d'un élément de $T_y \bar{\H})$) de $\XG_y/\bar{\F}$  et $\tilde{M}$ le  module de Cartier $\tilde{\XG}$. Par propriété de cristal de l'extension universelle (voir \ref{RemarkExtUniv}), on a une identification :
\[
\tilde{M}/\varpi \tilde{M} \cong M_y/ \varpi M_y \otimes \bar{\F}[\epsilon] 
\]
Si on note (pour éviter toute confusion) $\tilde{V}$ le Verschiebung sur $\tilde{M}$, le théorème principal de la théorie de Dieudonné-Messing (\cite[§5 Theorem 1.6]{messing}) affirme  que le foncteur \[\tilde{M}\mapsto (\tilde{V}\tilde{M}/ \varpi \tilde{M}\subset \tilde{M}/ \varpi \tilde{M})\]  est une équivalence des catégories entre les déformations de groupes $p$-divisibles  de  $\XG_y$ sur  $\bar{\F}[\epsilon]$ et les relevés de la filtration   $VM_y/\varpi M_y \subset M_y/ \varpi M_y$. 

Pour tout morphisme $\psi :\YG\to\ZG$ de modules formels et toutes déformations $\tilde{\YG}$, $\tilde{\ZG}$ de $\YG$, $\ZG$, le résultat précédent avec la propriété de cristal de l'extension universelle entraîne que le morphisme $\psi$ peut se relever en $\tilde{\psi} : \tilde{\YG}\to\tilde{\ZG}$ si et seulement si la flèche $\psi_* \otimes \bar{\F}[\epsilon]: M(\YG)\otimes \bar{\F}[\epsilon]\to M(\ZG)\otimes \bar{\F}[\epsilon]$ préserve les filtrations associées à $\tilde{\YG}$, $\tilde{\ZG}$ par l'équivalence précédente où $\psi_*$ le morphisme induit au niveau des modules de Cartier. Dans ce cas, le relevé est unique et correspond à $\psi_* \otimes \bar{\F}[\epsilon]:M(\tilde{\YG})\to M(\tilde{\ZG} )$.

Si une déformation $(\tilde{\XG},\tilde{\varphi})$ de $(\XG_y,\varphi_y)$ est un point de $T_y \bar{\H}$, alors la quasi-isogénie $\tilde{\varphi}$ est uniquement déterminée par la donnée de $\tilde{\XG},\XG_y,\varphi_y$ d'après la discussion précédente. Ainsi, $T_y \bar{\H}$ est en bijection avec les déformations de $\XG_y$ qui sont des $\OC_D$-modules. Comme la $\Z/2\Z$-graduation sur un $\OC_D$-module de Cartier est équivalente à l'action de $\OC_{(2)}$ et la multiplication par $\Pi $ à l'action de $\Pi_D\in\OC_D$, la déformation $\tilde{\XG}$ est un $\OC_D$-module si et seulement si la filtration $(\tilde{V}\tilde{M}/ \varpi \tilde{M}\subset \tilde{M}/ \varpi \tilde{M})$ est stable par la multiplication par $\Pi\otimes \bar{\F}[\epsilon] $ et admet aussi une $\Z/2\Z$-graduation qui relève celle  sur $M_y$. Ces filtrations sont exactement  celles apparaissant dans l'énoncé.
\end{proof}

\begin{rem}
Reprenons le $\OC_D$-module formel $\XG_y$ ainsi qu'une déformation $\tilde{\XG}$ dans $T_y\bar{\H}$. Notons que, sur le module de Cartier $\tilde{M}$ associé à $\tilde{\XG}$, on a deux multiplications possibles par $\Pi$, $V$, $F$. Comme  $\tilde{M}$ provient d'un $\OC_D$-module formel spécial, il est naturellement muni d'une structure de $E_{\OC_K}(\bar{\F})[\Pi]$-modules ce qui le munit d'une action par d'opérateurs que l'on notera $\tilde{\Pi}$, $\tilde{V}$ et $\tilde{F}$ ($\tilde{V}$ coïncide avec celui de la preuve précédente).

Suivant l'isomorphisme $\tilde{M}/\varpi\tilde{M}\cong M_y/\varpi M_y\otimes \bar{\F}[\epsilon]$ provenant de la propriété de cristal de l'extension universelle, on peut aussi considérer l'action de \[\Pi[\epsilon]:=\Pi\otimes \bar{\F}[\epsilon],\ \ V[\epsilon]:=V\otimes \bar{\F}[\epsilon],\ \ F[\epsilon]:=F\otimes \bar{\F}[\epsilon].\] Dans le prochain énoncé, nous montrerons que,  pour toute déformation  dans $T_y \bar{\H}$, la filtration associée est stable par $\Pi[\epsilon]$, $V[\epsilon]$, $F[\epsilon]$ qui sont alors induits par les uniques relevés à $\tilde{\XG}$ de \[\Pi:\XG_y\to \XG_y, \ F: \XG_y\to (\XG_y)^{(q)}, \ V:(\XG_y)^{(q)}\to \XG_y\]
qui agissent par $\Pi$, $V$ et $F$.

Il est aussi important de noter que les opérateurs $\tilde{V}$ et $\tilde{F}$ ne proviennent pas de morphismes du groupe formel $\tilde{\XG}$ quand la déformation n'est pas triviale (car $\F[\epsilon]$ n'est pas parfait) et ces derniers ne coïncident pas avec $V[\epsilon]$ et $F[\epsilon]$ dans ce cas. Pour $\tilde{\Pi}$, cet opérateur  provient du morphisme induit par $\Pi_D\in D$ et coïncide donc avec $\Pi[\epsilon]$.
\end{rem}

Pour pouvoir terminer la description de $T_y\bar{\H}$, nous donnons une base canonique à chacune des déformations $D\subset M_y\otimes \bar{\F}[\epsilon]$ dans $T_y\bar{\H}$. Pour simplifier l'énoncé qui va suivre, nous définissons l'élément suivant pour $j=0,1$ et $a\in \bar{\F}$: \[e_j(b):=V x_{i+j,j+1} + \epsilon  b x_{1+i+j,j}\in M_{y,j} \otimes \bar{\F}[\epsilon]\]
\begin{coro}
On a \begin{equation}\label{eq:ty}
T_y\bar{\H}=\begin{cases}\{ \bar{\F}[\epsilon]e_i(a_0) \oplus\bar{\F}[\epsilon] e_{i+1}(a_1) : a_0,a_1\in\bar{\F}\} & \si y \in \F\\
\{  \bar{\F}[\epsilon] e_i(a) \oplus \bar{\F}[\epsilon] e_{i+1}(0):a\in\bar{\F}\} &\sinon\end{cases}
\end{equation}
et chacune de ces filtrations est stable par $V[\epsilon]$ et  $F[\epsilon]$.

De plus, quand $y\in \F$, on a $T_y\bar{\H}=T_y\P_s\oplus T_y\P_{s'}$ avec $s'$ le deuxième sommet pour lequel  $y\in \P_s(\bar{\F})$. Cette  décomposition correspond à 
\[T_y\P_s=\{ \bar{\F}[\epsilon]e_i(a_0) \oplus\bar{\F}[\epsilon] e_{i+1}(0) : a_0\in\bar{\F}\}\]
\[T_y\P_{s'}=\{ \bar{\F}[\epsilon]e_i(0) \oplus\bar{\F}[\epsilon] e_{i+1}(a_1) : a_1\in\bar{\F}\}\]
 sous l'identification \eqref{eq:ty}.

\end{coro}

\begin{proof}
D'après \ref{lemfiltmess}, les éléments de $T_y\bar{\H}$ sont en bijection avec les déformations $D\subset M_y\otimes \bar{\F}[\epsilon]$ de la filtration $VM_y/\varpi M_y\subset M_y/\varpi M_y$ qui respecte la $\Z/2\Z$-graduation de $M_y$ et est stable par  multiplication par $\Pi[\epsilon]$. Nous classifions d'abord les filtrations $\Z/2\Z$-graduées et distingueront celle stable par multiplication par $\Pi[\epsilon]$ plus tard. Dans ce cas, chaque terme $D_j$ est libre de rang $1$ sur $ \bar{\F}[\epsilon]$ et admet un générateur qui relève $Vx_{0,j+1}$. Comme les éléments $Vx_{0,j+1}$ et $x_{1+i+j,j}$ engendrent $M_{y,j}$, chacun de ces générateurs peuvent s'écrire sous la forme  $e_j(a)+\epsilon \lambda Vx_{0,j}$ (unique en $a,\lambda\in\bar{\F}$). On observe la relation \[(\mu +\epsilon \lambda)e_j(a)=\mu e_j(a) +\epsilon\lambda Vx_{0,j}\] qui prouve que deux éléments $e_j(a)+\epsilon \lambda Vx_{0,j}$ et $e_j(a')+\epsilon \lambda' Vx_{0,j}$ engendrent la même $ \bar{\F}[\epsilon]$-droite  si et seulement si $a=a'$. Ainsi, toute déformation $D\subset M_y\otimes \bar{\F}[\epsilon]$ respectant la $\Z/2\Z$-graduation peut s'écrire de manière unique (en $a_0,a_1\in\bar{\F}$) sous la forme \[D= \bar{\F}[\epsilon]e_i(a_0) \oplus\bar{\F}[\epsilon] e_{i+1}(a_1)\]

Nous décrivons maintenant les déformations qui sont stables par multiplication par $\Pi[\epsilon]$. Si $y\in\F$, pour $a,b\in \bar{\F}$, on a $\Pi[\epsilon]e_j(a)=a\varepsilon e_{j+1}(b)\in\bar{\F}[\epsilon] e_{j+1}(b)$ et cette condition de stabilité n'entraîne aucune restriction.  En particulier, cela prouve \[T_y\bar{\H}=\{ \bar{\F}[\epsilon]e_i(a_0) \oplus\bar{\F}[\epsilon] e_{i+1}(a_1) : a_0,a_1\in\bar{\F}\}.\] De plus, les éléments $D=\bar{\F}[\epsilon]e_i(a_0) \oplus\bar{\F}[\epsilon] e_{i+1}(0) $ sont exactement  ceux qui ont pour indice critique $i$.  Cela entraîne \[T_y\P_s=\{ \bar{\F}[\epsilon]e_i(a_0) \oplus\bar{\F}[\epsilon] e_{i+1}(0) : a_0\in\bar{\F}\}.\] Le même raisonnement sur $i+1$ montre \[T_y\P_{s'}=\{ \bar{\F}[\epsilon]e_i(0) \oplus\bar{\F}[\epsilon] e_{i+1}(a_1) : a_1\in\bar{\F}\}.\] Supposons maintenant $y\notin \F$, on a alors ($y\neq 0$) \[\Pi[\epsilon] e_i(a)=(y+y^{1/q}+\epsilon a)Vx_{0,i}=(y+y^{1/q}+\epsilon a)e_{i+1}(0).\] On en déduit que $ \bar{\F}[\epsilon]\Pi[\epsilon]e_i(a) =\bar{\F}[\epsilon] e_{i+1}(0)$ et les filtrations recherchées  sont de la forme $D= \bar{\F}[\epsilon]e_i(a_0) \oplus\bar{\F}[\epsilon] e_{i+1}(0)$. Comme $\Pi[\epsilon]e_{i+1}(0)=0$, ces dernières conviennent ce qui termine la preuve.
\end{proof}
Terminons maintenant la preuve de \ref{theoOrdOmegai}. Pour cela on considère la déformation $\tilde{M}$ associée à la filtration (voir le résultat précédent) \[D(=\tilde{V}\tilde{M}/\varpi\tilde{M})= \bar{\F}[\epsilon]e_i(a) \oplus\bar{\F}[\epsilon] e_{i+1}(0) \] avec $a\neq 0$. On veut montrer que l'application $\Pi[\epsilon]$ (resp. $F[\epsilon]$) induit un morphisme non nul \[\Pi_* :\tilde{M}_{i+1}/\tilde{V}\tilde{M}_{i}\to \tilde{M}_{i}/\tilde{V}\tilde{M}_{i+1}\ {\rm (resp. }\ F_* :\tilde{M}_{i+1}/\tilde{V}\tilde{M}_{i}\to \tilde{M}_{i}/\tilde{V}\tilde{M}_{i+1})\]  quand $y\in \F$ (resp. $y\in \F_{(2)}$). Les éléments $x_{0,i+1}$ et $x_{1,i}$ engendrent toujours $\tilde{M}_{i+1}/\tilde{V}\tilde{M}_{i}$ et $\tilde{M}_{i}/\tilde{V}\tilde{M}_{i+1}$ respectivement et on a
\[\Pi_* x_{0,i+1}\equiv ([y]-[y^{1/q}] )x_{1,i}+x_{0,i}+[y^{1/q}] x_{1,i}\equiv -\epsilon a x_{1,i}\not\equiv 0 \pmod{\tilde{V}\tilde{M}_{i+1}}\]
\[(\text{resp. } F_* x_{0,i+1}\equiv ([y^q]-[y^{1/q}] )x_{1,i}+x_{0,i}+[y^{1/q}] x_{1,i}\equiv -\epsilon a x_{1,i} \not\equiv 0 \pmod{\tilde{V}\tilde{M}_{i+1}})\]
pour le choix convenable de $y$.
  
\end{proof}

De manière similaire,  on peut aussi étudier le fibré équivariant 
\[
U\subset \hat{\H} \mapsto (M_{\XG|_U} )/ ( F M_{\XG|_U} ) .  
\]
Ce dernier s'identifie au faisceau des formes différentielles invariantes $\omega_{\XG^{\vee}}= \omega_{\XG^{\vee},0}\oplus \omega_{\XG^{\vee},1}$ du dual de Cartier
  du module universel sur $\hat{\H}$.  Un raisonnement similaire permet de caractériser entièrement la classe de ces parties isotypiques dans $\pic_{[G]_2}(\hat{\H})$.

\begin{theo}
Soit $s\in  \BC\TC_0^{(i)}$,    les flèches    $\Pi^*:\omega_{\XG^{\vee},i}|_{\P_s} \to \omega_{\XG^{\vee}, i+1} |_{\P_s}$ et $V^*:  \omega_{\XG^{\vee},i}|_{\P_s}\to\omega_{\XG^{\vee},i+1}^{(q)}|_{\P_s}$  sont nulles et 
\begin{align*}
 \omega_{\XG^{\vee}, i} |_{\P_s}/\Pi^*\omega_{\XG^{\vee},i+1}|_{\P_s}& \cong \bigoplus_{y \in  \P_s(\F)}   \iota_{y,*} \bar{\F}  \\
\omega_{\XG^{\vee},i}^{(q)}|_{\P_s}  / V^* \omega_{\XG^{\vee},i+1}|_{\P_s} &\cong   \bigoplus_{y \in  \P_{s}(\F_{(2)})} \iota_{y,*} \bar{\F}.      
\end{align*}
\end{theo}

\begin{coro}\label{coroordreomega2}

On a 
\begin{enumerate}

\item $\ord_{s_0}(\omega_{\XG^{\vee},0})=q $ et    $\ord_{s_0}(\omega_{\XG^{\vee}, 1})= -1$. 
 
\item  $\ord_{s_1}(\omega_{\XG^{\vee},0})=  -1$  et $\ord_{s_1}(\omega_{\XG^{\vee},1})= q$.
\end{enumerate}
\end{coro}

\begin{rem}
Le résultat précédent montre que $\omega_{i}$ et $\omega_{\XG^{\vee},i+1}$ sont isomorphes en tant que fibrés en droites.  Cela été prouvée  quand $K=\Q_p$ dans \cite[Corollary 6.7.]{pan} en exhibant une polarisation,  montrant ainsi que l'isomorphisme est équivariante.   Nous allons  en donner une preuve dans \ref{propdualomega} 
\end{rem}

\subsection{Formes différentielles logarithmiques et morphisme de  Kodaira-Spencer}

Dans cette section, nous  étudions le faisceau $\Omega^1(\log) / \hat{\H}$    des formes différentielles   à pôles logarithmiques sur les intersections de composantes irréductibles de $\bar{\H}$.  L'action naturelle de $\gln_2(K)$ sur $\hat{\H}$ induit une action sur $\Omega^1(\log)$

\begin{prop}\label{propordreomegalog}
Pour tout sommet $s\in \BC \TC_0$,  on a $\ord_s \Omega^1(\log)=q-1$.  
\end{prop}
\begin{proof}
Par définition,  $\Omega^1(\log)|_{\P_s}= \Omega^1_{\P_s}(D)\cong \Of_{\P_s}(-2)\otimes \Of_{\P_s}(q+1)\cong \Of(q-1)_{\P_s}$,  où $D= \sum_{y\in \P_s(\F)} y$. 
\end{proof}

Nous souhaitons maintenant relier le faisceau $\Omega^1(\log)$ à $\omega_0$ et $\omega_1$. 

\begin{theo}[Isomorphisme de Kodaira-Spencer]\label{theoks}
On a un isomorphisme $ \omega_0\otimes \omega_1 \to \Omega^1(\log)$ dans $\pic_{G}(\hat{\H})$.   
\end{theo}
\begin{proof}

D'après  \ref{coroOrdreOmega}  et \ref{propordreomegalog}, les fibrés $\omega_0\otimes \omega_1$ et $\Omega^1(\log)$ sont isomorphes quand on oublie l'action de $G$. On a en particulier un isomorphisme $G$-équivariant $\omega_0\otimes \omega_1 \iso \Omega^1(\log)(\chi)$ pour un certain caractère $\chi$ de $G$. En appliquant \ref{RemPullbackPeriodes}, on en déduit un isomorphisme $G$-équivariant en fibre générique $\Of(-1)^{2}\iso\Of(-2)(\chi)$ qui montre que le caractère $\chi$ est trivial.

\end{proof}

Le résultat suivant tente de rendre un peu plus naturel l'isomorphisme précédent (modulo un isomorphisme $\omega_{i} \cong \omega_{\XG^{\vee},i+1}^{\vee}$ dans $\pic_{[G]_2}(\hat{\H})$).

\begin{prop}\label{propdualomega}
Si $i=0,1$,  on a encore $\omega_{i} \cong \omega_{\XG^{\vee},i+1}^{\vee}$ dans $\pic_{[G]_2}(\hat{\H})$. 
\end{prop}
\begin{proof}
D'après \ref{coroordreomega2} et \ref{coroOrdreOmega}, on a $\omega_{i} \cong \omega_{\XG^{\vee},i+1}^{\vee}(\chi)$ pour un certain caractère de $[G]_2$. Il suffit alors de construire un morphisme non-nul $[G]_2$-équivariant entre $\omega_{i}$ et $\omega_{\XG^{\vee},i+1}^{\vee}$. 

\'Ecrivons $\XG_U$ la restriction du module universel sur $U$, et $M_U$ le module de Cartier associé.     La connexion de Gauss-Manin fourni un morphisme 
\[
\nabla: M_U / F V M_U \to  M_U/ FV M_U \otimes \Omega^1(\log). 
\]
On a aussi la filtration de Hodge 
\[
0\to M_U/FM_U \xrightarrow{V} M_U/FV M_U \xrightarrow{pr_0} M_U/V M_U \to 0 
\]
concentrée en degrés $[0,1]$.  On en déduit par composition  le morphisme de Kodaira-Spencer
\[
KS:  M_U/FM_U \xrightarrow{ (pr_0\otimes id) \circ \nabla\circ V}   M_U/VM_U \otimes \Omega^1(\log).
\]
En globalisant  et en fixant une partie isotypique,  on en déduit un morphisme de faisceaux 
\[
KS: \omega_{\XG^{\vee},i} \to  \omega_{\XG,i}^{\vee}\otimes \Omega^1(\log)\iso \omega_{\XG,i+1}.   
\]
Cette flèche est l'application recherchée et il suffit de prouver qu'elle est non-nulle et se restreindre à un ouvert $\bar{\H}_s$ de $\bar{\H}$. Rappelons que nous avons construit dans \ref{theopsi} explicitement l'application de périodes $
\tilde{\psi}: \bar{\H }_{s}\to \P^1$ qui identifie $\bar{\H}_{s}$ à l'ouvert $\P_{s}\backslash \P_{s}(\F)$. D'après la remarque \ref{RemPullbackPeriodes},  on a   les identifications   
\[
M_{1,\bar{\H}_{s}}/FV M_{1,\bar{\H}_{s}} = \tilde{\psi}^*( \mathrm{St}), \;\;  M_{0,\bar{\H}_{s}}/FM_{1,\bar{\H}_{s}}= \tilde{\psi}^*(\Of(-1)) \et M_{1,\bar{\H}_{s}}/VM_{0,\bar{\H}_{s}} = \tilde{\psi}^*(\Of(1)). 
\]
Il   suffit alors de vérifier que l'application $KS$ associée à ce dernier est un isomorphisme, qui est l'isomorphisme naturel   $\Of(-2)\xrightarrow{\sim} \Omega^1_{\P^1}$. 
\end{proof}

\subsection{Fibrés équivariants provenant du premier revêtement \label{ssecrayn}}

Comme il a été décrit dans \ref{ssecex}, l'uniformisante $\Pi_{D}\in D^{\times}$ agit par isogénies sur le module formel universel $\XG/\hat{\H}$ et nous nous intéressons dans cette section aux  fibrés en droites construit à partir du groupe $\XG[\Pi_D]=\ker(\Pi_{D})$ fini et plat sur $\hat{\H}$. Ce dernier admet une action de $\OC_D/ \Pi_D$ qui en fait un schéma en $\F_{(2)}$-espaces vectoriels. De plus, le caractère spécial de $\XG$  entraîne que $\XG[\Pi_D]$ est un schéma de Raynaud \cite[§2.3.1]{wa}. Plus précisément, l'idéal d'augmentation $\IC$ de $\XG[\Pi_D]$ admet une décomposition le long des caractères de $\F_{(2)}^*$ dans $\OC_{\breve{K}}^*$, $\IC= \bigoplus_{\chi} \Lf_{\chi}$ et la condition d'être un schéma de Raynaud est équivalente au fait que chacun de ces faisceaux $\Lf_{\chi}$ est localement libre de rang $1$ sur $\hat{\H}$.   Parmi tous ces caractères, nous en distinguons deux  $\chi_0$ et $\chi_1$ qui correspondent aux deux  plongements d'anneaux \footnote{où on a prolongé $\chi_0$ et $\chi_1$ par  $0$ en $0$.} $\F_{(2)}\to \OC_{\breve{K}}/\varpi =\bar{\F}$ et nous noterons  $\Lf_0$ et $\Lf_1$ les deux parties isotypiques associées.  D'après \cite[Proposition 1.3.1]{Rayn}, toutes les autres parties isotypiques peuvent s'écrire comme des produits de $\Lf_0$, $\Lf_1$.

De plus, la multiplication et la comultiplication sur $\XG[\Pi_D]$ induisent  des morphismes 

\[
\begin{gathered}
d_{\chi,\chi'} : \Lf_{\chi}\otimes \Lf_{\chi'} \to \Lf_{\chi\chi'}\\
c_{\chi,\chi'} : \Lf_{\chi\chi'}\to \Lf_{\chi}\otimes \Lf_{\chi'}
\end{gathered}
\] 
et cela induit des morphismes par associativité de la structure de space vectoriel 
\[
\begin{gathered}
d_{i} : \Lf_{i}^{\otimes q} \to \Lf_{i+1}\\
c_{i} : \Lf_{i+1}\to  \Lf_{i}^{\otimes q}
\end{gathered}
\]
car $\chi_{i+1}=\chi_i^{q}$.  Le théorème principal de la théorie de Raynaud montre que le schéma de Raynaud $\XG[\Pi_D]$ est entièrement déterminé par la donnée de $\Lf_0$, $\Lf_1$, $d_i$ et $c_i$ ($i=0,1$).

\begin{theo}[\cite{Rayn} théorème 1.4.1]
Soit $X/S$ un schéma de Raynaud, écrivons encore $\IC=\bigoplus_{\chi} \Lf_{\chi} $ la décomposition en parties isotypiques de l'idéal d'augmentation.  Le foncteur 
\[
X\mapsto (\Lf_0, \Lf_1, d_i, c_i)
\]
établit une équivalence des catégories entre les schémas de Raynaud sur $S$ et le tuples données par des fibrés en droites $\Lf_0$ et $\Lf_1$, et des morphismes $d_i: \Lf_{i}^{\otimes q}\to \Lf_{i+1}$ et $c_i: \Lf_{i+1}\to \Lf_{i}^{\otimes q}$, tel que $d_i \circ c_i= w\mathrm{id}$ avec $w\in \OC_{\breve{K}}$  une constante indépendante de $S$ et de $X$ 
\end{theo}

Le but de cette section est de décrire les faisceaux $[G]_2$-équivariant $\Lf_0$, $\Lf_1$ 

\begin{theo}\label{theoordli}
 On a 
 \[
 \begin{gathered}
 \ord_{s_0}(\Lf_0)= 1, \;\; \ord_{s_1}(\Lf_0)=-1  \\
  \ord_{s_0}(\Lf_1)= -1, \;\; \ord_{s_1}(\Lf_1)=1 
 \end{gathered}
 \]
\end{theo}
Le résultat découle de ces deux assertions:



\begin{prop}\label{propinv}
$\Lf_0$ et $\Lf_1$ sont inverse l'un de l'autre dans $\pic(\hat{\H})$. 
\end{prop}  
\begin{proof}
Cela découle de \cite{pan} quand $K=\Q_p$ et de \cite[Corollaire 6.8.]{J3} quand $K$ est quelconque (et $d=1$ ici).   
\end{proof}

\begin{prop}
  Soit $s\in \BC\TC_0^{(j)}$, on a 
  \[
  \Lf_i/ d_{i+1}(\Lf_{i+1}^{\otimes q})|_{\P_s}   = \begin{cases}
   \bigoplus_{y\in \P_s(\F)} \iota_{y,*} \bar{\F} & \si i=j ,\\ 
   0 & \sinon. 
\end{cases}
  \]
\end{prop}
\begin{proof}
Cela revient à appliquer \ref{theoOrdOmegai} au lemme suivant 
\begin{lem}
On a d'une part 
\[
\omega_{\XG[\Pi]/\hat{\H}} = \IC/\IC^2= \Lf_0/d_1(\Lf_1^{\otimes q}) \oplus \Lf_1/\Pi d_0(\Lf_0^{\otimes q})
\]
et d'autre part
\[
\omega_{\XG[\Pi]/\hat{\H}} = \omega_{\XG/\hat{\H}}/\Pi \omega_{\XG/\hat{\H}} = \omega_0/\Pi \omega_1 \oplus \omega_1/\Pi\omega_0.
\]
\end{lem}
\begin{proof}
Voir \cite[Proposition 3.8]{J3}. 
\end{proof}
\end{proof}

\begin{proof}[Démonstration du Théorème \ref{theoordli}]
Les deux résultats précédents fournissent  le système d'équations
\[
\begin{gathered}
\ord_{s_j}(\Lf_i)=-\ord_{s_j}(\Lf_{i+1})\\
\ord_{s_{i}}(\Lf_i)-q\ord_{s_{i}}(\Lf_{i+1})= q+1
\end{gathered}
\]
pour $i,j\in \{0,1\}$. Résoudre ce système donne la solution. 
\end{proof}

Nous avons exhibé des générateurs de $\pic_{[G]_2}(\hat{\H})$, $\pic_{[G]_2}(\bar{\H})$ et nous avons calculé leurs ordres en les différents sommets de l'arbre $\BC\TC$. Cette connaissance nous permet de déterminer toutes  les relations entre ces générateurs mais seulement un caractère près. Le résultat suivant les détermine quand $p\neq 2$ et prolonge aussi \ref{propinv}. On rappelle que l'on a défini dans \eqref{eq:leg} le caractère Legendre $\leg$ de $[G]_2$ 

\begin{prop}\label{propinvchar}

Si $p\neq 2$, on a dans $\pic_{[G]_2}(\hat{\H})$.  
\[\omega_1\cong\omega_0\otimes\Lf_0^{q+1}(\leg)\et \Lf_0\cong\Lf_1^{-1}(\leg).\]
\end{prop}

\begin{proof}
On vérifie sans peine par construction des fibrés étudiés $\omega_0$, $\omega_1$, $\Lf_0$, $\Lf_1$ que l'élément $\varpi\id$ agit trivialement sur ces fibrés. La relation sur les ordres des fibres modulaires (voir  \ref{coroOrdreOmega}, \ref{theoordli} et \ref{propordreomegalog}) établit l'existence de caractères $\chi_1$, $\chi_2$ et d'isomorphismes \[\omega_1\cong\omega_0\otimes\Lf_0^{q-1}(\chi_1)\et \Lf_0\cong\Lf_1^{-1}(\chi_2)\] Le but est de déterminer $\chi_1$, $\chi_2$ . Pour cela on considère $\omega_0\otimes\Lf_0^{(q+1)/2}$ qui a pour ordre $(q-1)/2$ en tous les sommets de l'arbre $\BC\TC$. On en déduit l'existence d'un caractère $\chi_3$ et d'un isomorphisme $\omega_1\otimes\Lf_1^{(q-1)}\cong w^{-1}\omega_0\otimes\Lf_0^{(q+1)/2}\cong \omega_0\otimes\Lf_0^{(q+1)/2}(\chi_3))$. En appliquant deux fois $w^{-1}$ à $\omega_0\otimes\Lf_0^{(q+1)/2}$, on a \[\omega_0\otimes\Lf_0^{(q+1)/2}\cong \omega_0\otimes\Lf_0^{(q+1)/2}(\chi_3^2)\] Par connexité géométrique de $\hat{\H}$, le caractère $\chi_3$ est  à valeurs dans $\{-1,1\}$ et est donc $1$ ou $\leg$. En reprenant les arguments de la preuve de \ref{proppoids}, on trouve deux entiers $r_0$, $r_1$ dans $\Z/(q-1)\Z$ pour lesquels \[\bar{\Lf}(\P_{s_i})\cong \sym^{(q-1)/2}(\bar{\F}^2)\otimes {\det}^{\otimes r_i}.\] D'après la relation \eqref{eq:r0r1}, ces entiers sont différents modulo $(q-1)$ et cela prouve que le caractère $\chi_3$ ne peut être trivial alors que l'élément $\varpi\id$ agit trivialement. Ainsi, on a $\chi_3=\leg$ et \[\omega_1\cong\omega_0\otimes\Lf_0^{q+1}(\leg).\]

Passons à la deuxième identité. Il suffit de raisonner en fibre générique. L'application de multiplication $d_i:\Lf_{i+1}^q[1/\varpi]\to\Lf_i[1/\varpi]$ est un isomorphisme sur $\H$ et on en déduit \[\Of(\chi_2)\cong\Lf_0^{q+1}[1/\varpi]\] D'après le premier point, on a aussi $\Lf_0^{q+1}[1/\varpi]\cong \omega_1[1/\varpi]\otimes\omega_0[1/\varpi](\leg)\cong \Of(leg)$. D'où, $\chi_2=\leg$ ce qui termine l'argument.
\end{proof}

\section{Cohomologie mod $p$  des  fibrés équivariants}

\subsection{Théorèmes d'annulation}
Nous étudions maintenant la cohomologie cohérente des fibrés de $\Lf\in\pic_{[G]_2}(\hat{\H})$. Cela fournit des $\OC_{\breve{K}}$-représentations  de $[G]_2$. Nous pouvons aussi étudier la cohomologie cohérente du fibré induit en fibre spéciale  $\bar{\Lf}\in\pic_{[G]_2}(\bar{\H})$ qui fournira des représentations mod $p$. Le résultat suivant relie les représentations obtenues mod $p$ avec leur analogue en niveau entier. 
\begin{prop}\label{propannul}
On a les points suivants 
\begin{enumerate}
\item Si $\hhh^{0}(\bar{\H}, \overline{\Lf})=0$, alors $\hhh^{0}(\hat{\H}, \Lf)=0$ et $\hhh^{1}(\hat{\H}, \Lf)$ est plat sur $\OC_K
$. 
\item  Si $\hhh^{1}(\bar{\H}, \overline{\Lf})=0$, alors $\hhh^{1}(\hat{\H}, \Lf)=0$.
\end{enumerate}

Dans les deux cas, on a $\hhh^{j}(\bar{\H}, \overline{\Lf})=\hhh^{j}(\hat{\H}, {\Lf})/\varpi\hhh^{j}(\hat{\H}, {\Lf})$ pour $j=0,1$.
\end{prop}

\begin{proof}
L'argument suit \cite[Theo. 2.1]{GK3}. On a une suite exacte longue 
\[
0 \to \hhh^0(\hat{\H}, \Lf)\fln{\times \varpi}{}  \hhh^0(\hat{\H}, \Lf) \to  \hhh^0(\bar{\H}, \overline{\Lf})\to  \hhh^1(\hat{\H}, \Lf)\fln{\times \varpi}{} \hhh^1(\hat{\H}, \Lf)\to \hhh^1(\bar{\H}, \overline{\Lf})\to 0 .
\]

L'annulation de $\hhh^0(\bar{\H}, \overline{\Lf})$ entraîne la surjectivité de $ \hhh^0(\hat{\H}, \Lf)\fln{\times \varpi}{}  \hhh^0(\hat{\H}, \Lf) $, l'injectivité de $ \hhh^1(\hat{\H}, \Lf)\fln{\times \varpi}{} \hhh^1(\hat{\H}, \Lf)$ et donc  la platitude de $\hhh^1(\hat{\H}, \Lf)$. Mais $ \hhh^0(\hat{\H}, \Lf)$ s'injecte dans $\hhh^0(\hat{\H}_{a}, \Lf)\cong\Of(\hat{\H}_a)$, qui n'admet aucun élément $\varpi$-divisible, d'où l'annulation de $\hhh^{0}(\hat{\H}, \Lf)$. 

Passons au deuxième point, on observe $\Lf=\varprojlim_t \Lf/\varpi^t$. La suite spectrale de Leray nous donne 
\[
E^{r,s}_2= \rrr^r \varprojlim_{t} \hhh^{s}(\hat{\H}, \Lf/\varpi^t)  \Rightarrow \hhh^{r+s}(\hat{\H}, \Lf)
\] 
d'où la suite exacte  courte par dégénérescence (voir \cite[§1]{jen})
\[
0 \fl \rrr^1 \varprojlim_{t} \hhh^0(\hat{\H}, \Lf/\varpi^t) \fl \hhh^1(\hat{\H}, \Lf) \fl \varprojlim_{t} \hhh^1(\hat{\H}, \Lf/\varpi^t)
\]
Mais, on dispose d'une suite exacte courte 
\[
0 \fl \overline{\Lf} \fln{\times \varpi^{t}}{} \Lf/\varpi^{t+1} \fl \Lf/\varpi^t \fl 0
\]
L'hypothèse d'annulation $\hhh^1(\bar{\H}, \overline{\Lf})=0$ entraîne la surjectivité de $\hhh^{0}(\hat{\H}, \Lf/\varpi^{t+1})\fl \hhh^0(\hat{\H}, \Lf/\varpi^{t})$ et la bijectivité  de $\hhh^1(\hat{\H}, \Lf/\varpi^{t+1})\fl \hhh^1(\hat{\H}, \Lf/\varpi^{t})$ (la cohomologie en degré supérieur à $2$ est nulle pour les courbes). On en déduit que  $ \rrr^1 \varprojlim_{t} \hhh^0(\hat{\H}, \Lf/\varpi^t)=0$ et \[\varprojlim_{t} \hhh^1(\hat{\H}, \Lf/\varpi^t)=\hhh^1(\bar{\H}, \overline{\Lf})=0,\] d'où $\hhh^1(\hat{\H}, \Lf) =0$. 

\end{proof}

Pour pouvoir appliquer ce résultat, nous aurons besoin de contrôler la cohomologie en fibre spéciale. Avant d'énoncer le résultat, nous introduisons quelques représentations. Les représentations $(\sym^k \bar{\F}^2)\otimes{\det}^r$ peuvent être vues par inflation comme des $G^{\circ}\varpi^\Z$-modules et on note\footnote{voir \eqref{eq:w} pour la notation $((\sym^k \bar{\F}^2)\otimes \chi)^w$}
\[\sigma^{(1)}_{k,r}= \ind_{G^{\circ}\varpi^\Z}^{[G]_2}(\sym^k \bar{\F}^2)\otimes{\det}^r \et \sigma^{(0)}_{k,r}= \ind_{wG^{\circ}w^{-1}\varpi^\Z}^{[G]_2}((\sym^k \bar{\F}^2))^w\otimes {\det}^r\] \[\tau^{(1)}_{k,r}= \ind_{G^{\circ}\varpi^\Z}^{[G]_2}(\sym^k \bar{\F}^2)^\vee\otimes{\det}^r \et \tau^{(0)}_{k,r}= \ind_{wG^{\circ}w^{-1}\varpi^\Z}^{[G]_2}((\sym^k \bar{\F}^2)^\vee)^w\otimes {\det}^r\]
si $k\ge 0$ et $\sigma^{(0)}_{k,r}=\sigma^{(1)}_{k,r}=\tau^{(0)}_{k,r}=\tau^{(1)}_{k,r}=0$ sinon. Rappelons que l'on a défini $\mu_{n,m}$ pour $n,m\in\Z$ un caractère  du Borel $B$ (voir \eqref{eq:mu}) et écrivons \[I(m,n):=\ind_{I\varpi^\Z}^{[G]_2} \mu_{n,m}\] pour $n,m\in\Z$. En suivant les arguments de \ref{proppoids}, on peut toujours construire des morphismes de représentations \[f_0:\sigma^{(0)}_{k,r}\to I(k+r,r), \ \ f_1:\sigma^{(1)}_{k,r}\to I(r,r+k).\]

\begin{lem}\label{lemcomb}

Soit  $\bar{\Lf}\in\pic_{[G]_2}(\bar{\H})$ de quadruplet\footnote{On peut toujours se ramener à ce cas quitte à tordre par un $\delta_a$ voir \eqref{eq:del} bien choisi.\label{footndel}} $(1,r,k_0,k_1)$ (voir \ref{coromodspe}),   on a une suite exacte de $[G]_2$-représentation :
\[0\to \bar{\Lf}(\bar{\H})\to \sigma^{(0)}_{k_0,r+k_1}\oplus \sigma^{(1)}_{k_1,r}\to I(r,r+k_1)\to \hhh^1(\bar{\H}, \overline{\Lf})\to \tau^{(0)}_{-2-k_0,-r-k_1} \oplus \tau^{(1)}_{-2-k_1,-r}\to 0.\]

\end{lem}

\begin{proof}

On a la suite exacte de faisceaux\footnote{ La surjectivité dans la suite exacte provient du fait que $\iota_{s,s',*}\Lf|_{\P_s\cap \P_{s'}}$ est un faisceau gratte-ciel dont le support est $\P_s\cap \P_{s'}$ et $\BC\TC$ est un arbre.} :
\[
0 \fl \overline{\Lf} \fl \prod_{s\in \BC \TC_0} \iota_{s,*}  \overline{\Lf}|_{\P_s}  \to \prod_{a=(s,s')\in \BC\TC_1} \iota_{s,s',*} \overline{\Lf}|_{\P_s\cap \P_{s'}} \fl 0
\] 
où $\iota_s: \P_s \fl \bar{\H}$ est l'immersion fermée de la composante $\P_s$ dans la fibre spéciale, idem pour $\iota_{s,s'}$.  On en déduit une suite exacte longue 
\[
0 \fl \overline{\Lf}(\bar{\H})  \fl \prod_{s\in \BC\TC_0} \overline{\Lf}(\P_s) \fl \prod_{a=(s,s')\in \BC\TC_1} \overline{\Lf}(\P_s\cap \P_{s'}) \fl \hhh^1(\bar{\H}, \overline{\Lf})\fl \prod_{s\in \BC\TC_0} \hhh^1(\P_s, \overline{\Lf}) \fl 0.
\]
D'après la description des fibrés équivariants sur $\P^1$ et de leur cohomologie (voir \ref{corocohoequiv}), on a 
\[\bar{\Lf}(\P_{s_1})\cong \sym^{k_1} \bar{\F}^2\otimes{\det}^r, \bar{\Lf}(\P_{s_0})\cong(\sym^{k_0} \bar{\F}^2\otimes{\det}^{r+k_1})^w\et \overline{\Lf}(\P_{s_0}\cap \P_{s_1})\cong \mu_{r,r+k_1}\]
\[\hhh^1(\P_{s_1}, \overline{\Lf})\cong (\sym^{-2-k_1} \bar{\F}^2)^\vee\otimes{\det}^{-r}\et \hhh^1(\P_{s_0}, \overline{\Lf})\cong ((\sym^{-2-k_0} \bar{\F}^2)^\vee)^w\otimes{\det}^{-r-k_1}. \] 
 En suivant les orbites et les stabilisateurs de l'action de $[G]_2$ sur les composantes irréductibles, on obtient des isomorphismes
\begin{gather*}
\prod_{s\in \BC\TC_0} \overline{\Lf}(\P_s)\cong \ind_{G^{\circ}\varpi^\Z}^{[G]_2} \bar{\Lf}(\P_{s_1}) \oplus \ind_{wG^{\circ}w\varpi^\Z}^{[G]_2} \bar{\Lf}(\P_{s_0})\cong\sigma^{(0)}_{k_0,r+k_1}\oplus \sigma^{(1)}_{k_1,r}\\
\prod_{a=(s,s')\in \BC\TC_1} \overline{\Lf}(\P_s\cap \P_{s'}) \cong \ind_{I\varpi^\Z}^{[G]_2} \bar{\Lf}(\P_{s_1}\cap\P_{s_0})\cong I(r,r+k_1)\\
\prod_{s\in \BC\TC_0} \hhh^1(\P_s, \overline{\Lf})\cong  \ind_{G^{\circ}\varpi^\Z}^{[G]_2} \hhh^1(\P_{s_1}, \overline{\Lf}) \oplus \ind_{wG^{\circ}w\varpi^\Z}^{[G]_2} \hhh^1(\P_{s_0}, \overline{\Lf})\cong\tau^{(0)}_{-2-k_0,-r-k_1} \oplus \tau^{(1)}_{-2-k_1,-r}
\end{gather*}
ce qui conclut l'argument.


\end{proof}

\begin{coro}
\label{coroAnnulation}
Soit $\Lf$ un fibré en droites $[G]_2$-équivariant sur $\hat{\H}$, on a les points suivantes 
\begin{enumerate}

\item Si $\ord_{s_0} (\Lf)<0$  et $\ord_{s_1}(\Lf)<0$, alors $\hhh^0(\hat{\H}, \Lf)=0$ et $\hhh^1(\hat{\H}, \Lf)$ est plat sur $\OC_{\breve{K}}$.

\item Si $\ord_{s_0}(\Lf)\geq -1$ et $\ord_{s_1}(\Lf) \geq -1$, alors  $\hhh^1(\hat{\H}, \Lf)=0$.

\item Si $0\leq \ord_{s_i}(\Lf)<q+1$ et $\ord_{s_{i+1}}(\Lf) <0$ pour un certain $i=0,1$, alors $\hhh^0(\hat{\H}, \Lf)=0$ et $\hhh^1(\hat{\H}, \Lf)$ est plat sur $\OC_{\breve{K}}$.  

\end{enumerate}
\end{coro}

\begin{proof}

On justifie que l'on peut appliquer \ref{propannul} pour les trois cas.

Dans le premier cas, on a $\overline{\Lf}(\P_{s_0})=\overline{\Lf}(\P_{s_1})=0$ d'où $\hhh^0(\P_s, \overline{\Lf})=0$. 

Dans le second cas,  $\hhh^1(\P_{s_0}, \overline{\Lf})=\hhh^1(\P_{s_1}, \overline{\Lf})=0$ et $\overline{\Lf}(\P_s)\neq 0$. En particulier, $\overline{\Lf}(\P_s)\fl \overline{\Lf}(\P_s\cap \P_{s'})=\bar{\F}$ est surjective, ce qui nous donne la surjectivité de 
\[
 \prod_{s\in \BC\TC_0} \overline{\Lf}(\P_s) \fl \prod_{a=(s,s')\in \BC\TC_1} \overline{\Lf}(\P_s\cap \P_{s'})
\]
par contractibilité de $\BC\TC$. 

Supposons la troisième hypothèse et montrons $\overline{\Lf}(\bar{\H})=0$. Pour cela, prenons une section globale $f\in \overline{\Lf}(\bar{\H})$ et restreignons-la en  $f_s$ sur  $\P_s$.  Par hypothèse, on a  $f_s= 0$ pour $s$ est dans l'orbite de $s_{i+1}$.  Si $s$ est dans l'orbite de $s_i$, $f_s$ doit s'annuler sur l'intersection de $\P_s$ avec les autres composantes qui sont toutes dans l'orbite de $s_{i+1}$. Ainsi, $f_s$ se voit comme une section de $\Of_{\P_s}(\ord_{s_i}(\Lf)-1-q )$. Comme $\ord_{s_i}(\Lf)-1-q<0$,  $f$ s'annule sur chaque sommet, elle s'annule globalement. 
 
\end{proof}

\begin{coro}
On a $H^j(\hat{\H},  \omega_i)=  H^j(\hat{\H}, \Lf_i)=0$ pour $i$, $j$ dans $\{0, 1\}$.  
\end{coro}
\begin{proof}
Il suffit d'appliquer les énoncés   \ref{coroOrdreOmega}, \ref{theoordli} et \ref{coroAnnulation}.  
\end{proof}

Au vu du résultat précédent, il est naturel d'introduire les notions suivantes.

\begin{defi}\label{defiposneg}

Soit $\Lf$ un élément de $\pic_{[G]_2}(\hat{\H})$ ou de $\pic_{[G]_2}(\bar{\H})$, nous dirons qu'il est positif (respectivement négatif, mixte) si les signes des entiers $\ord_{s_i}(\Lf)$(pour $i=0,1$) sont positifs (respectivement négatifs, opposés).

\end{defi}

\subsection{Filtrations de la cohomologie mod $p$}

Soit  $\bar{\Lf}\in\pic_{[G]_2}(\bar{\H})$ de quadruplet $(a,r,k_0,k_1)$ (voir \ref{coromodspe} et ), on pourra noter pour simplifier $r_1=r$ et $r_0=r+k_1$. Le but de cette section est de montrer  :

\begin{theo}\label{theofilt}

Soit  $\bar{\Lf}\in\pic_{[G]_2}(\bar{\H})$ de quadruplet $(1,r,k_0,k_1)$ (voir \ref{footndel}) et $r_0$, $r_1$ comme ci-dessus,   on   a les points suivants :
\begin{enumerate}
\item  Si $k_i\ge q+1$ pour $i=0,1$, on a  des suites exactes :  \[0\to \sigma^{(0)}_{k_0-q-1,1+r_0}\oplus \sigma^{(1)}_{k_1-q-1,1+r_1}\to \hhh^0(\bar{\H},\bar{\Lf})\to I(r_1,r_0)\to 0\]
\item  Si $0\le k_i\le q$ pour $i=0,1$, on a   :  \[ \hhh^0(\bar{\H},\bar{\Lf})\cong\sigma^{(0)}_{k_0,r_0}\cap \sigma^{(1)}_{k_1,r_1}\subset I(r_1,r_0),\] 
\[\hhh^0(\bar{\H},\bar{\Lf})^\vee\cong  I(r_1,r_0)^\vee/\ker(f_0^\vee)+\ker(f_1^\vee)\] avec $f_i=\sigma^{(i)}_{k_i,r_i}\to I(r_1,r_0)$.
\item  Si $0\le k_i\le q$ pour un certain  $i=0,1$ et $k_{i+1}\ge q+1$, on a  des  suites exactes :  \[0\to \sigma^{(i+1)}_{k_{i+1}-q-1,1+r_{i+1}}\to \hhh^0(\bar{\H},\bar{\Lf})\to \sigma^{(i)}_{k_i,r_i}\to 0\]
\item  Si $k_i\le -1$ pour $i=0,1$, on a  des suites exactes :  \[0\to I(r_1,r_0)\to \hhh^1(\bar{\H}, \overline{\Lf})\to \tau^{(0)}_{-2-k_0,-r_0} \oplus \tau^{(1)}_{-2-k_1,-r_1}\to 0\]
\item  Si $k_i\le -1$ pour un certain $i=0,1$ $k_{i+1}\le q$, on a  des suites exactes :
\[0\to I(r_1,r_0)/\sigma^{(i+1)}_{k_{i+1},r_{i+1}}\to \hhh^1(\bar{\H}, \overline{\Lf})\to \tau^{(i)}_{-2-k_i,-r_i} \to 0\]
\item  Si $k_i\le -1$ pour un certain $i=0,1$ et $k_{i+1}\ge q+1$, on a  des isomorphismes :  \[\hhh^0(\bar{\H},\bar{\Lf})\cong \sigma^{(i+1)}_{k_{i+1}-q-1,1+r_{i+1}},\ \hhh^1(\bar{\H}, \overline{\Lf})\cong \tau^{(i)}_{-2-k_{i},r_i}\]
\end{enumerate}%

\end{theo}

\begin{rem}

On peut rendre un peu plus naturel certaines filtrations apparaissant dans le résultat ci-dessus. Dans \ref{theoOrdOmegai}, nous avons étudié les points d'annulation de $\Pi:\omega_i\to  \omega_{i+1}$. Grâce à ce résultat, on peut voir que les termes de la forme $\sigma^{(j)}_{k_{j}-q-1,1+r_{j}}$ apparaissant dans les points $1$, $3$, $6$ sont exactement images de la flèche induite par $\Pi$ \[\hhh^0(\bar{\H},\bar{\Lf}\otimes \omega_j\otimes\omega_{j+1}^{-1})\to\hhh^0(\bar{\H},\bar{\Lf})\].

\end{rem}

\begin{proof}

Pour la description des sections globales (mod $p$) obtenus dans les points $1$, $2$, $3$, $6$, cela découle de \ref{lemcomb} et des deux lemmes suivants, tous deux appliqués à $f_i:\sigma^{(i)}_{k_i,r_i}\to I(r_1,r_0)$  :

\begin{lem}\label{leminjsur}

La flèche $f_i$ est injective quand $k_i<q+1$ et surjective de noyau $\sigma^{(i)}_{k-q-1,1+r_i}$ sinon.

\end{lem}

\begin{lem}\label{lemgenmod}

Étant donné trois modules $M_1$, $M_2$, $M_3$ sur un anneau $A$, deux morphismes $A$-linéaires  $(f_i:M_i\to M_3)_{i=1,2}$, on a  un  une suite exacte
\[0\to\ker f_1\oplus\ker f_2 \to \ker(f_1+f_2)\to \imm f_1\cap \imm f_2\to 0\]pour $f_1+f_2:M_1\oplus M_2\to  M_3$ la flèche naturelle.

\end{lem}

\begin{proof}[Démonstration de \ref{leminjsur}]
Par \eqref{eq:ssjac} et la description de $\sym^k\bar{\F}^2$ en tant qu'espaces de polynômes homogènes On peut trouver des isomorphismes d'espaces vectoriels qui rendent le diagramme suivant commutatif
\[
\begin{tikzcd}
\sigma^{(i)}_{k_i,r_i} \ar[r,"f_i"]\ar[d, "\sim"] & I(r_1,r_0)\ar[d, "\sim"] \\
\prod_{s\in \BC\TC^{(i)}} \bar{\F}[X,Y]_{i} \ar[r, "\prod_s f_s"]  & \prod_{s\in \BC\TC^{(i)}}\prod_{[a_0,a_1]\in \P_s(\F)}\frac{\bar{\F}[X,Y]_{k_i}}{(a_0X+a_1Y)\bar{\F}[X,Y]_{k_i-1}} 
\end{tikzcd}
\] 
avec $\bar{\F}[X,Y]_k$ l'ensemble des polynômes homogènes de degré $k$. On peut alors raisonner sur chacun des morphismes $f_s:\bar{\F}[X,Y]_k\to \prod_{[a_0,a_1]\in \P_s(\F)}\frac{\bar{\F}[X,Y]_{k_i}}{(a_0X+a_1Y)\bar{\F}[X,Y]_{k_i-1}}$ pour $s\in \BC\TC^{(i)}$. Si l'on choisit un générateur pour chacune des $\bar{\F}$-droites $\frac{\bar{\F}[X,Y]_{k_i}}{(a_0X+a_1Y)\bar{\F}[X,Y]_{k_i-1}}$, l'application $f_s$ peut s'interpréter comme l'évaluation des polynômes homogènes de $\bar{\F}[X,Y]_{k_i}$ en les points rationnels $\P_s(\F)$ (évaluation qui n'a de sens qu'une fois ces choix réalisés). On en déduit qu'un polynôme homogène est dans le noyau de $f_s$ si et seulement si il est divisible par $X^qY-Y^qX=c\prod_{[a_0,a_1]\in \P_s(\F)} (a_0X+a_1Y)$. Cela prouve l'injectivité dans le premier cas et le calcul du noyau dans le second (notons  a la relation $g\cdot (X^qY-Y^qX)=\det g (X^qY-Y^qX)$ pour $g\in \bar{G}$.) La surjectivité dans le second cas   se prouve grâce aux polynômes interpolateurs de Lagrange ou en observant que le noyau a la bonne dimension. 

\end{proof}

\begin{proof}[Démonstration de \ref{lemgenmod}]
Clair.
\end{proof}

Les points $4.$, $5.$ et $6.$, quant à eux, s'obtiennent par application de \ref{lemcomb} (et de \ref{leminjsur} pour le point $6.$). Enfin, la deuxième assertion du point $2.$ est la version duale de la première.
\end{proof}

\begin{rem}

Dans l'énoncé du lemme \ref{leminjsur}, nous n'avons pas décrit le conoyau de $f_i$ dans le cas où cette dernière est une injection. La proposition \ref{lemjordmod} fournit une description dans le cas particulier où le corps résiduel de $K$ est  $\F_p$. Le cas général est plus compliqué (voir \cite[Proposition 1.1.]{diam}) et nous ne nous en servirons pas.

\end{rem}

Nous allons maintenant nous concentrer sur les sections globales mod $p$ des fibrés en droite $[G]_2$-équivariants positifs (voir \ref{defiposneg}) de poids $-1$ sur $\bar{\H}$. Ces derniers correspondent aux fibrés étudiés dans le point $3.$ du théorème précédent. Pour  énoncer le résultat obtenu, nous aurons besoin de la notion de représentation irréductible supersingulière de $[G]_2$. Il s'agit des représentations irréductibles de $[G]_2$ qui sont facteurs directs de la restriction d'une représentation irréductible supersingulière de $G$ (voir et comparer les théorèmes de classification \ref{theoclassg} et \ref{theoclassgdeux} pour plus de précisions). Le résultat que nous voulons montrer est  le suivant

\begin{theo}\label{theopoidsun}

Si $K=\Q_p$, la flèche \[\bar{\Lf}\mapsto \hhh^0(\bar{\H},\bar{\Lf})^\vee\] définit une bijection entre les fibrés en droite $[G]_2$-équivariants positifs de poids $-1$ sur $\bar{\H}$ et les représentations irréductibles supersingulières de  $[G]_2$.

\end{theo}

\begin{coro}

Si $K=\Q_p$ et $\Lf$ un fibré en droite $[G]_2$-équivariant positif de poids $-1$ sur $\hat{\H}$,   la représentation $ \hhh^0(\hat{\H},{\Lf})^\vee\otimes \Q_p$ est irréductible dont la réduction mod $p$ du réseau naturel $ \hhh^0(\hat{\H},{\Lf})^\vee$ est une $\bar{\F}$-représentation irréductible supersingulière.

\end{coro}

L'étude de ces $\breve{\Q}_p$-représentations irréductibles fera l'objet d'un travail ultérieur.

\begin{rem}

D'après le point 3. de \ref{theofilt}, le théorème ci-dessus revient à prouver que la représentation qui $I(r_1,r_0)^\vee/\ker(f_0^\vee)+\ker(f_1^\vee)$ est un quotient de $\cind_{IZ}^{[G]_2}\mu_{-r_1,r_0}$ est une représentation supersingulière de $[G]_2$. Notons que pour un quotient analogue de $\cind_{IZ}^{G}\mu_{-r_1,r_0}$, cela a déjà été fait dans \cite[Theorem 1.1.]{anand} et nous pouvons voir ce résultat comme une adaptation de ce résultat.

\end{rem}

La preuve de ce  théorème occupera  les trois sections suivantes. Dans les deux  premières, nous rappelons la classification des de $\sln_2(K)$ par Abdellatif \cite{abd1} et Cheng \cite{cheng}. Dans la dernière, nous obtenons comme conséquence directe un résultat analogue pour les représentations irréductibles admissibles de $[G]_2$ (définissant au passage  la notion de représentation irréductible supersingulière pour $[G]_2$) puis  terminerons la preuve de ce résultat.

\subsection{Poids de Serre et algèbre de Hecke}

Rappelons que nous avons noté $q=p^f$ avec $\F=\F_q$ le corps résiduel de $K$. Nous avons d'abord besoin d'étudier les représentations de $\bar{G}=\gln_2(\F)$. Pour cela, écrivons et introduisons pour chaque  $f$-uplet $\vec{k}=(k_1,\cdots,k_f)\in \left\llbracket 0,p-1\right\rrbracket^f$ et $r\in\left\llbracket 0,q-1\right\rrbracket$ une représentation de $\bar{G}$ \[\sym^{\vec{k}}(\bar{\F}^2)\otimes{\det}^r :=\delta_a\otimes{\det}^r\otimes\sym^{k_1}(\bar{\F}^2)\otimes\fro^*\sym^{k_2}(\bar{\F}^2)\otimes\cdots \otimes\fro^{*,f-1}\sym^{k_f}(\bar{\F}^2)\] que l'on peut aussi voir comme une $G^\circ Z$-représentation  par inflation. Chacune de ces représentations est appelée poids de Serre. On a la classification suivante (voir par exemple \cite[Proposition 1.]{BL2}) :

\begin{theo}\label{theobarg}

Toute représentation irréductible de $\bar{G}$ et admissible de $G^\circ Z$ (où $\varpi\id$ agit trivialement) est un poids de Serre pour un unique couple $(\vec{k},r)$.

\end{theo}


Le point-clé dans l'étude qui va suivre est de comprendre la structure de la $G$-représentation $\delta_a\otimes\cind_{G^\circ Z}^G \sym^{\vec{k}}(\bar{\F}^2)\otimes{\det}^r=:\delta_a\otimes \cind_{G^\circ Z}^G \sigma$. Pour ce faire, nous allons introduire et décrire les propriétés de l'algèbre de Hecke de $\sigma$

\begin{defi}

Étant donné  $\sigma= \sym^{\vec{k}}(\bar{\F}^2)\otimes{\det}^r $ comme précédemment, nous définissons l'algèbre de Hecke comme étant 
\[\Hf(\sigma):=\en_G(\cind_{G^\circ Z}^G \sigma)\]
où le produit est donné par la composition.

\end{defi}

Fixons $\sigma= \sym^{\vec{r}}(\bar{\F}^2)\otimes\delta_a\otimes{\det}^r $ et écrivons $V_\sigma$ l'espace vectoriel sous-jacent. L'espace sous-jacent de $ \cind_{G^\circ Z}^G \sigma$ est alors $\bigoplus_{s\in\BC\TC_0=G/G^\circ Z}V_s$ où chaque $V_s$ est isomorphe (non-canoniquement) à $V_\sigma$. Si l'on voit $\bigoplus_{s\in\BC\TC_0}V_s$ comme un sous-espace de fonctions de $\homm_{\ens}(G,V_\sigma)$ on peut alors voir \[\Hf(\sigma)=\en_G(\cind_{G^\circ Z}^G \sigma)=\hom_{G^\circ Z}(\sigma,\cind_{G^\circ Z}^G \sigma|_{G^\circ Z})\] comme un sous espace de $\homm_{\ens}(V_\sigma\times G, V_\sigma)$ et même de $\homm_{\ens}(G,\en_{\bar{\F}\vectot}(V_\sigma))$ si on fixe la variable dans $G$ et on fait varier celle dans $V_\sigma$. Plus précisément, elle peut s'identifier à l'espace de fonctions $\varphi : G\to \en_{\bar{\F}\vectot}(V_\sigma)$ à support compact (modulo le centre) vérifiant
\begin{equation}\label{eq:heckphi}
\forall h_1,h_2\in G^\circ Z,\forall g\in G, \varphi(h_1 g h_2)=h_1\circ \varphi(g)\circ h_2
\end{equation} où $h_1$, $h_2$ sont vus comme des endomorphismes de $V_\sigma$. 
 La structure d'algèbre induite est donnée ici par le produit convolution. 
Nous allons nous servir de cette interprétation pour exhiber des éléments de l'algèbre de Hecke. Par la décomposition de Cartan \eqref{eq:cartan}, un élément $\varphi : G\to \en_{\bar{\F}\vectot}(V_\sigma)$ dans $\Hf(\sigma)$ est déterminé par ses valeurs en $\alpha^n:=\left( \begin{matrix}
1 &0 \\  0& \varpi^n
\end{matrix} \right)$ qui doivent être toutes nulles sauf un nombre fini. Les éléments $U_{\vec{k}}=U_{k_1}\otimes\cdots\otimes \fro^{*,f-1}U_{k_f}\in \en_{\bar{\F}\vectot}(V_\sigma)$ avec  
\[U_r : \sum_{i=0}^k a_i X^i Y^{k-i}\mapsto a_k Y^k,  \forall k\in\left\llbracket 0,p-1\right\rrbracket\]
 auront une importance fondamentale dans l'étude que nous voulons réaliser comme en témoigne le résultat suivant :

\begin{theo}\label{theoalghecke}

Soit $\sigma= \sym^{\vec{k}}(\bar{\F}^2)\otimes\delta_a\otimes{\det}^r $, on a les points suivants :
\begin{enumerate}
\item Pour tout $n$, il existe un unique $\varphi_n : G\to \en_{\bar{\F}\vectot}(V_\sigma)$ vérifiant \eqref{eq:heckphi} supporté en la double classe $G^\circ Z \alpha^n G^\circ$ qui envoie $\alpha^n$ sur $U_{\vec{k}}$ et nous noterons $T_n$ l'endomorphisme de $\cind_{G^\circ Z}^G \sigma$ associé,
\item L'ensemble des éléments $\varphi : G\to \en_{\bar{\F}\vectot}(V_\sigma)$ de $\Hf(\sigma)$ supporté en la double classe $G^\circ Z \alpha^n G^\circ$ est la droite $\bar{\F}\cdot \varphi_n$,
\item Pour $T=T_1$, on a les relations suivantes \[\forall n\ge 2, T_n =\begin{cases} T^n-T^{n-2}\si \vec{r}=0\\ T^n \sinon\end{cases}, \]
\item $\Hf(\sigma)$ est isomorphe à l'algèbre de polynômes $\bar{\F}[T]$.

\end{enumerate}

\end{theo}

\begin{proof}
La preuve de $1$ et $2$ peut être trouvé dans \cite[Lemma 7]{BL1}. Pour $3$, cela se fait par calcul direct (voir Proposition 8 de \textit{loc. cit.}). D'après $2$ et la décomposition de Cartan, on a \[\Hf(\sigma)\cong\bigoplus_{n\ge 0}T_n=\bar{\F}[T]\] où la dernière égalité est une application de $3$.
\end{proof}

Nous allons terminer cette section en étudiant l'action de $\Hf(\sigma)$ sur  $W:=\bigoplus_{s\in\BC\TC_0}V_s$ que l'on décomposera de la manière suivante : \[W=W_0\oplus W_1:=\bigoplus_{s\in\BC\TC_0^{(0)}}V_s\oplus \bigoplus_{s\in\BC\TC_0^{(1)}}V_s\] Plus précisément, pour un élément de l'algèbre de Hecke représenté par un polynôme $P(T)\in \bar{\F}[T]$, nous allons traduire certains invariants du polynôme $P$ en des  propriété de l'action de $P(T)$ sur $\bigoplus_{s\in\BC\TC_0}V_s$.  On rappelle que l'on a une notion de support pour les  éléments $v=(v_s)_{s\in\BC\TC}$ de $\bigoplus_{s\in\BC\TC_0}V_s$ voir \eqref{eq:supp} 
\begin{lem}\label{lemp}

Soit $v\in V_{s}$ pour $s=gs_1\in \BC\TC_0$ avec $s_1$ le sommet standard, $P(T)\in\bar{\F}[T]=\Hf(\sigma)$ et $\varphi : G\to \en_{\bar{\F}\vectot}(V_\sigma)$ l'application  vérifiant \eqref{eq:heckphi} associé.  on a \[\supp(P(T)\cdot v)=(g\supp \varphi )/G^\circ Z\]

\end{lem}

\begin{proof}
Soit $P(T)$, $\varphi$ et $v$ comme dans l'énoncé, on a $v=g\tilde{v}$ avec $\tilde{v}\in V_{s_1}$ et \[\supp(P(T)\cdot v)=\supp(gP(T)\cdot \tilde{v})=g\supp(P(T)\cdot \tilde{v}).\] En écrivant $(w_s)_{s\in \BC\TC_0}=P(T)\cdot \tilde{v}\in \bigoplus_{s\in\BC\TC_0}V_s$, on a alors  \[\varphi(h)(\tilde{v})=h^{-1}w_{hs_1}\] pour tout $g\in G$ d'où $\supp(P(T)\cdot \tilde{v})=\supp \varphi /G^\circ Z$.
\end{proof}
\begin{coro}\label{coropoly}

\begin{enumerate}
\item $\bar{\F}[T^2]$ est l'ensemble des éléments  de $\Hf(\sigma)$ qui préservent $W_0$, $W_1$ et   $T\bar{\F}[T^2]$ est l'ensemble des éléments qui les échangent,
\item Soit $P(T)\in\bar{\F}[T]=\Hf(\sigma)$ et $v\in V_{s}$ pour $s\in \BC\TC_0$, on a $\deg P=\max_{s'\in \supp(P(T)\cdot v)}d(s',s)$ pour $d$ la distance naturelle entre les sommets  sur le graphe $\BC\TC$.
\end{enumerate}

\end{coro}

\begin{proof}

Comme \[{G^\circ Z}\alpha{G^\circ}/{G^\circ Z}=\{s'\in\BC\TC_0: (s',s_1)\in\BC\TC_1\},\] on a $\supp(T\cdot v)=\{s'\in\BC\TC_0: (s',s)\in\BC\TC_1\}$ pour $v\in V_{s}$ d'après le résultat précédent. Donc $T$ échange les facteurs directs $W_0$, $W_1$. Ainsi, $\bar{\F}[T^2]$ préserve ces espaces  là où $T\bar{\F}[T^2]$ les échange. Comme on a $\bar{\F}[T]=\bar{\F}[T^2]\oplus T\bar{\F}[T^2]$, ces deux facteurs directs sont exactement les ensembles d'éléments qui préservent et échangent respectivement  les deux sous-espaces $W_0$, $W_1$.

Soit $P(T)$, $\varphi$ et $v$ comme dans \ref{lemp}, d'après le point $3$ de \ref{theoalghecke}, on a \[\deg P=\min \{n: P(T)\in \bigoplus_{k\le n}\bar{\F} \cdot T_k\}=\max \{n : \varphi(\alpha^n)\neq 0\}=\max_{s'\in \supp(P(T)\cdot v)}d(s',s)\] où la dernière égalité découle de  \[(\bigsqcup_{k\le n} G^\circ Z \alpha^k G^\circ)/G^\circ Z=\{s=gs_1\in \BC\TC_0 :d(s,s_1)\le n\}.\] 
\end{proof}


\subsection{Représentations irréductibles de $G$ et $\sln_2(K)$\label{ssecg}}

Nous rappelons ici la description des représentations irréductibles admissibles de réalisée par Bartel-Livné \cite{BL2, BL1}, Breuil \cite{breuil2}.... Le point de départ de cette classification provient de l'observation fondamentale suivante

\begin{theo}[\cite{BL2} Proposition 32., Corollary 31.]

Toute représentation irréductible admissible de $G$ est un  quotient d'une représentation \[V(a,r,\vec{k},\lambda):=\delta_a\otimes{\det}^r \otimes(\cind_{{G^\circ Z}}^G \sym^{\vec{k}}\bar{\F}/T-\lambda)\] pour $a\in\bar{\F}$, $r\in \left\llbracket 0,q-1\right\rrbracket$, $\vec{k}\in \left\llbracket 0, p-1\right\rrbracket^f$ et $\lambda\in\bar{\F}$.

\end{theo}


Le théorème de classification s'énonce ainsi.

\begin{theo}[\cite{BL2} Theorem 30.]\label{theoclassg}
\begin{enumerate}
\item Pour tous $(a,r,\vec{k},\lambda)$ avec $\lambda\neq 0$, la représentation $V(a,r,\vec{k},\lambda)$ admet un unique quotient irréductible $W(a,r,\vec{k},\lambda)$ qui est exactement $V(a,r,\vec{k},\lambda)$ quand $(\vec{k},\lambda)\notin\{\vec{0},\overrightarrow{p-1}\}\times\{\pm 1\}$ avec $\vec{k}=(k,\cdots,k)\in \left\llbracket 0, p-1\right\rrbracket^f$.
\item Les représentations irréductibles admissibles de $G$ peuvent être décomposé en quatre familles disjointes :
\begin{enumerate}
\item Les caractères $\psi$ de $G$ correspondant à une représentation de la forme  $W(a,r,\vec{0},\pm 1)$.
\item La série spéciale\footnote{$\stb$ pour Steinberg} $\psi\otimes (\cind_{B}^G \mathbbm{1}_B)/\mathbbm{1}_G=: \psi\otimes \stb$ qui correspond  à une représentation de la forme $W(a,r,\overrightarrow{p-1},\pm 1)$
\item La série principale  $\psi\otimes \cind_{B}^G \mu_1\otimes\mu_2$ avec $\mu_1\neq \mu_2$ isomorphe à $V(a,r,\vec{k},\lambda)$ quand  $(\vec{k},\lambda)\notin\{\vec{0},\overrightarrow{p-1}\}\times\{\pm 1\}$ et $\lambda\neq 0$.
\item Les représentations supersingulières obtenues comme quotient de $V(a,r,\vec{k},0)$.
\end{enumerate}
\end{enumerate}

\end{theo}


Parmi les quatre familles précédentes, les représentations supersingulières forment la famille la plus mystérieuse. Toutefois, elles admettent une description explicite dans lorsque $K=\Q_p$.

\begin{theo}\label{theoclassgsupsing}

Si $K=\Q_p$, la représentation $V(a,r,{k},0)$ est irréductible et on a des isomorphismes $V(a,r,{k},0)\cong V(-a,r,{k},0)\cong V(a,r+k,{p-1-k},0)\cong V(-a,r+k,{p-1-k},0)$ et ce sont les seuls isomorphismes entre les représentations supersingulières $V(a,r,{k},0)$.

\end{theo}

Plaçons nous maintenant dans le cas d'un corps $K$ général où on peut donner une description similaire aux représentations irréductibles de $\sln_2(K)$. L'énoncé suivant a été prouvé indépendamment  par Abdellatif \cite{abd1} et Cheng \cite{cheng}.

\begin{theo}\label{theoclasssl}

Les représentations irréductibles admissibles de $\sln_2(K)$ peuvent être séparées en quatre familles disjointes
\begin{enumerate}[label=(\alph*)]
\item La représentation triviale $\mathbbm{1}_{ \sln_2(K)}$
\item La représentation spéciale $\spg:=\stb|_{ \sln_2(K)}$
\item La série principale $\cind_{B\cap  { \sln_2(K)}}^{ \sln_2(K)} \mu\otimes \mathbbm{1}$
\item Les représentations  supersingulières qui sont quotients de $V(a,r,\vec{k},0)|_{ \sln_2(K)}$.
\end{enumerate}
\end{theo}

Pour les familles $(a)$, $(b)$, $(c)$ pour  $\sln_2(K)$, toutes les représentations considérées sont des restrictions de représentations de $G$ dans les familles analogues $(a)$, $(b)$, $(c)$. De plus, deux représentations irréductibles de $G$ de $(a)$, $(b)$, $(c)$ ont même restrictions si l'une est obtenue à partir de l'autre par torsion d'un caractère.

Comme dans le cas de $G$, très peu est connu sur les représentations supersingulières de  $\sln_2(K)$ sur un corps général. Quand $K=\Q_p$, ces représentations sont totalement comprises ;

\begin{theo}\label{theoclassslsupsing}

La restriction à  $\sln_2(\Q_p)$ de toutes représentations supersingulières $V(a,r,{k},0)$ de $G$ se décompose en une somme directe de deux représentations irréductibles \[V(a,r,{k},0)|_{ \sln_2(\Q_p)}\cong\pi_k^0\oplus\pi_k^1\] soumises aux isomorphismes $\pi_k^0\cong \pi_{p-1-k}^1$ qui sont  les seuls entre ces représentations. En particulier, on a exactement $p$  représentations supersingulières $(\pi_k^0)_{0\le k\le p-1}$ de $\sln_2(K)$.

\end{theo}

\begin{proof}
La preuve a été écrite dans \cite[Théorème 0.7.3.]{abd1}  \cite[Theorem 1.7.]{cheng} mais nous rappelons la construction des deux facteurs directs $\pi_k^0$, $\pi_k^1$ dont nous nous servirons dans la section suivante. Notons que l'on a la décomposition suivante (qui est celle apparaissant dans \ref{coropoly}.1.) \begin{align*}
\cind_{{G^\circ Z}}^G\sym^k \bar{\F}^2&=\cind_{G^\circ Z}^{[G]_2}\sym^k \bar{\F}^2 \oplus \cind_{wG^\circ w^{-1} Z}^{[G]_2}(\sym^k \bar{\F}^2)^w\\
&=\cind_{G^\circ\cap  \sln_2(K)}^{ \sln_2(K)}\sym^k \bar{\F}^2 \oplus \cind_{wG^\circ w^{-1} \cap{ \sln_2(K)}}^{ \sln_2(K)}(\sym^k \bar{\F}^2)^w=V_1(k)\oplus V_0(k)
\end{align*}

L'opérateur $T$ définit dans \ref{theoalghecke}.1. permutent ces deux termes \ref{coropoly}.1. ce qui fournit une décomposition $V(a,r,{k},0)|_{ \sln_2(K)}=V_1(k)/TV_0(k)\oplus V_0(k)/TV_1(k)$. Les termes recherchés  sont alors $\pi_k^i\cong V_i(k)/TV_{i+1}(k)$.
\end{proof}

\subsection{Représentations irréductibles de $[G]_2$ et preuve de \ref{theopoidsun} \label{ssecgdeux}}

Nous allons maintenant décrire les représentations irréductibles de $[G]_2$. Nous verrons la classification obtenue comme une conséquence directe des résultats des deux sections précédentes. Il est aussi possible de recopier mot pour mot la preuve pour  $\sln_2(K)$ mais le lien entre toutes ces classifications motive ces résultats. Pour pouvoir relier les représentations de $G$ à celles de  $[G]_2$, nous aurons  besoin de l'énoncé général suivant :

\begin{prop}\label{propinddeux}

Soit un groupe $G$ et $H$ un sous-groupe normal d'indice $2$. Soit $\rho$ est une représentation irréductible de $G$ (resp. $H$), 

\begin{itemize}
\item La restriction $\rho|_H$ (resp.  l'induite  $\ind^G_H \rho$) est irréductible ou se décompose en la somme directe de deux représentations irréductibles.
\item Si $\rho|_H$ (resp.  $\ind^G_H \rho$) se décompose i.e. $\rho|_H=\rho_1\oplus \rho_2$ (resp. $\cind^G_H \rho=\rho_1\oplus \rho_2$) alors $\rho=\ind_H^G\rho_i$ (resp. $\rho=\rho_i|_H$) pour $i=1,2$.
\item En particulier, toute représentation irréductible de $G$ (resp. de $H$) est un facteur direct de la restriction (resp. de l'induite) d'une représentation irréductible de $H$ (resp. de $G$)
\end{itemize}

\end{prop}

\begin{proof}
Étudions pour commencer une représentation irréductible $\rho$ de $G$. Si la restriction $\rho|_H$ n'est pas irréductible, donnons-nous une  sous-représentation stricte $\sigma\neq 0$. Il suffit de prouver l'identité $\rho|_H= \sigma \oplus g\sigma$ avec $g\in G\setminus H$ qui entraîne les points $1$ et $2$ pour $G$. On sait que  l'espace $\sigma + g\sigma$ admet une action du groupe $G$ tout entier en tant que quotient de $\ind_H^G\sigma$, il s'agit d'une sous-représentation de $\rho$ qui doit être égale à $\rho$ tout entier par irréductibilité. L'intersection $\sigma \cap g\sigma$ est aussi stable par $G$ et est strictement incluse dans $\rho$ par hypothèse. On a donc $\sigma \cap g\sigma=0$ et $\rho=\sigma \oplus g\sigma=\ind_H^G\sigma$. Pour tout sous-objet $\sigma'$ strict de $\sigma$, $\ind_H^G\sigma'$ est un sous-objet $\sigma$ strict de $\rho$ qui est donc trivial. Les représentations $\sigma$ et $g\sigma$ sont donc irréductibles. Pour le troisième point, on vient de voir que lorsque la restriction se décompose, on a $\rho=\ind_H^G\sigma$. Quand la restriction est irréductible, on peut toujours voir $\rho$ comme un  quotient et un sous objet de $\ind_H^G(\rho|_H)$.

Intéressons-nous maintenant à une représentation irréductible $\rho$ de $H$. Prenons un sous-$G$-objet  irréductible  $\sigma$ de $\ind_H^G\rho$. Trouver un tel sous-objet  est possible car les représentations $\rho$ et $\ind_H^G\rho$ sont de longueur finie sur $H$. L'étude de la série de Jordan-H\"older de $\ind_H^G\rho =\rho\oplus g\rho$ (avec $g\in G\setminus H$)  montre que $\sigma=\rho\oplus g\rho$ ou $\sigma|_H\cong\rho$; $\sigma|_H\cong g\rho$. Dans le premier cas, on a $\sigma=\ind_H^G\rho$ est irréductible. Dans les autres, on a $\ind_H^G\rho =\ind_H^G(\sigma|_H)$ et $\rho$ vaut $\sigma|_H$ ou $(\sigma|_H)^g$ ce  qui prouve les points $1.$, $2.$. Le troisième point découle du deuxième comme précédemment.
\end{proof}

La classification est la suivante:


\begin{theo}\label{theoclassgdeux}

Les représentations irréductibles admissibles de $[G]_2$ peuvent être décomposées en quatre familles disjointes :

\begin{enumerate}
\item Les caractères $\psi$.
\item La série spéciale $ \psi\otimes \stb|_{[G]_2}$.
\item La série principale  $\psi\otimes \cind_{B}^G \mu_1\otimes\mu_2|_{[G]_2}$ avec $\mu_1\neq \mu_2$.
\item Les représentations supersingulières obtenues comme facteurs directs  de $V(a,r,\vec{k},0)|_{[G]_2}$.
\end{enumerate}

\end{theo}

\begin{proof}
D'après \ref{propinddeux}, il suffit d'étudier la restriction des représentations apparaissant dans le théorème \ref{theoclassg}. Pour les familles $(a)$, $(b)$, $(c)$, la restriction à $[G]_2$ est déjà irréductible  car elle l'est lorsqu'on restreint à $\sln_2(K)$ d'après \ref{theoclasssl}. Par définition, les représentations supersingulières de $[G]_2$ sont exactement les facteurs directs des représentations supersingulières de $G$.
\end{proof}

Toujours lorsque  $K=\Q_p$, on peut décrire explicitement les représentations supersingulières.

\begin{theo}\label{theoclassgdeuxsupersing}

Si $K=\Q_p$, la restriction à $[G]_2$ des représentations  supersingulières $V(a,r,{k},0)$ de $G$ se décompose en somme directe de deux représentations irréductibles \[V(a,r,{k},0)|_{[G]_2}\cong \pi_{a^2,k,r}^0\oplus \pi_{a^2,k,r}^1.\] Ces représentations sont soumises aux isomorphismes $\pi_{a^2,k,r}^0\cong\pi_{a^2, p-1-k,r+k}^1$ qui sont les seuls entre ces représentations.

\end{theo}

\begin{proof}
Prenons une $G$-représentation supersingulière admissible $V(a,r,{k},0)$, étudions sa restriction à $[G]_2$. Cette dernière se décomposent en une somme directe de deux  $\sln_2(K)$-représentations $\pi_k^0$, $\pi_k^1$ (voir \ref{theoclasssl}). De par leur construction rappelée dans la preuve de \ref{theoclasssl}, on peut les voir naturellement comme des restrictions de représentations $\pi_{a^2,k,r}^i|_{\sln_2(K)}\cong \pi_k^i$ de $[G]_2$. L'énoncé \ref{propinddeux} permet alors de conclure.

Pour les isomorphismes entre les représentations supersingulières de $ [G]_2$, cela découle des isomorphismes analogues entre les représentations  $V(a,r,{k},0)$.
\end{proof}

\begin{rem}

Il est amusant d'observer que lorsqu'on restreint une $G$-représentation de la famille $(a)$, $(b)$, $(c)$, on obtient un représentation irréductible de $[G]_2$ de la famille analogue  qui sont toutes obtenues de cette manière. La contraposé montre que si une représentation irréductible de $G$ se décompose en une somme de deux représentations irréductibles de $[G]_2$, alors elle est supersingulière. On peut se demander par exemple si la réciproque est vraie i.e. si la restriction d'une représentation supersingulière de $G$ se décompose. Quand $K=\Q_p$, le résultat \ref{theoclassgdeuxsupersing} montre que c'est bien le cas  mais je ne saurais me prononcer dans le cas général.

\end{rem}

Nous allons maintenant procéder à la preuve du théorème \ref{theopoidsun}. Nous allons en fait montrer résultat un peu plus général et supposer $K$ totalement  ramifié i.e. $\F=\F_p$. En suivant les notations de la preuve de \ref{theoclassslsupsing}, on écrira \[V_0(a,r,k)=\delta_{a}\otimes{\det}^r\otimes \cind_{wG^\circ w^{-1}Z}^{[G]_2}(\sym^{k}\bar{\F}^2)^w,\]
\[V_1(a,r,k)=\delta_{a}\otimes{\det}^r\otimes \cind_{G^\circ Z}^{[G]_2}\sym^{k}\bar{\F}^2,\] de telle manière que $V(a,r,k,0)|_{[G]_2}=V_0(a^2,r,k)\oplus V_1(a^2,r,k)$ et, quand $K=\Q_p$, $\pi^i_{a,k,r}\cong V_i(a,r,k)/TV_{i+1}(a,r,k)$. Nous voulons établir:
\begin{theo}\label{theopoidsunram}

Si $K$ est totalement ramifié, $\bar{\Lf}$ un fibré en droite $[G]_2$-équivariant positif de poids $-1$ sur $\bar{\H}$ de quadruplet $(a,r_1,k_0,k_1)=(a,r_1,p-1-k_1,k_1)$ avec $r_0=r_1+k_1$, on a un isomorphisme $[G]_2$-équivariant \[\hhh^0(\bar{\H},\bar{\Lf})^{\vee}\cong V_1(a,-r_0,k_1)/TV_{0}(a,-r_0,k_1)\cong V_0(a,-r_1,k_0)/TV_{1}(a,-r_1,k_0).\]
En particulier,  $\hhh^0(\bar{\H},\bar{\Lf})^{\vee}\cong \pi^0_{a,k_0,-r_1} \cong \pi^1_{a,k_1,-r_0}$ si $K=\Q_p$
\end{theo}

 Il est à noter que lorsque $K$ n'est pas $\Q_p$, la représentation $\hhh^0(\bar{\H},\bar{\Lf})^{\vee}$ est loin d'être irréductible et même de longueur finie.

D'après \ref{footndel}, on peut supposer $a=1$. D'après le point $3$ de \ref{theofilt}, on sait que la représentation $\hhh^0(\bar{\H},\bar{\Lf})^{\vee}$ est \[\hhh^0(\bar{\H},\bar{\Lf})^\vee\cong  I(r_1,r_0)^\vee/\ker f_0^\vee+\ker f_1^\vee\] avec $f_i :\sigma^{(i)}_{k_i,r_i}\to I(r_1,r_0)$. De manière équivalente, la représentation recherchée est aussi le conoyau de la composée suivante :\[\varphi:\ker f_i^\vee\flinj I(r_1,r_0)^\vee\flsur I(r_1,r_0)^\vee/\ker f_{i+1}^\vee\cong \sigma^{(i+1),\vee}_{k_{i+1},r_{i+1}}.\] On veut comprendre cette flèche. Pour cela, on commence par les résultats intermédiaires classiques  suivants :

\begin{lem}

On a $(\sym^k \bar{\F}^2)^\vee\cong \sym^k \bar{\F}^2\otimes {\det}^{-k}$ d'où un isomorphisme \[\sigma^{(i),\vee}_{k_{i},r_{i}}\cong V_i(1,-r_{i+1},k_i)\et I(r_1,r_0)^\vee\cong\cind_{I Z}^{[G]_2}\mu_{-r_1,-r_0}\]

\end{lem}

\begin{proof}
Pour la première partie de l'énoncé, il suffit de voir que le dual d'une représentation irréductible de dimension finie est  irréductible de même dimension. Par la classification des représentations irréductibles de $\bar{G}$ données dans \ref{theobarg}, le dual $(\sym^k \bar{\F}^2)^\vee$ est $\sym^k \bar{\F}^2$ à un caractère près. Étudier l'action des matrices diagonales sur la base duale donne le résultat. Pour la deuxième partie, cela provient du faite que le dual d'une induite est l'induite compacte du dual.
\end{proof}

\begin{lem}\label{lemjordmod}

On a une suite exacte \[0\to\sym^k\bar{\F}^2\otimes{\det}^r\to\ind_{IZ}^{G^\circ Z}\mu_{r,r+k}  \to \sym^{p-1-k}\bar{\F}^2\otimes {\det}^{r+k}\to 0\] d'où \[0\to V_1(1,r,k)\to\cind_{IZ}^{[G]_2}\mu_{r,r+k}  \to V_1(1,k+r,p-1-k)\to 0\]\[0\to V_0(1,r,k)\to\cind_{IZ}^{[G]_2}\mu_{r+k,r}  \to V_0(1,k+r,p-1-k)\to 0\] \[0\to\cind_{G^\circ Z}^{G}\sym^k\bar{\F}^2\otimes{\det}^r\to\cind_{IZ\sqcup wIZ}^{G}\mu_{r,r+k} \to\cind_{G^\circ Z}^{G} \sym^{p-r-1}\bar{\F}^2\otimes{\det}^{k+r}\to 0.\]
Notons que les morphismes de la deuxième suite exacte sont obtenus en restreignant les morphismes de la troisième.
\end{lem}

\begin{proof}
La première  suite exacte de l'énoncé est bien connue \cite[Proposition 1.1.]{diam} et le reste découle de  l'exactitude de l'induite compacte.
\end{proof}

\begin{rem}

Ces deux résultats utilisent de manière cruciale le faite que $K$ est totalement ramifiée et montrent que le résultat est faux pour un corps général. Une description explicite de $\hhh^0(\bar{\H},\bar{\Lf})^{\vee}$ est peut-être encore  possible mais demandera plus de travail et de calcul combinatoire.

\end{rem}


D'après les deux résultats précédents, on a $\ker f_i^\vee\cong V_i(1,-r_i,p-1-k_i)=V_i(1,-r_i,k_{i+1})$ et on s'est ramené à étudier le noyau de la flèche obtenue par composition \[\tilde{\varphi}:V_i(1,-r_i,k_{i+1})\to \cind_{IZ}^{[G]_2}\mu_{1,-r_1,r_0}\to V_{i+1}(1,-r_i,k_{i+1}).\] La dernière  suite exacte dans le lemme précédent montre que cette dernière est la restriction d'un élément de \[\Hf(\sym^{k_{i+1}}\bar{\F}^2\otimes{\det}^{-r_i})=\en_G(\cind_{G^\circ Z}^{G}\sym^r\bar{\F}^2\otimes{\det}^{-r_i})=\bar{\F}[T].\] D'après \ref{theoalghecke}.3., il lui correspond un polynôme $P(T)$ en l'opérateur $T$ définit dans \textit{loc.cit}. Par construction, $\tilde{\varphi}$ échange les deux sous-représentations $V_{i}(1,-r_i,k_{i+1})$, $V_{i+1}(1,-r_i,k_{i+1})$ de $[G]_2$ et $P\in T\bar{\F}[T^2]$ d'après \ref{theoalghecke}.1. Le lemme \ref{lemsupp} qui va suivre permet d'appliquer le résultat \ref{theoalghecke}.2. et prouver que $P$ est non-nul de degré $1$. Cela montre que $\tilde{\varphi}$ est la restriction de $\lambda T\in \en_G(\cind_{G^\circ Z}^{G}\sym^r\bar{\F}^2\otimes\chi)$ avec $\lambda \neq 0$ de ce qui termine la preuve de \ref{theopoidsunram}. 

Avant d'énoncer le lemme, nous aimerions rappeler que l'espace sous-jacent des représentations  $V_{j}(1,-r_i,k_{i+1})$ et $\cind_{IZ}^{[G]_2}\mu_{-r_1,r_0}$ est relié aux sommets et aux arêtes de l'arbre de Bruhat-Tits (voir \eqref{eq:ssjac}) et que tout élément de ces espaces admet alors un support (voir \eqref{eq:supp}) qui est un sous-ensemble fini de $\BC\TC_0^{(j)}$ ou de $ \BC\TC_1$. 

\begin{lem}\label{lemsupp}

Prenons  un élément non-nul $v_s \in V_s\subset\oplus_{s\in \BC\TC_0^{(i)}}V_s$  supporté en un sommet $s\in \BC\TC_0^{(i)}$, son image $\tilde{\varphi}(v)$ est non-nulle supportée en  les voisins de $s$ dans $\BC\TC_0^{(i+1)}$ i.e. les sommets $s'$ pour lesquels $\{s,s'\}$ forme une arête dans $\BC\TC_1$.

\end{lem}

\begin{proof}
Par définition, $\tilde{\varphi}$ est la composition de deux fonctions \[\tilde{\varphi}_1:V_{i}(1,-r_i,k_{i+1})\to \cind_{IZ}^{[G]_2}\mu_{-r_1,-r_0}\ \ \et\ \ \tilde{\varphi}_2: \cind_{IZ}^{[G]_2}\mu_{-r_1,-r_0}\to V_{i+1}(1,-r_i,k_{i+1}).\] Étudions d'abord $\tilde{\varphi}_1(v_s)$ avec $v_s$ comme dans l'énoncé. D'après \ref{leminjsur}, on peut voir $\tilde{\varphi}_1(v_s)$ comme « l'évaluation » de $v_s$ (vu comme un polynôme homogène de degré $k_1$ sur $\P_s$) en les différents points rationnels $y\in \P_s(\F_p)$ qui sont en bijection avec les arêtes $a$ de $\BC\TC_1=[G]_2/IZ$ contenant le sommet $s$. En particulier, $\tilde{\varphi}_1(v_s)$ est non-nul par injectivité de $\tilde{\varphi}_1$ et supportée en les arrêtes adjacentes à $s$. Soit $v_{a_0}$ un élément non nul de $\oplus_{s\in \BC\TC_1}V_a$  supporté en une unique arête $a_0$, un raisonnement dual permet de montrer que $\tilde{\varphi}_2(v_{a_0})$  s'enverra sur un élément non nul de $\oplus_{s\in \BC\TC_0^{(i+1)}}V_s$ supporté en le sommet de $a_0$ dans $ \BC\TC_0^{(i+1)}$. Par composition, on voit que $\tilde{\varphi}(v_s)$ est non-nulle et supportée en les voisins de $s$ dans $\BC\TC_0^{(i+1)}$. 
\end{proof} 

\begin{rem}

Une autre preuve de \ref{theopoidsunram} qui nous a été communiquée par Zhixiang Wu (dont les commentaires ont largement inspiré ce résultat) pourrait consister à prendre un générateur de la représentation $V_{i}(1,-r_i,k_{i+1})i$ et calculer explicitement son image en suivant les deux morphismes intermédiaires $\tilde{\varphi}_1$ et $\tilde{\varphi}_2$. L'intérêt de cette preuve plus calculatoire est qu'elle permet de montrer que l'endomorphisme $\tilde{\varphi}$ est bien $T$ et non un multiple.

\end{rem}
%



\bibliographystyle{alpha}
\bibliography{cohcohhhat_v1}

\begin{thebibliography}{FvdP04}

\bibitem[AB15]{anand}
U.~K. Anandavardhanan and Gautam~H. Borisagar.
\newblock Iwahori-{H}ecke model for supersingular representations of {${\rm
  GL}_2(\Bbb{Q}_p)$}.
\newblock {\em J. Algebra}, 423:1--27, 2015.

\bibitem[Abd14]{abd1}
Ramla Abdellatif.
\newblock Classification des repr\'{e}sentations modulo {$p$} de {${\rm
  SL}(2,F)$}.
\newblock {\em Bull. Soc. Math. France}, 142(3):537--589, 2014.

\bibitem[BC91]{boca}
J.-F. Boutot and H.~Carayol.
\newblock Uniformisation {$p$}-adique des courbes de {S}himura: les
  th\'{e}or\`emes de \v{C}erednik et de {D}rinfel'd.
\newblock {\em Ast\'{e}risque}, (196-197):7, 45--158 (1992), 1991.
\newblock Courbes modulaires et courbes de Shimura (Orsay, 1987/1988).

\bibitem[Ber94]{ber3}
Vladimir~G. Berkovich.
\newblock Vanishing cycles for formal schemes.
\newblock {\em Invent. Math.}, 115(3):539--571, 1994.

\bibitem[Ber99]{ber7}
Vladimir~G. Berkovich.
\newblock Smooth {$p$}-adic analytic spaces are locally contractible.
\newblock {\em Invent. Math.}, 137(1):1--84, 1999.

\bibitem[BL94]{BL2}
L.~Barthel and R.~Livn\'{e}.
\newblock Irreducible modular representations of {${\rm GL}_2$} of a local
  field.
\newblock {\em Duke Math. J.}, 75(2):261--292, 1994.

\bibitem[BL95]{BL1}
L.~Barthel and R.~Livn\'{e}.
\newblock Modular representations of {${\rm GL}_2$} of a local field: the
  ordinary, unramified case.
\newblock {\em J. Number Theory}, 55(1):1--27, 1995.

\bibitem[BM10]{breuilmez2}
Christophe Breuil and Ariane M\'{e}zard.
\newblock Repr\'{e}sentations semi-stables de {${\rm GL}_2(\Bbb Q_p)$},
  demi-plan {$p$}-adique et r\'{e}duction modulo {$p$}.
\newblock {\em Ast\'{e}risque}, (331):117--178, 2010.

\bibitem[Bre03a]{breuil1}
Christophe Breuil.
\newblock Sur quelques repr\'{e}sentations modulaires et {$p$}-adiques de
  {${\rm GL}_2(\bold Q_p)$}. {I}.
\newblock {\em Compositio Math.}, 138(2):165--188, 2003.

\bibitem[Bre03b]{breuil2}
Christophe Breuil.
\newblock Sur quelques repr\'{e}sentations modulaires et {$p$}-adiques de
  {${\rm GL}_2(\bold Q_p)$}. {II}.
\newblock {\em J. Inst. Math. Jussieu}, 2(1):23--58, 2003.

\bibitem[Che13]{cheng}
Chuangxun Cheng.
\newblock Mod {$p$} representations of {$SL_2(\Bbb{Q}_p)$}.
\newblock {\em J. Number Theory}, 133(4):1312--1330, 2013.

\bibitem[Dia07]{diam}
Fred Diamond.
\newblock A correspondence between representations of local {G}alois groups and
  {L}ie-type groups.
\newblock In {\em {$L$}-functions and {G}alois representations}, volume 320 of
  {\em London Math. Soc. Lecture Note Ser.}, pages 187--206. Cambridge Univ.
  Press, Cambridge, 2007.

\bibitem[DLB17]{brasdospi}
Gabriel Dospinescu and Arthur-C\'{e}sar Le~Bras.
\newblock Rev\^{e}tements du demi-plan de {D}rinfeld et correspondance de
  {L}anglands {$p$}-adique.
\newblock {\em Ann. of Math. (2)}, 186(2):321--411, 2017.

\bibitem[Dri76]{dr2}
V.~G. Drinfel'd.
\newblock Coverings of {$p$}-adic symmetric domains.
\newblock {\em Funkcional. Anal. i Prilo\v{z}en.}, 10(2):29--40, 1976.

\bibitem[FvdP04]{frvdp}
Jean Fresnel and Marius van~der Put.
\newblock {\em Rigid analytic geometry and its applications}, volume 218 of
  {\em Progress in Mathematics}.
\newblock Birkh\"{a}user Boston, Inc., Boston, MA, 2004.

\bibitem[Geh22]{gehr}
Lennart Gehrmann.
\newblock Invertible analytic functions on {D}rinfeld symmetric spaces and
  universal extensions of {S}teinberg representations.
\newblock {\em Res. Math. Sci.}, 9(1):Paper No. 16, 21, 2022.

\bibitem[Gek20]{gek}
Ernst-Ulrich Gekeler.
\newblock Invertible functions on nonarchimedean symmetric spaces.
\newblock {\em Algebra \& Number Theory}, 14(9):2481–2504, Oct 2020.

\bibitem[GI63]{iwgo}
O.~Goldman and N.~Iwahori.
\newblock The space of {$\mathfrak{p}$}-adic norms.
\newblock {\em Acta Math.}, 109:137--177, 1963.

\bibitem[GK04a]{GK3}
Elmar Grosse-Kl\"{o}nne.
\newblock De {R}ham cohomology of rigid spaces.
\newblock {\em Math. Z.}, 247(2):223--240, 2004.

\bibitem[GK04b]{GK6}
Elmar Grosse-Kl\"{o}nne.
\newblock Integral structures in automorphic line bundles on the {$p$}-adic
  upper half plane.
\newblock {\em Math. Ann.}, 329(3):463--493, 2004.

\bibitem[GK05]{GK7}
Elmar Grosse-Kl\"{o}nne.
\newblock Integral structures in the {$p$}-adic holomorphic discrete series.
\newblock {\em Represent. Theory}, 9:354--384, 2005.

\bibitem[Jen72]{jen}
C.~U. Jensen.
\newblock {\em Les foncteurs d\'{e}riv\'{e}s de {$\varprojlim$} et leurs
  applications en th\'{e}orie des modules}.
\newblock Lecture Notes in Mathematics, Vol. 254. Springer-Verlag, Berlin-New
  York, 1972.

\bibitem[Juna]{J3}
Damien Junger.
\newblock \'equations pour le premier revêtement de l'espace symétrique de
  drinfeld.
\newblock https://arxiv.org/abs/2202.01018.

\bibitem[Junb]{J2}
Damien Junger.
\newblock Un autre calcul des fonctions inversibles sur l'espace symétrique de
  drinfeld.
\newblock https://arxiv.org/abs/2111.10274.

\bibitem[Jun23]{J1}
Damien Junger.
\newblock Cohomologie analytique des arrangements d'hyperplans.
\newblock {\em Algebra Number Theory}, 17(1):1--43, 2023.

\bibitem[Mes72]{messing}
William Messing.
\newblock {\em The crystals associated to {B}arsotti-{T}ate groups: with
  applications to abelian schemes}.
\newblock Lecture Notes in Mathematics, Vol. 264. Springer-Verlag, Berlin-New
  York, 1972.

\bibitem[MM81]{mormur1}
Yasuo Morita and Atsushi Murase.
\newblock Analytic representations of {${\rm SL}_{2}$} over a {${p}$}-adic
  number field.
\newblock {\em J. Fac. Sci. Univ. Tokyo Sect. IA Math.}, 28(3):891--905 (1982),
  1981.

\bibitem[Mor84]{mormur2}
Yasuo Morita.
\newblock Analytic representations of {${\rm SL}_2$} over a {${p}$}-adic number
  field. {II}.
\newblock In {\em Automorphic forms of several variables ({K}atata, 1983)},
  volume~46 of {\em Progr. Math.}, pages 282--297. Birkh\"{a}user Boston,
  Boston, MA, 1984.

\bibitem[Mor85]{mormur3}
Yasuo Morita.
\newblock Analytic representations of {${\rm SL}_2$} over a {${p}$}-adic number
  field. {III}.
\newblock In {\em Automorphic forms and number theory ({S}endai, 1983)},
  volume~7 of {\em Adv. Stud. Pure Math.}, pages 185--222. North-Holland,
  Amsterdam, 1985.

\bibitem[Orl08]{orl2}
Sascha Orlik.
\newblock Equivariant vector bundles on {D}rinfeld's upper half space.
\newblock {\em Invent. Math.}, 172(3):585--656, 2008.

\bibitem[OS10]{orlstrau}
Sascha Orlik and Matthias Strauch.
\newblock On the irreducibility of locally analytic principal series
  representations.
\newblock {\em Represent. Theory}, 14:713--746, 2010.

\bibitem[Pan17]{pan}
Lue Pan.
\newblock First covering of the {D}rinfel'd upper half-plane and {B}anach
  representations of {${\rm GL}_2(\mathbb{Q}_p)$}.
\newblock {\em Algebra Number Theory}, 11(2):405--503, 2017.

\bibitem[PSS14]{psm}
Deepam Patel, Tobias Schmidt, and Matthias Strauch.
\newblock Locally analytic representations and sheaves on the {B}ruhat-{T}its
  building.
\newblock {\em Algebra Number Theory}, 8(6):1365--1445, 2014.

\bibitem[Ray74]{Rayn}
Michel Raynaud.
\newblock Sch\'{e}mas en groupes de type {$(p,\dots, p)$}.
\newblock {\em Bull. Soc. Math. France}, 102:241--280, 1974.

\bibitem[SS91]{scst}
P.~Schneider and U.~Stuhler.
\newblock The cohomology of {$p$}-adic symmetric spaces.
\newblock {\em Invent. Math.}, 105(1):47--122, 1991.

\bibitem[Tei89]{teit1}
Jeremy Teitelbaum.
\newblock On {D}rinfel' d's universal formal group over the {$p$}-adic upper
  half plane.
\newblock {\em Math. Ann.}, 284(4):647--674, 1989.

\bibitem[Tei90]{teit2}
Jeremy Teitelbaum.
\newblock Geometry of an \'{e}tale covering of the {$p$}-adic upper half plane.
\newblock {\em Ann. Inst. Fourier (Grenoble)}, 40(1):68--78, 1990.

\bibitem[Tei93]{teit4}
Jeremy~T. Teitelbaum.
\newblock Modular representations of {${\rm PGL}_2$} and automorphic forms for
  {S}himura curves.
\newblock {\em Invent. Math.}, 113(3):561--580, 1993.

\bibitem[VdP82]{vdp}
M.~Van~der Put.
\newblock Cohomology on affinoid spaces.
\newblock {\em Compositio Math.}, 45(2):165--198, 1982.

\bibitem[Wan14]{wa}
Haoran Wang.
\newblock L'espace sym\'{e}trique de {D}rinfeld et correspondance de
  {L}anglands locale {I}.
\newblock {\em Math. Z.}, 278(3-4):829--857, 2014.

\bibitem[Zin84]{zink}
Thomas Zink.
\newblock {\em Cartiertheorie kommutativer formaler {G}ruppen}, volume~68 of
  {\em Teubner-Texte zur Mathematik [Teubner Texts in Mathematics]}.
\newblock BSB B. G. Teubner Verlagsgesellschaft, Leipzig, 1984.
\newblock With English, French and Russian summaries.

\end{thebibliography}
\end{document}